\documentclass[11pt,a4paper]{article}
\usepackage[T1]{fontenc}\usepackage{lmodern}
\usepackage{amsmath,amssymb,amsthm,mathtools,mathrsfs}
\usepackage[margin=2cm]{geometry}\usepackage{authblk,longtable}
\usepackage[sf,bf]{titlesec}
\usepackage{microtype,enumitem,booktabs,needspace}\usepackage{array}
\usepackage{xcolor}
\usepackage[colorlinks=true,linkcolor=blue!50!black,urlcolor=blue!50!black,citecolor=blue!50!black]{hyperref}
\newtheorem{proposition}{Proposition}[section]
\newtheorem{backgroundproposition}[proposition]{Proposition}

\newtheorem{theorem}[proposition]{Theorem}

\theoremstyle{definition}
\newtheorem{definition}[proposition]{Definition}

\newtheorem*{question}{Question}
\newcommand{\F}{\mathbb F_2}
\newcommand{\Q}{\mathbb Q}
\newcommand{\Z}{\mathbb Z}
\newcommand{\Gal}{\operatorname{Gal}}
\newcommand{\Spec}{\operatorname{Spec}}
\newcommand{\Pic}{\operatorname{Pic}}
\newcommand{\Br}{\operatorname{Br}}

\newcommand{\im}{\operatorname{im}}

\newcommand{\et}{\mathrm{\acute et}}

\newcommand{\sh}{\mathrm{sh}}

\DeclareMathOperator{\Frac}{Frac}

\newcommand{\st}[1]{\href{https://stacks.math.columbia.edu/tag/#1}{\textup{[Stacks, #1]}}}
\newcommand{\pigeom}{\pi_1^{\mathrm{geom}}}
\newcommand{\piet}{\pi_1^{\et}}
\title{\sffamily\bfseries\Large Golod--Shafarevich for arithmetic surfaces}
\author[1]{Timo Keller\thanks{Timo Keller, IBM Deutschland Research \& Development, IBM Campus 1, 71139 Ehningen, Germany, \texttt{math@kellertimo.de}.}}
\author[2]{Carlo Pagano\thanks{Department of Mathematics and Statistics, Concordia University, Quebec H3G 1M8, Canada, \texttt{carlein90@gmail.com}; Google DeepMind, \texttt{carlopagano@google.com}.}}
\affil[1]{IBM Deutschland Research \& Development}
\affil[2]{Concordia University and Google DeepMind}

\date{September 27, 2026}
\begin{document}
\maketitle
\begin{abstract}
We construct Golod--Shafarevich towers of arithmetic surfaces over $\Spec\Z$. In particular, we obtain an arithmetic surface with a section and infinite geometric \'etale fundamental group, answering a question raised by Jean-Beno\^it Bost and Fran\c cois Charles about the existence of such surfaces. Taking generic fibres gives curves over $\Q$ with infinite geometric \'etale towers in which a rational point splits completely. This answers questions posed by Yasutaka Ihara and by Gerhard Frey, Ernst Kani and Helmut V\"olklein. Writing $C_n=(S_n)_{\Q}$ for the generic fibres of the tower, we also obtain $\log\mathfrak N(C_n)=O(g(C_n))$ and $g(C_n)\to\infty$, answering a question asked by Akshay Venkatesh at PCMI in 2022.

The authors worked together on this problem since $2022$ and developed a general strategy. They were able to bring it to fruit only after a collaboration with \emph{GPT5.6 Sol}, \emph{Fable} and \emph{Aletheia} a Gemini-powered internal agent developed at Google DeepMind.
\end{abstract}

\section{Introduction}\label{sec:introduction}
At PCMI in 2022, Akshay Venkatesh asked how small the logarithmic conductor
of a curve over $\Q$ can be compared with its genus. For a smooth
projective geometrically connected curve $C/\Q$, we write
$\mathfrak N(C)$ for the conductor of its Jacobian.
Standard examples, such as
the modular curves $X_0(p)$ for prime $p$, give logarithmic conductor of
order $g\log g$: writing $g_p=g(X_0(p))$, we have
$g_p\sim p/12$ and $\mathfrak N(X_0(p))=p^{g_p}$; see \cite[Section~0]{Emerton} for their reduction.
\begin{question}[Venkatesh, 2022]
Can one instead find a sequence of smooth projective geometrically
connected curves $C_n/\Q$ with $g(C_n)\to\infty$ and
\begin{equation}\label{eq:intro-statement-1}
 \log\mathfrak N(C_n)=O(g(C_n))?
\end{equation}
\end{question}
This is the expected optimal order of magnitude. For example, assuming that the completed
$L$-function of the Jacobian of $C$ is entire of order one and satisfies its
expected functional equation, Mestre proves
$\log\mathfrak N(C)>\log(10) \cdot g(C)$ \cite[p.~229]{Mestre}.

The analogy with number fields suggests an approach to construct families achieving growth which is linear in the genus. Namely Golod--Shafarevich's
theory gives infinite unramified towers of number fields with bounded
root discriminant: their logarithmic discriminants grow only linearly
with their degrees; see \cite[Section~10]{Ershov}. One is therefore led to hunt
for an arithmetic surface $S$ with an infinite tower of finite \'etale
covers $S_n\to S$ whose generic curves $C_n=(S_n)_{\Q}$ are
geometrically connected. We seek $g(C_n)\to\infty$ and a finite set of primes $\Sigma$ such that
every $C_n$ has good reduction at $p\notin\Sigma$. For each $p\in\Sigma$,
we also require a finite extension $E_p/\Q_p$, independent of $n$, over which every $C_n$ has semistable reduction.

This brings us to a question raised by Bost and Charles
\cite[Section~9.3.4, after Corollary~9.3.8, p.~160]{BostCharles}.
After proving finiteness results for fundamental groups of arithmetic surfaces, they point out that no example with the following properties was known.

\begin{question}[Bost--Charles, 2022]
Does there exist a regular integral quasi-projective arithmetic surface
$S\to\Spec\Z$ with a section and infinite \'etale fundamental group?
\end{question}

Bost and Charles allow nonproper surfaces; our examples will be projective. Several finiteness results are known:
Bost and Charles treat models of elliptic curves and certain integral modular schemes \cite[Corollaries~9.3.4 and~9.3.8]{BostCharles}. Calegari, Dimitrov and Tang \cite[Theorem~7.4.10]{CDT} prove that a
natural smooth integral model $\mathcal X(N)$ of the modular curve $X(N)$ over $\mathcal O_K$, for $K\supseteq\Q(\mu_N)$, has a structure map that induces
\begin{equation}\label{eq:intro-statement-2}
 \piet(\mathcal X(N))\simeq\piet(\Spec\mathcal O_K).
\end{equation}
Their model is the part of the smooth locus of the minimal regular
model containing the cusp section at infinity, which they call the
connected N\'eron model. In particular, the image of
$\piet(\mathcal X(N)_{\overline K})$ in $\piet(\mathcal X(N))$
is trivial. For our surfaces $S$, the image of
$\piet(S_{\overline\Q})$ in $\piet(S)$ is infinite.

On the generic fibers, that is for curves over $\Q$, we have a related question. Following Ihara
\cite{Ihara}, Frey, Kani and V\"olklein \cite{FKV} introduced the $k$-rational geometric fundamental group, which classifies \'etale covers
in which a chosen rational point splits completely. For a finite cover $D\to C$ of degree $d$, this means that the fibre above $P\in C(k)$ is a disjoint union of $d$ copies of $\Spec k$. Diem and Frey ask explicitly whether such covers can have unbounded degree in characteristic zero
\cite[Section~1.5, Problem~1.9 and footnote~1]{DiemFrey}.

\begin{question}[Diem--Frey, 2003]
Does there exist a finitely generated field $k$ of characteristic zero,
a smooth projective geometrically connected curve $C/k$, and a point
$P\in C(k)$ admitting geometrically connected finite \'etale Galois
covers of unbounded degree in which $P$ splits completely?
\end{question}

The purpose of this work is to develop Golod--Shafarevich towers for
arithmetic surfaces and answer all three questions. We construct
a projective surface $S$ over $\Spec\Z$ with a section $s$ and
finite \'etale covers $S_n\to S$ of unbounded degree whose generic
curves $C_n/\Q$ are geometrically connected. The point
$s_{\Q}\in S_{\Q}(\Q)$ splits completely in every $C_n$, answering
the curve question over $\Q$. We also prove a linear bound for
$\log\mathfrak N(C_n)$ in terms of $g(C_n)$, answering Venkatesh's
question.

Throughout this paper an arithmetic surface is an integral proper flat
scheme over $\Spec\Z$ with one-dimensional fibres. For a surface with
smooth geometrically connected generic fibre, write
\begin{equation}\label{eq:intro-statement-3}
 \pigeom(S)=\im\bigl(\piet(S_{\overline\Q})\longrightarrow\piet(S)\bigr).
\end{equation}
A section splits completely in a degree-$d$
cover if its inverse image is a disjoint union of $d$ copies of
$\Spec\Z$. Further conventions are given in
Section~\ref{sec:dictionary}.

\begin{theorem}\label{thm:main-surface}
There exist a regular projective arithmetic surface $S$ over $\Spec\Z$,
a section $s:\Spec\Z\to S$, and a tower
\begin{equation}\label{eq:intro-statement-4}
\cdots\longrightarrow S_{n+1}\longrightarrow S_n\longrightarrow\cdots
\longrightarrow S_0=S
\end{equation}
with the following properties.
\begin{enumerate}[label=(\roman*)]
\item Each $S_n\to S$ is a connected finite \'etale Galois cover of
degree $2^n$, and each transition map $S_{n+1}\to S_n$ has degree two.
\item Every $S_n$ is a regular projective arithmetic surface with
smooth geometrically connected generic fibre.
\item The section $s$ splits completely at every level:
\begin{equation}\label{eq:intro-statement-5}
S_n\times_{S,s}\Spec\Z\simeq\coprod_{j=1}^{2^n}\Spec\Z.
\end{equation}
\end{enumerate}
In particular, $\pigeom(S)=\piet(S)$ is infinite.
\end{theorem}

For the tower in Theorem~\ref{thm:main-surface}, the generic fibers
$B_n=(S_n)_{\Q}$ form a tower over $B=S_{\Q}$ in which the rational
point $s_{\Q}$ splits completely. If $f:C\to B$ is a ramified double
cover, the pullbacks $C\times_B B_n\to C$ are again geometrically
connected finite \'etale Galois covers. By choosing $f$ so that
$s_{\Q}$ splits completely and $C$ has the required genus, we obtain
examples in every sufficiently large genus.

\begin{theorem}\label{thm:main-curves}
For every sufficiently large integer $g$, there are a smooth projective
geometrically connected curve $C/\Q$ of genus $g$, a point $P\in C(\Q)$,
and a tower of smooth projective geometrically connected curves
\begin{equation}\label{eq:intro-statement-6}
 \cdots\longrightarrow D_{n+1}\longrightarrow D_n
 \longrightarrow\cdots\longrightarrow D_0=C
\end{equation}
such that $D_n\to C$ is finite \'etale Galois of degree $2^n$ and $P$
splits completely in every $D_n$. In particular,
\begin{equation}\label{eq:intro-statement-7}
 g(D_n)-1=2^n(g-1),\qquad \#D_n(\Q)\ge2^n,
\end{equation}
and the $\Q$-rational geometric fundamental group of $(C,P)$ is infinite.
\end{theorem}

Here the $\Q$-rational geometric fundamental group is
\begin{equation}\label{eq:intro-statement-8}
 \pi_1^{\Q\text{-}\mathrm{geom}}(C,P)
 =\piet(C,\bar P)/
 \overline{\langle\!\langle s_P(\Gal(\overline\Q/\Q))
                       \rangle\!\rangle},
\end{equation}
Here $\bar P:\Spec\overline\Q\to C$ is the geometric point induced
by $P$, and $s_P$ is the section of
$\piet(C,\bar P)\to\Gal(\overline\Q/\Q)$ induced by $P$.
The denominator is the closed normal subgroup generated by
$s_P(\Gal(\overline\Q/\Q))$.
Our starting examples
are explicit: for $m\ge32768$ and $b=2^m-1$, we use the smooth
projective curve $C_{-1}$ of genus $b$ with affine equation
\begin{equation}\label{eq:intro-statement-9}
 y^2=\prod_{a=0}^{2b+1}(x-a)
\end{equation}
and the point $P_+$ at infinity where $y/x^{b+1}=1$. The choice
$b+1=2^m$ is used to extend the initial \'etale cover
$C_0\to C_{-1}$ to a finite map of arithmetic models
$S_0\to S_{-1}$ for which $S_0$ is regular, including over $2$.
For each sufficiently large $g$, we construct a ramified double cover
$f:C\to C_{-1}$ with $g(C)=g$ in which $P_+$ splits completely,
and pull back the \'etale tower over $C_{-1}$ along $f$.

Finally, the generic curves $C_n=(S_n)_{\Q}$ in
Theorem~\ref{thm:main-surface} have logarithmic conductors growing at
most linearly with $g(C_n)$. We recall the definition of the conductor
and the estimate used here in
Subsection~\ref{sec:conductor-background}.

\begin{theorem}\label{thm:main-conductor}
There exist a constant $A>0$ and a sequence of smooth projective
geometrically connected curves $C_n/\Q$ such that
\begin{equation}\label{eq:intro-statement-10}
 g(C_n)\longrightarrow\infty,\qquad
 \log\mathfrak N(C_n)\le A g(C_n)\quad(n\ge0).
\end{equation}
The curves can be chosen in a tower with finite \'etale Galois maps
$C_n\to C_0$ of degree $2^n$, in which a point of $C_0(\Q)$ splits
completely at every level.
\end{theorem}

%\paragraph{Independent work.}\textbf{TODO: Insert the precise statement of the independent results ofFran\c cois Charles's student, compare them withTheorems~\ref{thm:main-surface}--\ref{thm:main-conductor}, and add thereference to the simultaneously released paper.}

\subsection{Method of proof}\label{sec:intro-method}
We begin with a hyperelliptic curve $C_{-1}/\Q$ of genus $g$ and an
explicit \'etale cover
\begin{equation}\label{eq:intro-statement-11}
 C_0\longrightarrow C_{-1},\qquad
 \Gal(C_0/C_{-1})\simeq\F^g.
\end{equation}
A key property of $C_{-1}$ is to have bad reduction at no more than $O\left(\frac{g}{\text{log}(g)}\right)$ primes. We construct arithmetic models $S_0$ and $S_{-1}$ and a finite map
$q:S_0\to S_{-1}$ whose generic fibre is $C_0\to C_{-1}$;
$S_0$ is regular and has a section $s_0:\Spec\Z\to S_0$.
We construct further covers $S_n\to S_0$ that are finite \'etale,
requiring the composites $S_n\to S_{-1}$ to be Galois.
The morphism $q$ may ramify in the fibres of $S_0\to\Spec\Z$ over
finitely many primes, while $S_{n+1}\to S_n$ is \'etale for every $n\ge0$. The composites
$C_n\to C_{-1}$, where $C_n=(S_n)_{\Q}$, are all \'etale.

Suppose we have constructed $S_n$, and put
$G_n=\Gal(S_n/S_{-1})$. A class $\theta\in H^2(G_n,\F)$ corresponds
to a central extension
\begin{equation}\label{eq:intro-statement-12}
 1\longrightarrow\F\longrightarrow E_\theta
   \longrightarrow G_n\longrightarrow1.
\end{equation}
A nonzero class $\theta$ is realizable if there is a connected finite
\'etale double cover $S_{n+1}\to S_n$ whose composite
$S_{n+1}\to S_{-1}$ is Galois with group $E_\theta$, and for
which the restriction homomorphism
$\Gal(S_{n+1}/S_{-1})\to\Gal(S_n/S_{-1})$ identifies with
$E_\theta\to G_n$ in \eqref{eq:intro-statement-12}.
Together with the zero class, the realizable classes form a subspace
of $H^2(G_n,\F)$. We bound its codimension by successive approximations: we first proceed rationally, then we measure the obstruction to make the maps \'etale integrally. 

Write $\mathscr E_n$ for the subspace of realizable classes just
defined. In Section~\ref{sec:estimating-obstruction} we prove
\begin{equation}\label{eq:intro-obstruction-bound}
 \dim_{\F}\bigl(H^2(G_n,\F)/\mathscr E_n\bigr)
 \le\frac{5000g^2}{\log g}\qquad(n\ge0).
\end{equation}
The bound in \eqref{eq:intro-obstruction-bound} depends on $g$, the genus of $C_{-1}$, and is independent of $n$.

Starting with $G_0\simeq\F^g$, we choose nonsplit central extensions
$G_{n+1}\to G_n$ with kernel $\F$. Lifts to $G_{n+1}$ of a minimal
generating set of $G_n$ generate $G_{n+1}$, so each $G_n$ has minimal
number of generators $g$. Golod--Shafarevich gives
\begin{equation}\label{eq:intro-statement-13}
 \dim_{\F}H^2(G_n,\F)>\frac{g^2}{4}.
\end{equation}
For sufficiently large $g$,
$5000g^2/\log g<g^2/4<\dim_{\F}H^2(G_n,\F)$.
There is therefore a nonzero realizable class at every stage, and we can continue the tower indefinitely. For each $n$, the inverse image of $s_0$ under $S_n\to S_0$ is a disjoint union of $2^n$ copies of
$\Spec\Z$, and $C_n$ is geometrically connected. At each bad prime $p$, there is a finite extension $E_p/\Q_p$, independent of $n$, over which every $C_n$ has semistable reduction. For every prime $p$ of bad reduction, Proposition~\ref{prop:conductor-local-bound} gives a
constant $B_p$, depending only on $E_p/\Q_p$, such that
$a_p(C_n)\le B_p g(C_n)$ for every $n$. Here $a_p(C_n)$ is the
exponent of $p$ in $\mathfrak N(C_n)$. Summing
$a_p(C_n)\log p$ over the finitely many primes of bad reduction proves $\log\mathfrak N(C_n)=O(g(C_n))$.
\subsection{Related work}
Pablo Chenal \cite{Chenal26} has independently and simultaneously proved the existence of a proper regular arithmetic surface over \(\operatorname{Spec}\mathbb Z\) with a section and an infinite pro-\(2\) quotient of its étale fundamental group. His starting point is a curve closely related to our \(C_{-1}\): both are hyperelliptic, with rational Weierstrass points and \(O(g/\log g)\) primes of bad reduction, where \(g\) denotes the genus. Concretely, the main change in the respective Weierstrass equations is a factor \(8\) used in his analysis of the model at \(2\).

In his argument too, the desired curve is a finite étale cover of his starting curve. This cover plays the role of our \(C_0\) and serves as the base of the infinite tower. It is not determined explicitly and is obtained by an existence argument. This allows him to avoid much of the detailed analysis of integral models and successive central extensions carried out here. As a consequence, the resulting proof is substantially shorter and more elegant. It works directly with the pro-\(2\) geometric fundamental group of his starting curve and the action of \(\operatorname{Gal}(\overline{\mathbb Q}/\mathbb Q)\) on it induced by a rational point. His argument combines specialization at \(2\) with Vinberg’s refinement of the Golod–Shafarevich theorem.

Pablo Chenal and the present authors are working towards a unification of the two approaches, and plan to merge the two papers into a single joint work.

\subsection{Layout of the paper}\label{sec:intro-layout}
Section~\ref{sec:dictionary} recalls the geometric and
cohomological background. Section~\ref{sec:starting-surface}
gives the explicit starting cover, its integral models, the $\Z$-section and the bad fibres. Section~\ref{sec:realizable} defines the subspace
$\mathscr E(S_n,S_{-1})$ of realizable central extensions and the quotient
$H^2(G_n,\F)/\mathscr E(S_n,S_{-1})$. It bounds the dimension
of this quotient independently of $n$ by studying covers of $C_n/\Q$,
their quadratic twists and the bad fibres of $S_0\to S_{-1}$. Section~\ref{sec:main-proofs} proves the three
main theorems.

\subsection{Human--AI collaboration}\label{sec:intro-ai}
The authors began developing this approach from Summer 2022; inspired by the analogy with Shafarevich's original work \cite{Shafarevich}. The central issue they were stuck on was to find a starting arithmetic surface $S_0$ with many \'etale $\mathbb{F}_2$-covers. The naive example of hyperelliptic curves works very well over $\Q$, but integrally it only works on the nose when the discriminant of the associated Weierstrass polynomial is a unit, which can be done over a number field $K$. On the other hand changing the number field comes with a bill: it makes the codimension of the obstruction group grow with $\text{dim}_{\mathbb{F}_2}(\text{Sel}_2(K))$.

This seemed extremely delicate and the authors had never found the stamina to carry systematic experiments with this. Until Summer $2026$. 

At that point they decided to exploit the new technological advancements and run serious AI-experiments on this. This led to five days of intensive collaboration with \emph{GPT5.6 Sol},
\emph{Fable}, and several internal copies of \emph{Aletheia}, internally at Google DeepMind.

The agents constructed and studied many examples.
A recurring difficulty was that the $2$-Selmer ``bill" seemed to consistently conspire against the possibility to construct an example. Seeing this pattern, the authors ended up convincing themselves that under reasonable extra-assumptions this fails systematically. At that point the authors decided to ask the models to go back to $K=\Q$ and to allow more primes of bad reduction and to at least find a way to count obstructions only to settle \textbf{Question}(Venkatesh, 2022) above. At this point \emph{GPT 5.6 Sol} introduced the key idea of making everything \'etale on top of the auxiliary cover
$C_0\to C_{-1}$ and to have indeed only $g/\text{log}(g)$ primes of bad reduction for $C_{-1}$. During this interaction, the AI realized that with these choices one has only $O(g^2/\text{log}(g))$ obstructions at each step. And one gets a favorable $g^2/4$ different many classes to try: for $g$ sufficiently large one will be able to continue forever the construction of the tower. At that point the AI had realized to have settled \textbf{Question}(Venkatesh, 2022) and so went on and noticed that in this shifted perspective a resolution of \textbf{Question}(Bost--Charles, 2022) and \textbf{Question}(Diem--Frey, 2003) was available. 

When the process settled, the authors digested the new fundamental trick of making the tower only on top of the multiquadratic extension $C_0$ of $C_{-1}$. Armed with this insight, they obtained a first draft of the paper giving instructions to \emph{Astra}. They subsequently rewrote significant parts or instructed \emph{Astra} to rewrite them or both, until the work reached their satisfaction. They subsequently proofread the work using \emph{Astra}, \emph{Fable} and internal models at Google DeepMind.  

We regard this work as an instance of category $\mathbf{C}$ in the
human--AI collaboration taxonomy of the \emph{Aletheia} paper
\cite[Section~6.1, Table~8]{Aletheia}: both the human authors and the AI contributed essential parts of the argument.

\subsection*{Acknowledgments}
CP thanks the \emph{Superhuman reasoning team} at Google DeepMind for
building \emph{Aletheia} and for several conversations about mathematics
and artificial intelligence. In particular, he thanks Thang Luong for
making available to him the resources he used during this project. CP also thanks Akshay Venkatesh for conversations on this problem in the Spring of 2023 at IAS. 

The authors thank Fran\c cois Charles for putting them in touch with
Pablo Chenal upon being informed of the authors' results; which allowed for a simultaneous release. 

The idea of sacrificing odd primes arose as a possibility to be studied further already in a conversation between CP and Jean Gillibert; and played a crucial role in conversations with the agents. Part of the CP's background on integral models on curves was acquired in a conversation between CP and Maarten Derickx. CP warmly thanks both of them for this important contribution to the present work. 
\section{Preliminaries}\label{sec:dictionary}

\subsection{Geometric and cohomological conventions}

All rings are commutative with identity. We use the standard language of schemes and
local algebra as in \cite{Stacks} and \cite{Liu}, and the conventions for
\'etale cohomology of \cite{MilneEtale}. Our convention for projectivity
over a general base is that of \st{01W7}; over $\Spec A$ this means a
closed immersion into $\mathbb P^n_A$ for some integer $n\geq0$. We use the fibrewise criterion
for smoothness in \st{01V4}. In particular, for a monic polynomial
$P(z)\in A[z]$, the morphism $\Spec A[z]/(P)\to\Spec A$
is \'etale at every prime $\mathfrak q\subset A[z]/(P)$ for which
$P'(z)\notin\mathfrak q$ (\st{02GH}). If $f:X\to Y$ and
$g:Y\to Z$ are finite (respectively flat, proper, smooth or \'etale),
then $g\circ f:X\to Z$ is finite (respectively flat, proper, smooth
or \'etale). If $f:X\to Y$ has
one of these properties and $Y'\to Y$ is a morphism, then
$X\times_Y Y'\to Y'$ has that property; see \st{01WG},
\st{01W0}, \st{01V4} and \st{02GH}.

\paragraph{Arithmetic surfaces and normalization.}
The purpose of this section is to recall the basics of normalization. In Section~3.2, we will use normalization to extend the cover of curves constructed in Section~3.1 to a finite morphism between arithmetic surfaces.

An arithmetic surface over $B=\Spec\Z$ is an integral proper flat
$B$-scheme all of whose nonempty fibres have pure dimension one. A model
of a smooth proper geometrically connected curve $C/\Q$ is an
arithmetic surface $S$ together with an isomorphism $S_{\Q}\simeq C$;
see also \st{0C2R}. A model is not required to be regular unless this is
stated. For a scheme $X$ over a base $B$ and a specified morphism
$\Spec k\to B$, we write $X_k=X\times_B\Spec k$; a
geometric fibre is taken over an algebraic closure of the residue field.
Geometric connectedness and geometric integrality can be checked over
an algebraic closure (Stacks, Sections 33.7--33.9).

Let $S$ be a normal integral scheme with function field $K$, and let
$L/K$ be a finite extension. For each affine open $U=\Spec A\subset S$,
write $\overline A^{\,L}$ for the integral closure of $A$ in $L$.
Normalization commutes with localization: for $f\in A\setminus\{0\}$,
\begin{equation}\label{eq:bg-normalization-localization}
 \overline{A_f}^{\,L}=(\overline A^{\,L})_f.
\end{equation}
The identifications in \eqref{eq:bg-normalization-localization} identify
$\Spec\overline{A_f}^{\,L}$ with the inverse image of
$\Spec A_f\subset\Spec A$ in $\Spec\overline A^{\,L}$. Gluing the
schemes $\Spec\overline A^{\,L}$ using
\eqref{eq:bg-normalization-localization} gives the
normalization $S_L\to S$ (\st{0BAK}). For a reduced scheme $X$ with
irreducible components $X_i$, its normalization is
$\coprod_i\widetilde X_i\to X$, where $\widetilde X_i$ is the
normalization of the reduced integral scheme $X_i$.

A finite dominant morphism $f:S'\to S$ between normal integral schemes
is called Galois if the induced extension $K(S')/K(S)$ of function
fields is Galois. In that case $S'$ is the normalization of $S$ in
$K(S')$, and $\operatorname{Gal}(K(S')/K(S))$ acts on $S'$ by
automorphisms commuting with $f$. This definition does not require
$f:S'\to S$ to be \'etale.

Let $X$ be an integral scheme of finite type over $\Z$ that dominates
$\Spec\Z$ and has a curve as generic fibre. Then every irreducible
component of $X_{\mathbb F_p}$ is one-dimensional for every rational
prime $p$. Indeed, let $Z$ be an irreducible component of
$X_{\mathbb F_p}$, let $z$ be the generic point of $Z$, and choose an
affine open $\Spec A\subset X$ containing $z$. Let
$\mathfrak q\subset A$ be the prime corresponding to $z$. Since
$\mathfrak q$ is minimal over $pA$, the principal ideal theorem
(\st{00KV}) gives $\operatorname{ht}(\mathfrak q)\leq1$.
Since $p$ is nonzero in the domain $A$, we have
$(0)\subsetneq\mathfrak q$, so $\operatorname{ht}(\mathfrak q)=1$.
The dimension formula over the universally catenary ring $\Z$
(\st{02IJ}) gives
\begin{equation}\label{eq:bg-explicit-fibre-dimension}
\operatorname{ht}(\mathfrak q)
=1+\operatorname{trdeg}_{\Q}\Frac(A)
-\operatorname{trdeg}_{\mathbb F_p}\kappa(\mathfrak q).
\end{equation}
Since $X_{\Q}$ is a curve,
$\operatorname{trdeg}_{\Q}\Frac(A)=1$.
Equation~\eqref{eq:bg-explicit-fibre-dimension} therefore shows that
$\kappa(\mathfrak q)$ has transcendence degree one over
$\mathbb F_p$. Thus $\dim Z=1$ (\st{00P0}).

\paragraph{Local calculations on the models.}
In this section we recall the facts about blow-ups and local rings needed for the construction of the models in Section~3.2 and the proof of regularity of the relevant surfaces in Section~3.3.

For regularity and normality criteria, see
\st{00KU}, \st{00N7}, \st{00NT} and \st{031S}; for the conventions on
discrete valuation rings and Dedekind schemes, see \st{00PD}.

Let $X$ be a noetherian integral scheme, and let
$\mathcal I\subset\mathcal O_X$ be a nonzero coherent ideal sheaf.
The blow-up of $X$ with centre $V(\mathcal I)$ is the morphism
\begin{equation}\label{eq:bg-explicit-blowup}
 b:\operatorname{Bl}_{\mathcal I}(X)
   =\operatorname{Proj}_X\left(\bigoplus_{n\ge0}\mathcal I^n\right)
   \longrightarrow X
\end{equation}
(\st{01OG}). Write $X'=\operatorname{Bl}_{\mathcal I}(X)$. For an
affine open $U=\Spec A\subset X$, choose nonzero generators
$a_1,\ldots,a_r$ of $I=\Gamma(U,\mathcal I)$. Then
\begin{equation}\label{eq:bg-blowup-affine-opens}
 b^{-1}(U)=\bigcup_{j=1}^r U_j,\qquad
 U_j=\Spec A\left[\frac{a_1}{a_j},\ldots,\frac{a_r}{a_j}\right],
 \qquad
 A\left[\frac{a_1}{a_j},\ldots,\frac{a_r}{a_j}\right]\subset A[1/a_j]
\end{equation}
(\st{01OF}). The exceptional divisor of $b$ is the closed subscheme
\begin{equation}\label{eq:bg-exceptional-divisor}
 E_b=V(\mathcal I\mathcal O_{X'})\subset X'.
\end{equation}
The restriction
$b|_{X'\setminus E_b}:X'\setminus E_b\to X\setminus V(\mathcal I)$
is an isomorphism. If $Z\subset X$ is a closed subscheme, the strict
transform of $Z$ under $b$ is
\begin{equation}\label{eq:bg-strict-transform}
 Z^{\mathrm{str},b}
   =\overline{\,b^{-1}(Z\setminus V(\mathcal I))\,}^{\,X'},
\end{equation}
where the bar denotes scheme-theoretic closure in $X'$.
A resolution of an integral surface $X$ is a
proper birational morphism $\widetilde X\to X$ from a regular integral
surface. A scheme of finite type over $\Z$ is excellent, and excellence
is preserved under localization and henselization. If $X$ is an
excellent integral scheme with function field $K$ and $L/K$ is finite,
the normalization $X_L\to X$ is finite. If $X$ is excellent, the set
$\{x\in X:\mathcal O_{X,x}\text{ is regular}\}$ is open. If $A$ is
a normal excellent local ring, its completion $\widehat A$ is normal
(\st{07QS}). We also use the resolution theorem for normal arithmetic
surfaces (\st{0BGP}).

For a local ring $(A,\mathfrak m)$ with residue field $k$, write $A^h$
for its henselization. Fix a separable closure $k^s$ of $k$, and write
$A^{\sh}$ for the strict henselization relative to $k^s$. The ring
$A^h$ is the filtered colimit of the local rings $B_{\mathfrak n}$,
where $B$ is an \'etale $A$-algebra, $\mathfrak n$ lies above
$\mathfrak m$, and the induced map $k\to\kappa(\mathfrak n)$ is an
isomorphism. The ring $A^{\sh}$ is the corresponding filtered colimit
when one specifies an embedding $\kappa(\mathfrak n)\to k^s$ over $k$
instead (Stacks, Sections 10.153 and 15.45).
A local ring $(A,\mathfrak m)$ is henselian if, for every monic
$P(z)\in A[z]$ and every $\bar a\in A/\mathfrak m$ satisfying
$\overline P(\bar a)=0$ and $\overline P'(\bar a)\ne0$, there is a
unique $a\in A$ with $P(a)=0$ and $a\bmod\mathfrak m=\bar a$.
A strictly henselian local ring is a henselian local ring with
separably closed residue field (\st{03QK}). The strict localization
of $X$ at a geometric point
$\bar x$ is $\Spec\mathcal O^{\sh}_{X,\bar x}$ (\st{03PS}). Every
finite \'etale cover of $\Spec\mathcal O^{\sh}_{X,\bar x}$ is a disjoint
union of copies of $\Spec\mathcal O^{\sh}_{X,\bar x}$ (\st{04GK}).
Normalization commutes with localization, henselization and strict
henselization. If $A$ is normal, then $A_f$ is normal for every
$f\in A$; if $A$ is a normal local ring, then $A^h$ and $A^{\sh}$
are normal. If $A$ is a regular noetherian ring, then $A_f$ is
regular; if $A$ is a regular noetherian local ring, then $A^h$ and
$A^{\sh}$ are regular (\st{081J}, \st{0CBM}). For a noetherian local ring
$(A,\mathfrak m)$, put $\widehat A=\varprojlim_n A/\mathfrak m^n$.
Then $A$ is regular if and only if $\widehat A$ is regular (\st{07NY}).
If $X$ is of finite type over a noetherian local ring $R$ and
$X\times_R\widehat R$ is regular, then $X$ is regular (\st{0BG6}).

\paragraph{Curves and their fibres.}
In this section we introduce the terminology used in Section~3.5 to describe the bad fibres of our arithmetic surfaces: their irreducible components, the multiplicity of each component and the singularities of the reduced fibres.

For a nonzero rational function $f$ on a regular integral curve $C/k$,
we use
$\operatorname{div}(f)=\sum_Pv_P(f)P$ and
$\deg(\sum_P n_PP)=\sum_Pn_P[\kappa(P):k]$, where $P$ runs over the
closed points of $C$ and $v_P$ is the valuation of
$\mathcal O_{C,P}$.

\begin{definition}\label{def:bg-normal-crossings}
Let $X$ be a regular surface. A reduced effective divisor $D\subset X$
has strict normal crossings if, for every closed point $x\in D$, there
is a regular system of parameters $(t_1,t_2)$ of $\mathcal O_{X,x}$
such that the ideal of $D$ in $\mathcal O_{X,x}$ is either $(t_1)$ or
$(t_1t_2)$. In the second case, the two components through $x$ have
local ideals $(t_1)$ and $(t_2)$; we say that the components meet
transversely at $x$. A divisor has normal crossings if the same
condition holds after replacing $X$ by an \'etale neighbourhood of
each point (\st{0CBN}).
\end{definition}

If $f:X\to\Spec R$ is flat, $R$ is a discrete valuation ring with
uniformizer $\pi$, and $X$ is regular, the special fibre is the
effective divisor $X_s=\operatorname{div}(\pi)=\sum_i m_iD_i$.
The reduced special fibre is $(X_s)_{\mathrm{red}}=\sum_iD_i$.
Here the $D_i$ are the distinct irreducible components of $X_s$ and
$m_i$ is the valuation of $\pi$ at the generic point of $D_i$.

A reduced curve $C$ over an algebraically closed field $k$ is nodal if
every singular point $c\in C$ satisfies
$\widehat{\mathcal O}_{C,c}\simeq k[[u,v]]/(uv)$ (\st{0C1P}).
Such a singular point is called a node or an ordinary double point.
If $C$ is proper and nodal, write
$\nu:\coprod_{i\in I}\widetilde C_i\to C$ for the normalization,
where each $\widetilde C_i$ is connected. Each $\widetilde C_i$ is a
smooth proper curve, and the \emph{total normalization genus} of $C$
is $\sum_{i\in I}g(\widetilde C_i)$. An irreducible component
$C_i\subset C$ is called rational when $C_i\simeq\mathbb P^1_k$.

For a reduced curve $Y$ over an algebraically closed field $k$, a
branch at $y\in Y(k)$ corresponds to a
minimal prime $\mathfrak p$ of $\widehat{\mathcal O}_{Y,y}$, or
equivalently to a point of the normalization above $y$. The branch
defined by $\mathfrak p$ is smooth if
$\widehat{\mathcal O}_{Y,y}/\mathfrak p\simeq k[[t]]$.
For two distinct branches with primes $\mathfrak p_1,\mathfrak p_2$,
the intersection multiplicity of the branches is
\begin{equation}\label{eq:bg-explicit-branch-intersection}
 i_y(\mathfrak p_1,\mathfrak p_2)
 =\dim_k\widehat{\mathcal O}_{Y,y}/(\mathfrak p_1+\mathfrak p_2).
\end{equation}
If $\widehat{\mathcal O}_{Y,y}=k[[u,v]]/(uv)$, the two branch ideals
are $(u)$ and $(v)$, and \eqref{eq:bg-explicit-branch-intersection}
gives $i_y((u),(v))=\dim_k k=1$. The branches of
$\Spec k[[u,v]]/(v(v-u^2))$ are smooth and have intersection
multiplicity two, since the quotient in
\eqref{eq:bg-explicit-branch-intersection} is $k[[u]]/(u^2)$.

A finite undirected graph is a triple $\Gamma=(V,E,\partial)$,
where $V$ and $E$ are finite sets and $\partial$ assigns to every
edge an unordered pair of vertices, allowing repetitions. Thus
loops and several edges with the same endpoints are allowed. Two
vertices $v,w\in V$ are connected if there are an integer $r\geq0$,
vertices $v_0=v,v_1,\ldots,v_r=w$, and edges $e_1,\ldots,e_r$ with
$\partial(e_j)=\{v_{j-1},v_j\}$ for $1\leq j\leq r$.
Let $c(\Gamma)$ be the number of equivalence classes in $V$ for the
relation just defined, and put
\begin{equation}\label{eq:bg-explicit-graph-betti}
 b_1(\Gamma)=|E|-|V|+c(\Gamma).
\end{equation}
A tree is a nonempty connected graph with $b_1(\Gamma)=0$;
equivalently, it is connected and $|E|=|V|-1$.

Let $Y$ be a proper reduced curve over an algebraically closed field
$k$ such that every singular point has exactly two branches.
Write $\nu:\coprod_{i\in I}\widetilde Y_i\to Y$ for the
normalization, with each $\widetilde Y_i$ connected. Define
\begin{equation}\label{eq:bg-explicit-incidence-graph}
 \Gamma(Y)=(I,\operatorname{Sing}(Y),\partial),\qquad
 \partial(y)=\{i,j\},
\end{equation}
where the two points of $\nu^{-1}(y)$ lie on
$\widetilde Y_i$ and $\widetilde Y_j$.
The graph $\Gamma(Y)$ in \eqref{eq:bg-explicit-incidence-graph} is the
incidence graph of $Y$; when $Y$ is nodal, $\Gamma(Y)$ is also called
the dual graph. The definition of $\partial(y)$ does not impose a
condition on the intersection multiplicity of the two branches at $y$.

We write $\Pic(X)$ for the group of line bundles on $X$, identified on a
regular integral curve with divisors modulo principal divisors. For a
smooth proper geometrically connected curve $C/k$,
$g(C)=\dim_kH^1(C,\mathcal O_C)$ (\st{0BY7}), and the canonical bundle
$\omega_C=\Omega^1_{C/k}$ has degree $2g(C)-2$ (\st{0C1A}).

The Jacobian $J(C)$ is the abelian variety representing the degree-zero
Picard functor of $C/k$; see Stacks, Section 44.6, and \st{03RN}. If
$C$ has a rational point $P\in C(k)$, then
$J(C)(k)=\Pic^0(C)$, where $\Pic^0(C)$ is the group of degree-zero line
bundles on $C$ modulo isomorphism. Let us recall the descent argument.
Choose a separable closure $k^s$ of $k$ and a line bundle $L$ on
$C_{k^s}$ whose class is fixed by $G_k=\Gal(k^s/k)$. For each
$\sigma\in G_k$, choose an isomorphism
$\alpha_\sigma:\sigma^*L\xrightarrow{\sim}L$. Fix a trivialization
$L|_{P_{k^s}}\simeq k^s$ and rescale each $\alpha_\sigma$ so that it
carries the conjugate trivialization to the chosen trivialization of
$L|_{P_{k^s}}$. For $\sigma,\tau\in G_k$, put
\begin{equation}\label{eq:bg-rigidified-descent-cocycle}
 c_{\sigma,\tau}
 =\alpha_\sigma\circ\sigma^*\alpha_\tau\circ
       \alpha_{\sigma\tau}^{-1}\in\operatorname{Aut}(L).
\end{equation}
Since $C_{k^s}$ is geometrically connected and proper,
$\operatorname{Aut}(L)=(k^s)^\times$. The restriction of
$c_{\sigma,\tau}$ to $L|_{P_{k^s}}$ is the identity by the
normalization of $\alpha_\sigma$, $\alpha_\tau$ and
$\alpha_{\sigma\tau}$. Hence $c_{\sigma,\tau}=1$, or equivalently
$\alpha_\sigma\circ\sigma^*\alpha_\tau=\alpha_{\sigma\tau}$.
The line bundle $L$ and the trivialization at $P_{k^s}$ are defined
over a finite extension of $k$, so the isomorphisms
$\alpha_\sigma$ give continuous descent data. Descent for line
bundles (\st{023N}) gives a line bundle $L_0$ on $C$ with
$(L_0)_{k^s}\simeq L$.

By a hyperelliptic curve over $k$ we mean the smooth proper model of a
separable quadratic extension of $k(x)$. For existence, uniqueness and
functoriality of the smooth proper model over a perfect field, see
\st{0BXX}.

\paragraph{Cohomology.}
In this section we recall the basics we shall need later on about group and \'etale cohomology. We will use group cohomology in Section~4.1 to describe central extensions of Galois groups, and \'etale cohomology in Section~4.2 to bound the obstructions to realizing these extensions by finite \'etale covers of arithmetic surfaces.

For an \'etale sheaf $\mathcal F$ on $X$ and a morphism $f:X\to Y$,
we use $H^i_{\et}(X,\mathcal F)$ for \'etale cohomology and
$R^if_*\mathcal F$ for the higher direct image sheaf on $Y$.
For a geometric point $\bar x$ of $X$, the geometric stalk
$\mathcal F_{\bar x}$ is the colimit of $\mathcal F(U)$ over
\'etale morphisms $U\to X$ equipped with a lift of $\bar x$; see
\cite{MilneEtale} and \st{03RY}. We use the constructibility
conventions of the same references. We write
$h^i(X,\mathcal F)=\dim_{\F}H^i_{\et}(X,\mathcal F)$ and
$h^i(G,M)=\dim_{\F}H^i(G,M)$ when the coefficient sheaf or module is
over $\F$.
If $E_r^{a,b}\Rightarrow H^{a+b}$ is a first-quadrant spectral
sequence of finite-dimensional $\F$-vector spaces, then $H^n$ has
a finite filtration with quotients $E_\infty^{a,n-a}$. Since
$E_\infty^{a,n-a}$ is a subquotient of $E_r^{a,n-a}$, we have
\begin{equation}\label{eq:bg-spectral-sequence-dimension}
 \dim_{\F}H^n\leq\sum_{a=0}^n\dim_{\F}E_r^{a,n-a}.
\end{equation}

We will use this inequality in the proof of Proposition~\ref{prop:realizable-local-bounds} to bound the dimensions of the cohomology groups of the bad fibres of our arithmetic surfaces.

For a profinite group $G$ and a discrete continuous $G$-module $M$,
$H^i(G,M)$ means continuous group cohomology; for finite groups it is
ordinary group cohomology. We write $G_K=\Gal(K^{\mathrm{sep}}/K)$.
On a scheme where $2$ is invertible, we identify $\mu_2$ with the
constant sheaf $\F$ by sending $-1$ to $1$. For a field $K$ of
characteristic different from $2$, the Kummer character of a class
$[a]\in K^\times/K^{\times2}$ is the homomorphism
$\chi_a:G_K\to\F$ determined by
$\sigma(\sqrt a)=(-1)^{\chi_a(\sigma)}\sqrt a$; it corresponds to the
quadratic extension $K(\sqrt a)/K$ when $[a]\ne1$ (\st{03PK}). We
write $\Br(K)=H^2(G_K,(K^{\mathrm{sep}})^\times)$ for the Brauer group
of a field (\st{03R7}); for a scheme $X$, $\Br(X)$ denotes the torsion
subgroup of $H^2_{\et}(X,\mathbb G_m)$. We will use the Brauer--Hasse--Noether theorem in the proof of Proposition~\ref{prop:good-detection}. 

\subsection{Conductors}\label{sec:conductor-background}
We will use the results of this subsection in Section~\ref{sec:proof-conductor} to bound the logarithm of the conductor of each curve in our tower by a constant times its genus. Proposition~\ref{prop:conductor-model-transport} shows that a finite \'etale cover of a proper model induces, by normalization, a finite \'etale cover of any regular proper model of the same generic curve. Proposition~\ref{prop:conductor-common-extension} then shows that, at each prime, a single finite extension of the local field suffices (for the generic curves of the various covers) to acquire semistable reduction. Proposition~\ref{prop:conductor-local-bound} bounds their conductor exponents by a constant times their genera, with the constant depending only on this extension. Summing these bounds over the bad primes gives Proposition~\ref{prop:conductor-uniform-integral-covers}, whose bound is independent of the degree of the cover.

For a smooth projective geometrically connected curve $D/\Q$, a rational
prime $p$, and an auxiliary prime $q\ne p$, choose an algebraic closure
$\overline\Q_p$. Put $D_{\overline\Q_p}=D\times_{\Q}\Spec\overline\Q_p$
and
\begin{equation}\label{eq:bg-explicit-cohomology-representation}
 V_{D,q}=H^1_{\et}(D_{\overline\Q_p},\Q_q)
 =\left(\varprojlim_r H^1_{\et}
       (D_{\overline\Q_p},\Z/q^r\Z)\right)\otimes_{\Z_q}\Q_q.
\end{equation}
The action of $G_{\Q_p}$ on $D_{\overline\Q_p}$ induces a continuous
representation $\rho_{D,q}:G_{\Q_p}\to\operatorname{GL}(V_{D,q})$,
and $\dim_{\Q_q}V_{D,q}=2g(D)$.

For a finite extension $F/\Q_p$, let $I_F\subset G_F$ be the inertia
subgroup and let $G_F^u$ denote the upper ramification groups, defined
below. If $V$ is a finite-dimensional continuous $\Q_q$-representation
of $G_F$, write $V^H=\{v\in V:\sigma v=v\text{ for every }\sigma\in H\}$
for the invariant subspace of a subgroup $H\subset G_F$. Define the
Swan conductor $\operatorname{Sw}_F(V)$ and the Artin conductor
$a_F(V)$ by
\begin{equation}\label{eq:conductor-artin-swan}
 \operatorname{Sw}_F(V)=\int_0^\infty
           \dim_{\Q_q}(V/V^{G_F^u})\,du,\qquad
 a_F(V)=\dim_{\Q_q}(V/V^{I_F})+\operatorname{Sw}_F(V).
\end{equation}
For a smooth proper curve $D/F$, put
$a_F(D)=a_F(H^1_{\et}(D_{\overline F},\Q_q))$. For a curve $D/\Q$,
put $a_p(D)=a_{\Q_p}(V_{D,q})$. The integer $a_p(D)$ is independent
of $q\ne p$. The conductor of $D$, equivalently the conductor of
the Jacobian $J(D)$, is
\begin{equation}\label{eq:bg-explicit-conductor-product}
 \mathfrak N(D)=\prod_p p^{a_p(D)},\qquad
 a_p(D)=v_p(\mathfrak N(D)).
\end{equation}
Here $v_p(p)=1$. The name \emph{conductor exponent at $p$} for $a_p(D)$
refers to the exponent of $p$ in
\eqref{eq:bg-explicit-conductor-product}.

We follow the ramification conventions of
\cite[Chapters~IV and VI]{SerreLocal}. For a finite Galois extension
$E/F$ of local fields, with group $H$ and normalized valuation $v_E$,
the lower ramification groups are
\begin{equation}\label{eq:bg-explicit-lower-ramification}
 H_i=\{\sigma\in H:v_E(\sigma(a)-a)\ge i+1
                     \text{ for every }a\in\mathcal O_E\}\quad(i\ge0).
\end{equation}
Set $H_t=H_{\lceil t\rceil}$ for $t>0$ and define
\begin{equation}\label{eq:bg-upper-ramification}
 \varphi(t)=\int_0^t[H_0:H_s]^{-1}\,ds,\qquad
 H^u=H_{\varphi^{-1}(u)},\qquad
 G_F^u=\varprojlim_{E/F}\operatorname{Gal}(E/F)^u\quad(u\geq0).
\end{equation}
The inverse limit ranges over finite Galois extensions $E/F$ in a
chosen separable closure of $F$. The quotient map
$\operatorname{Gal}(E'/F)\to\operatorname{Gal}(E/F)$ sends
$\operatorname{Gal}(E'/F)^u$ onto $\operatorname{Gal}(E/F)^u$ for
$E\subset E'$, so the inverse limit in
\eqref{eq:bg-upper-ramification} defines a subgroup of $G_F$.
Write $P_F$ for wild inertia, the maximal pro-$p$ subgroup of $I_F$.
For $u>0$, $G_F^u$ is a pro-$p$ subgroup of $P_F$.
If $q\ne p$, the image of $P_F$ in $\operatorname{GL}(V)$ is finite
for every continuous finite-dimensional $\Q_q$-representation $V$
of $G_F$. Consequently there is a real number $b\geq0$ such that
$V^{G_F^u}=V$ for every $u>b$. Thus the integral defining
$\operatorname{Sw}_F(V)$ in \eqref{eq:conductor-artin-swan} is finite.

Let $R$ be a discrete valuation ring with fraction field $F$ and
uniformizer $\pi$. A regular semistable model of a smooth proper curve
$D/F$ is a regular proper flat $R$-scheme $W$, with a specified
identification $W_F\simeq D$, whose geometric special fibre is reduced
and has only nodes. At a node $w$ of the geometric special fibre
$W_{\overline k}$, where $k=R/(\pi)$, the equation of $W$ in an
\'etale neighbourhood of $w$ after an \'etale extension of $R$ is
$uv=\pi$. A proper nodal model $M_0/R$ with smooth generic fibre can
have equations $uv=\pi^e$ with $e\geq1$. To obtain a regular model
from $M_0$, define recursively
\begin{equation}\label{eq:bg-nodal-resolution-sequence}
 Z_n=\{z\in M_n:\mathcal O_{M_n,z}\text{ is not regular}\}_{\mathrm{red}},
 \qquad M_{n+1}=\operatorname{Bl}_{Z_n}(M_n)
       \quad\text{if }Z_n\ne\varnothing.
\end{equation}
Here $Z_n$ has the reduced closed subscheme structure. For the local
equations $uv=\pi^e$, there is an integer $N\geq0$ for which
$Z_N=\varnothing$. The model $M_N/R$ is regular, and the geometric
special fibre $(M_N)_{\overline k}$ is reduced and nodal
(\st{0CDD}, \st{0CDE}).

\begin{backgroundproposition}\label{prop:conductor-model-transport}
Let $R$ be an excellent discrete valuation ring with fraction field $F$,
and let $X/R$ be proper with smooth geometrically connected generic fibre
$D/F$. Let $Y\to X$ be finite \'etale with geometrically connected
generic fibre $D'$. If $W/R$ is a regular proper model of $D$, then the
normalization $W'$ of $W$ in $F(D')$ is finite \'etale over $W$.
\end{backgroundproposition}

\begin{proof}
The finite \'etale morphism $D'\to D$ identifies $F(D')$ as a finite
separable extension of the function field $F(D)$ of $W$. Since $W$ is
excellent, its normalization $\nu:W'\to W$ in $F(D')$ is finite.
Let $w$ be a codimension-one point of $W$. The local ring
$B=\mathcal O_{W,w}$ is a discrete valuation ring with fraction field
$F(D)$. The identification $W_F\simeq D=X_F$ gives an $R$-morphism
$j:\Spec F(D)\to X$. Since $X\to\Spec R$ is proper, the valuative
criterion gives an $R$-morphism $j_B:\Spec B\to X$ satisfying
$j_B|_{\Spec F(D)}=j$ (\st{0BX5}).

Write $Y\times_X\Spec B=\Spec A_B$. The $B$-algebra $A_B$ is finite
\'etale and has generic algebra
$A_B\otimes_BF(D)=F(D')$. Since $B$ is regular and $A_B$ is
\'etale over $B$, the ring $A_B$ is regular and therefore normal.
Since $A_B$ is torsion-free over $B$, the
map $A_B\to A_B\otimes_BF(D)=F(D')$ is injective. Since $A_B$ is
finite over $B$, every element of $A_B$ is integral over $B$.
Conversely, if $a\in F(D')$ is integral over $B$, then $a$ is
integral over $A_B$, so $a\in A_B$ because $A_B$ is normal. Thus
\begin{equation}\label{eq:bg-transport-integral-closure}
 A_B=\overline B^{\,F(D')},\qquad
 W'\times_W\Spec B\simeq\Spec A_B.
\end{equation}
The morphism $\Spec A_B\to\Spec B$ is \'etale. Therefore
$\nu:W'\to W$ is \'etale at every point of $\nu^{-1}(w)$, for
every codimension-one point $w\in W$.

The scheme $W'$ is normal and $W$ is regular. Choose an affine open
$\Spec A\subset W$ and write $\nu^{-1}(\Spec A)=\Spec A'$, where
$A\subset A'$ is a finite extension of normal domains. For
$w'\in\Spec A'$ and $w=\nu(w')$, the going-down theorem and
incomparability for primes in $A\subset A'$ give
\begin{equation}\label{eq:bg-transport-local-dimensions}
 \dim\mathcal O_{W',w'}=\dim\mathcal O_{W,w}
\end{equation}
(\st{00ON}). The morphism $\nu:W'\to W$ is generically \'etale
because $D'\to D$ is
\'etale. Purity of the branch locus (\st{0BMB}) therefore shows that
$\nu:W'\to W$ is \'etale everywhere.
\end{proof}

\begin{backgroundproposition}\label{prop:conductor-common-extension}
Let $F$ be a finite extension of $\Q_p$, and let $X/\mathcal O_F$ be
proper with smooth geometrically connected generic curve $D/F$.
There is a finite Galois extension $E/F$ and a regular semistable model
$W/\mathcal O_E$ of $D_E$ with the following property. For every finite
\'etale cover $Y\to X$ with geometrically connected generic curve $D'$,
the normalization $W'$ of $W$ in the function field $E(D'_E)$ is finite
\'etale over $W$ and is a regular semistable
model of $D'_E$.
\end{backgroundproposition}

\begin{proof}
Semistable reduction for curves (\st{0CDN}) gives a finite extension
$E_0/F$ and a semistable model $M_0/\mathcal O_{E_0}$ of $D_{E_0}$.
Let $E/F$ be the Galois closure of $E_0/F$. The base change
$M=M_0\times_{\mathcal O_{E_0}}\mathcal O_E$ is a nodal proper model
of $D_E$. Apply the recursion
\eqref{eq:bg-nodal-resolution-sequence} with initial model $M$.
Let $N$ be the first index for which $Z_N=\varnothing$, and put
$W=M_N$. By \st{0CDE}, the scheme $W$ is regular and
$W_{\overline k}$ is reduced and nodal, where $k$ is the residue
field of $E$. Thus $W/\mathcal O_E$ is a regular semistable model
of $D_E$.

Apply Proposition~\ref{prop:conductor-model-transport} to
$X_{\mathcal O_E}=X\times_{\mathcal O_F}\mathcal O_E$, the cover
$Y_{\mathcal O_E}\to X_{\mathcal O_E}$ and the regular model $W$.
The proposition gives a finite \'etale morphism $\nu:W'\to W$,
where $W'$ is the normalization of $W$ in $E(D'_E)$. Since $W$ is
regular and $\nu$ is \'etale, $W'$ is regular. Since $\nu$ is finite
and flat and $W\to\Spec\mathcal O_E$ is proper and flat,
$W'\to\Spec\mathcal O_E$ is proper and flat. The induced morphism
$W'_{\overline k}\to W_{\overline k}$ is finite \'etale. Since
$W_{\overline k}$ is reduced and nodal, $W'_{\overline k}$ is
reduced and nodal. Normalization commutes with restriction to
$W_E=D_E$, so $(W')_E$ is the normalization of $D_E$ in $E(D'_E)$.
The finite morphism $D'_E\to D_E$ has normal source and induces the
extension $E(D'_E)/E(D_E)$; hence $(W')_E\simeq D'_E$.
Thus $W'/\mathcal O_E$ is a regular semistable model of $D'_E$.
The choices of $E$ and $W$ depend on $D$, independently of the
cover $Y\to X$.
\end{proof}

For a smooth proper curve $D/E$ over a finite extension $E/\Q_p$
with a regular semistable model over $\mathcal O_E$, semistable
monodromy says that $I_E$ acts unipotently on
$V=H^1_{\et}(D_{\overline E},\Q_q)$ and the wild inertia subgroup
$P_E\subset I_E$ acts trivially on $V$, for every $q\ne p$.
The assertions about the
actions of $I_E$ and $P_E$ follow from the
semistable nearby-cycle calculation in
\cite[Section~6.3, p.~104]{Illusie} and the proper comparison in
\cite[Section~1.2(b), (1.7), p.~86]{Illusie}.
The triviality of the $P_E$-action also follows from the unipotence
of the $I_E$-action. To see this, let
$\rho:G_E\to\operatorname{GL}(V)$ be the cohomology representation.
Choose an open pro-$q$ subgroup $K_0\subset\rho(G_E)$. The group
$\rho(P_E)\cap K_0$ is both pro-$p$ and pro-$q$, so
$\rho(P_E)\cap K_0=1$. Since $K_0$ has finite index in the compact
group $\rho(G_E)$, the group $\rho(P_E)$ is finite. A unipotent
matrix of finite order over $\Q_q$ is the identity, so
$\rho(P_E)=1$.

\begin{backgroundproposition}\label{prop:conductor-local-bound}
Let $D/F$ be a smooth projective geometrically connected curve over a
finite extension of $\Q_p$. Suppose $D$ acquires semistable reduction
over a finite Galois extension $E/F$. Put
\begin{equation}\label{eq:bg-explicit-last-break}
 b(E/F)=\sup\bigl(\{0\}\cup
       \{u>0:\operatorname{Gal}(E/F)^u\ne1\}\bigr).
\end{equation}
Then the integer $a_F(D)$ defined in
\eqref{eq:conductor-artin-swan} satisfies
\begin{equation}\label{eq:bg-explicit-local-bound}
 a_F(D)\le 2\bigl(1+b(E/F)\bigr)g(D).
\end{equation}
\end{backgroundproposition}

\begin{proof}
Fix an algebraic closure $\overline F$ containing $E$. Put
$b=b(E/F)$ as in \eqref{eq:bg-explicit-last-break} and
$V=H^1_{\et}(D_{\overline F},\Q_q)$ for $q\ne p$.
For $u>b$, compatibility of upper numbering with the quotient map
$G_F\to\operatorname{Gal}(E/F)$ gives
\begin{equation}\label{eq:bg-local-bound-kernel}
 G_F^u\subseteq\ker\bigl(G_F\to\operatorname{Gal}(E/F)\bigr)=G_E.
\end{equation}
Since $u>0$, the group $G_F^u$ acts trivially on the algebraic
closure of the residue field. Hence $G_F^u\subseteq I_E$.
The group $G_F^u$ is pro-$p$, and every finite quotient of
$I_E/P_E$ has order prime to $p$. Therefore the homomorphism
$G_F^u\to I_E/P_E$ is trivial and $G_F^u\subseteq P_E$.
Semistable monodromy for $D_E$ gives $V^{P_E}=V$, so
\begin{equation}\label{eq:bg-local-bound-invariants}
 V^{G_F^u}=V\qquad(u>b).
\end{equation}
For $0<u\leq b$, we have
$\dim_{\Q_q}(V/V^{G_F^u})\leq\dim_{\Q_q}V=2g(D)$.
Using \eqref{eq:bg-local-bound-invariants} in
\eqref{eq:conductor-artin-swan}, we obtain
\begin{equation}\label{eq:bg-local-bound-integral}
 \operatorname{Sw}_F(V)
 =\int_0^b\dim_{\Q_q}(V/V^{G_F^u})\,du
 \leq2b\,g(D).
\end{equation}
Finally,
$a_F(D)=\dim_{\Q_q}(V/V^{I_F})+\operatorname{Sw}_F(V)
\leq2g(D)+2b\,g(D)=2(1+b)g(D)$, as claimed in
\eqref{eq:bg-explicit-local-bound}.
\end{proof}

\begin{backgroundproposition}\label{prop:conductor-uniform-integral-covers}
Let $S/\Spec\Z$ be a proper flat arithmetic surface with smooth
geometrically connected generic curve $C/\Q$. Let $\Sigma$ be a finite
set of primes such that $S$ is smooth over
$U=\Spec\Z[1/\prod_{p\in\Sigma}p]$.
For each $p\in\Sigma$, let $E_p/\Q_p$ be a finite Galois extension
over which $C$ has semistable reduction, and put $b_p=b(E_p/\Q_p)$.
Then every finite \'etale cover $T\to S$ with geometrically connected
generic curve $D/\Q$ satisfies
\begin{equation}\label{eq:bg-explicit-uniform-bound}
 a_p(D)=0\quad(p\notin\Sigma),\qquad
 \log\mathfrak N(D)\le A(S,\Sigma)g(D),
 \quad
 A(S,\Sigma)=2\sum_{p\in\Sigma}(1+b_p)\log p.
\end{equation}
In particular, the constant $A(S,\Sigma)$ in
\eqref{eq:bg-explicit-uniform-bound} is independent of the degree of $T\to S$.
\end{backgroundproposition}

\begin{proof}
The composition $T_U\to S_U\to U$ is smooth and proper because
$T_U\to S_U$ is finite \'etale and $S_U\to U$ is smooth and proper.
For every $p\notin\Sigma$ and $q\ne p$, smooth proper base change
(\st{095T}, \st{0GLI}) implies that $I_{\Q_p}$ acts trivially on
$H^1_{\et}(D_{\overline\Q_p},\Q_q)$, so $a_p(D)=0$.

For $p\in\Sigma$, choose a regular semistable model
$W_p/\mathcal O_{E_p}$ of $C_{E_p}$. Apply
Proposition~\ref{prop:conductor-model-transport} to
$S_{\mathcal O_{E_p}}$, its finite \'etale cover
$T_{\mathcal O_{E_p}}$ and $W_p$. The proposition gives a finite
\'etale morphism $W'_p\to W_p$, where $W'_p$ is the normalization
of $W_p$ in $E_p(D_{E_p})$. Write $k_p$ for the residue field of
$E_p$. Since $W_p$ is regular and
the morphism $W'_p\to W_p$ is \'etale, $W'_p$ is regular.
The base change
$(W'_p)_{\overline{k_p}}\to(W_p)_{\overline{k_p}}$ is finite
\'etale, and $(W_p)_{\overline{k_p}}$ is reduced and nodal.
Hence $(W'_p)_{\overline{k_p}}$ is reduced and nodal. Moreover,
$(W'_p)_{E_p}\simeq D_{E_p}$ and $W'_p\to\Spec\mathcal O_{E_p}$
is proper and flat. Thus $D_{E_p}$ has semistable reduction.
Proposition~\ref{prop:conductor-local-bound} gives
$a_p(D)\leq2(1+b_p)g(D)$ for every $p\in\Sigma$. Consequently
\begin{equation}\label{eq:bg-uniform-conductor-sum}
 \log\mathfrak N(D)
 =\sum_{p\in\Sigma}a_p(D)\log p
 \leq2g(D)\sum_{p\in\Sigma}(1+b_p)\log p
 =A(S,\Sigma)g(D),
\end{equation}
as claimed in \eqref{eq:bg-explicit-uniform-bound}.
\end{proof}
\subsection{Changing the genus of a pointed curve}
The goal of this section is to establish Proposition~\ref{prop:genus-spreading}, which will be used in Section~\ref{sec:proof-curves} to obtain the towers of Theorem~\ref{thm:main-curves} with a base curve of any, but finitely many, genera. The proof constructs a double cover using a section of the square of a line bundle. The degree of the line bundle determines the genus by Riemann--Hurwitz. 

Let $C/\Q$ be a smooth projective geometrically connected curve, and
let $P\in C(\Q)$. We construct a ramified double cover $f:D\to C$
with prescribed genus $g(D)$ and with two rational points above $P$.
We also prove that, for every geometrically connected finite \'etale
Galois cover $C'\to C$, the curve $D\times_C C'$ is geometrically
connected. Consequently, a tower of geometrically connected finite
\'etale Galois covers $C_n\to C$ gives a tower
$D_n=D\times_C C_n\to D$, with
\begin{equation}\label{eq:genus-tower-basechange}
 \deg(D_n\to D)=\deg(C_n\to C),\qquad
 \operatorname{Aut}_D(D_n)\simeq\operatorname{Aut}_C(C_n).
\end{equation}

\begin{proposition}\label{prop:genus-spreading}
Let $C/\Q$ be a smooth projective geometrically connected curve of genus
$b\ge2$, and let $P\in C(\Q)$. For every integer $h\ge3b$ there are a
smooth projective geometrically connected curve $D/\Q$ of genus $h$
and a ramified double cover $f:D\to C$ in which $P$ splits completely.
For every geometrically connected finite \'etale Galois cover $C'\to C$,
the pullback $D\times_C C'\to D$ is geometrically connected and has the
same Galois group and degree as $C'\to C$. If $P$ splits completely in
$C'\to C$, each of the two rational points in $f^{-1}(P)$ splits
completely in $D\times_C C'\to D$.
\end{proposition}

\begin{proof}
Put $k=h-2b+1\ge b+1$ and $L=\mathcal O_C(kP)$, and choose an
identification $L|_P\simeq\Q$. Use the induced identification
$L^{\otimes2}|_P\simeq\Q$ to define the value $s(P)$ of a section
$s\in H^0(C,L^{\otimes2})$.
For a line bundle $M$ on $C$ with $\deg M>2b-2$, Serre duality gives
\begin{equation}\label{eq:genus-explicit-duality}
 H^1(C,M)\simeq H^0(C,\omega_C\otimes M^{-1})^\vee=0,
\end{equation}
since $\deg(\omega_C\otimes M^{-1})<0$; see \st{0E39}.
For every field extension $E/\Q$, the inequality
$\deg(M_E)=\deg M>2b-2$ and Serre duality give $H^1(C_E,M_E)=0$.
Applying \eqref{eq:genus-explicit-duality} to $L^{\otimes2}(-P)$ and
then taking cohomology of
$0\to L^{\otimes2}(-P)\to L^{\otimes2}\to L^{\otimes2}|_P\to0$ gives
\begin{equation}\label{eq:genus-explicit-evaluation}
 H^1(C,L^{\otimes2}(-P))=0,\qquad
 H^0(C,L^{\otimes2})\twoheadrightarrow L^{\otimes2}|_P.
\end{equation}
Thus
\begin{equation}\label{eq:genus-explicit-sections}
 V=\{s\in H^0(C,L^{\otimes2}):s(P)=1\}
\end{equation}
is a nonempty affine space over $\Q$.

Write $\overline C=C_{\overline\Q}$ and
$\overline L=L|_{\overline C}$.
For $x\in\overline C(\overline\Q)$ and a local generator $e$
of $\overline L^{\otimes2}$ at $x$, write $s=ae$ and define
$\operatorname{ord}_x(s)=v_x(a)$.
Replacing $e$ by $ue$, with $u\in\mathcal O_{\overline C,x}^{\times}$,
replaces $a$ by $u^{-1}a$, so this definition does not depend on $e$.
We will choose $s\in V(\Q)$ such that
$\operatorname{ord}_x(s)=1$ at every zero
$x\in\overline C(\overline\Q)$ of $s$.
For $x\in\overline C(\overline\Q)\setminus\{P\}$, let
$2x\subset\overline C$ be the closed subscheme defined by the ideal
sheaf $\mathcal O_{\overline C}(-2x)$; the local ring of $2x$ at $x$ is
$\mathcal O_{\overline C,x}/\mathfrak m_x^2$ and has length two.
The inequality
\begin{equation}\label{eq:genus-explicit-restriction-degree}
 \deg(\overline L^{\otimes2}(-P-2x))=2k-3>2b-2
\end{equation}
and Serre duality give
$H^1(\overline C,\overline L^{\otimes2}(-P-2x))=0$.
The restriction sequence therefore gives a surjection
\begin{equation}\label{eq:genus-explicit-restriction-map}
 H^0(\overline C,\overline L^{\otimes2}(-P))
 \longrightarrow H^0(2x,\overline L^{\otimes2}|_{2x}).
\end{equation}
The vector space on the right of
\eqref{eq:genus-explicit-restriction-map} has dimension two over
$\overline\Q$. Fix $s_0\in V(\Q)$. Every element of
$V(\overline\Q)$ has a unique expression $s_0+u$, with
$u\in H^0(\overline C,\overline L^{\otimes2}(-P))$.
The condition $\operatorname{ord}_x(s_0+u)\ge2$ is equivalent to
\begin{equation}\label{eq:genus-explicit-double-zero}
 u|_{2x}=-s_0|_{2x}.
\end{equation}
Surjectivity in \eqref{eq:genus-explicit-restriction-map} implies
that the set of solutions of \eqref{eq:genus-explicit-double-zero}
is an affine subspace of $V_{\overline\Q}$ of codimension two.

We next construct a closed subscheme $Z\subset V\times C$ whose
geometric points satisfy
\begin{equation}\label{eq:genus-incidence-geometric-points}
 Z(\overline\Q)=
 \{(s,x)\in V(\overline\Q)\times\overline C(\overline\Q):
                       \operatorname{ord}_x(s)\ge2\}.
\end{equation}
Choose a basis $s_1,\ldots,s_r$ of
$H^0(C,L^{\otimes2}(-P))$, and give $V$ coordinates
$a_1,\ldots,a_r$ by the equality $s=s_0+\sum_{j=1}^r a_js_j$.
Choose an affine open covering $(W_\alpha)_\alpha$ of $C$ such that
each $L^{\otimes2}|_{W_\alpha}$ has a generator $e_\alpha$ and
each $W_\alpha$ admits an \'etale morphism
$t_\alpha:W_\alpha\to\mathbb A^1_\Q$.
For $0\le j\le r$, define $f_{j,\alpha}\in\mathcal O_C(W_\alpha)$
by $s_j|_{W_\alpha}=f_{j,\alpha}e_\alpha$.
For a regular function $a$ on $W_\alpha$, define
$\partial a/\partial t_\alpha$ by
$da=(\partial a/\partial t_\alpha)\,dt_\alpha$.
Define the function $F_\alpha$ and the ideal sheaf
$\mathcal I_\alpha$ on $V\times W_\alpha$ by
\begin{equation}\label{eq:genus-explicit-incidence-equations}
 F_\alpha=f_{0,\alpha}+\sum_{j=1}^r a_jf_{j,\alpha},\qquad
 \mathcal I_\alpha=
 \left(F_\alpha,\frac{\partial F_\alpha}{\partial t_\alpha}\right)
 \mathcal O_{V\times W_\alpha}.
\end{equation}
For $W_{\alpha\beta}=W_\alpha\cap W_\beta$, define units
$\lambda_{\alpha\beta},c_{\alpha\beta}$ on $W_{\alpha\beta}$ by
$e_\alpha=\lambda_{\alpha\beta}e_\beta$ and
$dt_\beta=c_{\alpha\beta}\,dt_\alpha$.
Restriction to $V\times W_{\alpha\beta}$ gives
\begin{equation}\label{eq:genus-explicit-gluing}
 \begin{aligned}
 F_\beta&=\lambda_{\alpha\beta}F_\alpha,\\
 \frac{\partial F_\beta}{\partial t_\beta}
 &=c_{\alpha\beta}^{-1}\lambda_{\alpha\beta}
       \frac{\partial F_\alpha}{\partial t_\alpha}
   +c_{\alpha\beta}^{-1}
       \frac{\partial\lambda_{\alpha\beta}}{\partial t_\alpha}
       F_\alpha,\\
 \mathcal I_\beta|_{V\times W_{\alpha\beta}}
 &=\mathcal I_\alpha|_{V\times W_{\alpha\beta}}.
 \end{aligned}
\end{equation}
The ideal sheaves $\mathcal I_\alpha$ therefore define an ideal
sheaf $\mathcal I\subset\mathcal O_{V\times C}$. Define $Z$ to be
the closed subscheme of $V\times C$ with ideal sheaf $\mathcal I$.
A geometric point $(s,x)$ belongs to $Z$ if and only if the constant
and linear coefficients of $s$ in a local parameter at $x$ are zero.
This proves the asserted description of $Z(\overline\Q)$.

Let $p_C:Z\to C$ and $p_V:Z\to V$ be the restrictions of the
projections $(s,x)\mapsto x$ and $(s,x)\mapsto s$, respectively.
Since $s(P)=1$ for every $s\in V$, the fibre $p_C^{-1}(P)$ is empty.
For $x\in\overline C(\overline\Q)\setminus\{P\}$, equation
\eqref{eq:genus-explicit-double-zero} identifies $p_C^{-1}(x)$
with an affine space of dimension $r-2$. The dimension bound for
the fibres of $p_C$ gives
\begin{equation}\label{eq:genus-explicit-incidence-dimension}
 \dim Z\le\dim C+r-2=\dim V-1.
\end{equation}
The morphism $p_V$ is proper, because $Z$ is closed in $V\times C$
and $C$ is proper over $\Q$. Hence $p_V(Z)$ is closed in $V$, and
\begin{equation}\label{eq:genus-explicit-image-dimension}
 \dim p_V(Z)\le\dim Z\le\dim V-1.
\end{equation}
The nonempty open subset $V\setminus p_V(Z)$ of the affine space
$V$ contains a rational point. Choose
$s\in(V\setminus p_V(Z))(\Q)$. Then $s(P)=1$ and
$\operatorname{ord}_x(s)=1$ at every geometric zero $x$ of $s$.

The section $s$ defines the $\mathcal O_C$-linear map
$\mu_s:L^{-2}\to\mathcal O_C$ that sends $e^{-2}$ to $a_W$
whenever $e$ generates $L|_W$ and $s|_W=a_We^2$.
Make $\mathcal B=\mathcal O_C\oplus L^{-1}$ an
$\mathcal O_C$-algebra by the multiplication rule
\begin{equation}\label{eq:genus-explicit-algebra-product}
 (a,u)(b,v)=(ab+\mu_s(u\otimes v),\,av+bu).
\end{equation}
Define the curve $D$ and the morphism $f$ by
\begin{equation}\label{eq:genus-explicit-relative-spectrum}
 D=\operatorname{Spec}_C(\mathcal B),\qquad f:D\longrightarrow C.
\end{equation}
Since $\mathcal B$ is locally free of rank two over $\mathcal O_C$,
the morphism $f$ is finite locally free of degree two.
For an affine open $W=\Spec A\subset C$ with a generator $e$
of $L|_W$, write $s|_W=a_We^2$. Then
\begin{equation}\label{eq:genus-explicit-affine-cover}
 f^{-1}(W)=\Spec A[w]/(w^2-a_W).
\end{equation}
On $f^{-1}(\{a_W\ne0\})$, the element $w$ is a unit and
$\partial(w^2-a_W)/\partial w=2w$ is a unit. Consequently, $f$ is
\'etale on $f^{-1}(\{a_W\ne0\})$, and $D$ is smooth over $\Q$
on that open set.
At a geometric zero $x$ of $a_W$, choose a local parameter $t$ on
$\overline C$ at $x$. The equality
$\operatorname{ord}_x(a_W)=1$ gives
$(\partial a_W/\partial t)(x)\ne0$. Hence
$\partial(w^2-a_W)/\partial t\ne0$ at the point $(x,0)$ of $D$,
so $D$ is smooth at $(x,0)$.
In the local ring of $D$ at $(x,0)$, the element $w$ is a
uniformizer and $a_W=w^2$. Thus $f$ has ramification index two at
$(x,0)$. It follows that the ramification points of $f$ are exactly
the points above the zeros of $s$. The number of geometric zeros
of $s$ is $\deg L^{\otimes2}=2k$.

Choose a rational generator $e$ of $\overline L$ and write $s=ae^2$,
with $a\in\overline\Q(C)^\times$. At a zero $x$ of $s$,
\begin{equation}\label{eq:genus-explicit-odd-valuation}
 v_x(a)=1-2\operatorname{ord}_x(e)\equiv1\pmod2.
\end{equation}
Therefore $a\notin\overline\Q(C)^{\times2}$ and
$\overline\Q(C)[w]/(w^2-a)$ is a field.
Every connected component of $D_{\overline\Q}$ maps onto
$\overline C$: the restriction of the finite flat map $f$ to a
connected component has nonempty open and closed image.
Since the generic fibre of $f_{\overline\Q}$ is the spectrum of
one field, $D_{\overline\Q}$ has one connected component.
Thus $D$ is geometrically connected.
The curve $D$ is projective, since $f$ is finite and $C$ is projective.
By $s(P)=1$, the fibre of $f$ above $P$ is
\begin{equation}\label{eq:genus-explicit-point-fibre}
 D\times_{C,P}\Spec\Q=\Spec\Q[w]/(w^2-1)
 \simeq\Spec\Q\amalg\Spec\Q.
\end{equation}
The Riemann--Hurwitz formula for $f$ (\st{0C1B}) gives
\begin{equation}\label{eq:genus-explicit-riemann-hurwitz}
 2g(D)-2=2(2b-2)+2k,\qquad g(D)=2b-1+k=h.
\end{equation}

Let $C'\to C$ be as in the statement. Put
$K=\overline\Q(C)$, $M=\overline\Q(D)$ and
$M'=\overline\Q(C')$, viewed as subfields of an algebraic closure of
$K$. The extension $M/K$ is ramified, whereas $M'/K$ is unramified
at every point of $\overline C$. Hence $M\not\subseteq M'$.
Since $[M:K]=2$, we obtain
\begin{equation}\label{eq:genus-explicit-disjointness}
 M\cap M'=K,\qquad [MM':M']=[M:K]=2.
\end{equation}
Multiplication induces an isomorphism of $K$-algebras
\begin{equation}\label{eq:genus-explicit-generic-tensor}
 M\otimes_K M'\xrightarrow{\ \sim\ }MM',\qquad
 a\otimes b\longmapsto ab.
\end{equation}
Put $T=(D\times_C C')_{\overline\Q}$, and let
$p_D:T\to D_{\overline\Q}$ be the morphism induced by
$(d,c')\mapsto d$. The morphism $p_D$ is finite \'etale, and
\begin{equation}\label{eq:genus-tensor-generic-fibre}
 T\times_{D_{\overline\Q}}\Spec M=\Spec(M\otimes_K M')
                                      =\Spec MM'.
\end{equation}
For each connected component $T_j\subset T$, the restriction
$p_D|_{T_j}$ is finite \'etale. Thus $p_D(T_j)$ is nonempty, open
and closed in the connected curve $D_{\overline\Q}$, and
$p_D(T_j)=D_{\overline\Q}$. Every $T_j$ therefore contains a point
above the generic point of $D_{\overline\Q}$.
Since $\Spec MM'$ consists of one point, $T$ has one connected
component. This proves that $D\times_C C'$ is geometrically connected.
Write $\pi_D:D\times_C C'\to D$ for the projection and put
$d=\deg(C'\to C)$. Base change gives
\begin{equation}\label{eq:genus-pullback-degree-group}
 \deg\pi_D=d,\qquad
 \operatorname{Aut}_D(D\times_C C')\simeq\operatorname{Aut}_C(C').
\end{equation}
Let $Q_1,Q_2\in D(\Q)$ be the points with $f(Q_j)=P$ from
\eqref{eq:genus-explicit-point-fibre}. If $P$ splits completely in
$C'\to C$, then, for $j=1,2$,
\begin{equation}\label{eq:genus-pullback-split-fibres}
 (D\times_C C')\times_{D,Q_j}\Spec\Q
 \simeq C'\times_{C,P}\Spec\Q
 \simeq\coprod_{a=1}^{d}\Spec\Q,
\end{equation}
as desired.
\end{proof}
\section{The starting surface}\label{sec:starting-surface}

We construct a hyperelliptic curve $C_{-1}/\Q$, an \'etale
$\F^g$-cover $C_0\to C_{-1}$, and a finite morphism
$S_0\to S_{-1}$ of arithmetic surfaces with generic fibre
$C_0\to C_{-1}$. The surface $S_0$ is integral, regular, projective
and flat over $\Spec\Z$, and has a section.
The quotient $\F^g$ will give a lower bound $g$ for the number of
generators in the Golod--Shafarevich argument. We bound the numbers
of irreducible components and nodes in the geometric bad fibres of
$S_{-1}$, and the stabilizers of geometric points of $S_0$ under
$\F^g$.

\subsection{The curves and their covering map}

Fix an integer $m\geq 4$ and put
\begin{equation}\label{eq:start-explicit-parameters}
 g=2^m-1,\qquad \ell=g+1=2^m,\qquad
 f_i(x)=(x-i)(x-i-\ell)\quad(0\leq i\leq g).
\end{equation}
The roots of $f_0,\ldots,f_g$ are distinct and their union is
$\{0,1,\ldots,2g+1\}$. Set
\begin{equation}\label{eq:start-explicit-polynomials}
 F(x)=\prod_{i=0}^{g}f_i(x)=\prod_{a=0}^{2g+1}(x-a),\qquad
 h_i(t)=(1-it)(1-(i+\ell)t).
\end{equation}
Define the affine curves and their coordinate rings by
\begin{equation}\label{eq:start-explicit-base-affine-rings}
 \begin{aligned}
 U_{-1,x}&=\Spec A_{-1,x},&
 A_{-1,x}&=\Q[x,y]/(y^2-F(x)),\\
 U_{-1,t}&=\Spec A_{-1,t},&
 A_{-1,t}&=\Q[t,w]/\bigl(w^2-\prod_{i=0}^g h_i(t)\bigr).
 \end{aligned}
\end{equation}
The assignments $t\mapsto x^{-1}$ and
$w\mapsto x^{-(g+1)}y$ define a ring isomorphism
\begin{equation}\label{eq:start-explicit-base-overlap}
 \beta_{-1}:A_{-1,t}[t^{-1}]\xrightarrow{\ \sim\ }
                 A_{-1,x}[x^{-1}],\qquad
 t\longmapsto x^{-1},\quad w\longmapsto x^{-(g+1)}y.
\end{equation}
Define $C_{-1}$ by identifying the open subschemes
$\Spec A_{-1,x}[x^{-1}]\subset U_{-1,x}$ and
$\Spec A_{-1,t}[t^{-1}]\subset U_{-1,t}$ through
$\Spec\beta_{-1}$.
Let $B_x=\Spec\Q[x]$ and $B_t=\Spec\Q[t]$ be the standard affine
opens of $\mathbb P^1_\Q$, identified by $t=x^{-1}$ on
$B_x\cap B_t=\Spec\Q[x,x^{-1}]$.
The inclusions $\Q[x]\to A_{-1,x}$ and
$\Q[t]\to A_{-1,t}$ commute with $\beta_{-1}$ after localization,
and hence define a morphism $\rho_{-1}:C_{-1}\to\mathbb P^1_\Q$.

Define $U_{0,x}=\Spec A_{0,x}$ and $U_{0,t}=\Spec A_{0,t}$ by
\begin{align}
 A_{0,x}&=\Q[x,z_0,\ldots,z_g]/(z_i^2-f_i(x))_{0\leq i\leq g},
 \label{eq:start-affine-cover}\\
 A_{0,t}&=\Q[t,v_0,\ldots,v_g]/(v_i^2-h_i(t))_{0\leq i\leq g}.
 \label{eq:start-infinity-cover}
\end{align}
The assignments $t\mapsto x^{-1}$ and $v_i\mapsto x^{-1}z_i$
define a ring isomorphism
\begin{equation}\label{eq:start-explicit-cover-overlap}
 \beta_0:A_{0,t}[t^{-1}]\xrightarrow{\ \sim\ }A_{0,x}[x^{-1}],
 \qquad t\longmapsto x^{-1},\quad v_i\longmapsto x^{-1}z_i.
\end{equation}
Define $C_0$ by identifying the open subschemes
$\Spec A_{0,x}[x^{-1}]\subset U_{0,x}$ and
$\Spec A_{0,t}[t^{-1}]\subset U_{0,t}$ through $\Spec\beta_0$.
The inclusions $\Q[x]\to A_{0,x}$ and $\Q[t]\to A_{0,t}$
commute with $\beta_0$ after localization and define a morphism
$\rho_0:C_0\to\mathbb P^1_\Q$.
Define ring homomorphisms
\begin{equation}\label{eq:start-explicit-ring-maps}
 \begin{aligned}
 \alpha_x:A_{-1,x}&\longrightarrow A_{0,x},&
       x&\longmapsto x,&y&\longmapsto\prod_{i=0}^g z_i,\\
 \alpha_t:A_{-1,t}&\longrightarrow A_{0,t},&
       t&\longmapsto t,&w&\longmapsto\prod_{i=0}^g v_i.
 \end{aligned}
\end{equation}
The identities $\prod_i z_i^2=F(x)$ and
$\prod_i v_i^2=\prod_i h_i(t)$ show that $\alpha_x$ and
$\alpha_t$ respect the defining relations of $A_{-1,x}$ and
$A_{-1,t}$. Their localizations satisfy
\begin{equation}\label{eq:start-ring-map-compatibility}
 \beta_0\circ\alpha_t[t^{-1}]
       =\alpha_x[x^{-1}]\circ\beta_{-1}
 :A_{-1,t}[t^{-1}]\longrightarrow A_{0,x}[x^{-1}].
\end{equation}
Consequently, $\Spec\alpha_x$ and $\Spec\alpha_t$ define a morphism
\begin{equation}\label{eq:start-curve-map}
 \pi_0:C_0\longrightarrow C_{-1},\qquad
 \begin{cases}
 (x,z_0,\ldots,z_g)\longmapsto(x,\prod_{i=0}^g z_i)
                                      &\text{on }U_{0,x},\\
 (t,v_0,\ldots,v_g)\longmapsto(t,\prod_{i=0}^g v_i)
                                      &\text{on }U_{0,t}.
 \end{cases}
\end{equation}
The identity $\rho_{-1}\circ\pi_0=\rho_0$ follows from
\eqref{eq:start-explicit-ring-maps}.
For background on smooth projective models and hyperelliptic curves, see
\cite[Section~7.4.3]{Liu} and \st{0BY1}, \st{0BY3}.

\begin{proposition}\label{prop:start-curves}
The curves $C_{-1}$ and $C_0$ are smooth, projective and geometrically
integral over $\mathbb Q$. The degrees of the morphisms to
$\mathbb P^1_\Q$ are
$\deg\rho_{-1}=2$ and $\deg\rho_0=2^{g+1}$.
\end{proposition}

\begin{proof}
Let $k$ be an algebraic closure of $\mathbb Q$.
For a nonempty subset $I\subseteq\{0,\ldots,g\}$ and $j\in I$,
the valuation at $x=j$ satisfies
$v_{x=j}(\prod_{i\in I}f_i)=1$. Hence
$\prod_{i\in I}f_i\notin k(x)^{\times2}$.
Thus the classes of $f_0,\ldots,f_g$ in
$k(x)^\times/k(x)^{\times2}$ are linearly independent, and
\begin{equation}\label{eq:start-explicit-geometric-degree}
 [k(x)(\sqrt{f_0},\ldots,\sqrt{f_g}):k(x)]=2^{g+1}.
\end{equation}
The $k[x]$-module $A_{0,x}\otimes_\Q k$ is free with basis
$\{\prod_{i=0}^g z_i^{e_i}:e_i\in\{0,1\}\}$.
Localization therefore gives an injective homomorphism
\begin{equation}\label{eq:start-geometric-ring-injection}
 A_{0,x}\otimes_\Q k\lhook\joinrel\longrightarrow
 (A_{0,x}\otimes_\Q k)\otimes_{k[x]}k(x)
       =k(x)(\sqrt{f_0},\ldots,\sqrt{f_g}).
\end{equation}
Consequently, $A_{0,x}\otimes_\Q k$ is a domain.
The $k[t]$-module $A_{0,t}\otimes_\Q k$ is free with basis
$\{\prod_{i=0}^g v_i^{e_i}:e_i\in\{0,1\}\}$.
The isomorphism $\beta_0$ in \eqref{eq:start-explicit-cover-overlap}
identifies the localization of $A_{0,t}\otimes_\Q k$ at
$k[t]\setminus\{0\}$ with the field
in \eqref{eq:start-explicit-geometric-degree}.
Thus $A_{0,t}\otimes_\Q k$ is also a domain. Moreover,
\begin{equation}\label{eq:start-cover-open-intersection}
 (U_{0,x})_k\cap(U_{0,t})_k
 =\Spec((A_{0,x}\otimes_\Q k)[x^{-1}])\ne\varnothing,
\end{equation}
because $x$ is a nonzero element of the domain
$A_{0,x}\otimes_\Q k$. This proves that $(C_0)_k$ is integral.

The equality $v_{x=0}(F)=1$ gives $F\notin k(x)^{\times2}$.
The free rank-two modules $A_{-1,x}\otimes_\Q k$ over $k[x]$
and $A_{-1,t}\otimes_\Q k$ over $k[t]$ consequently embed in
$k(x)(\sqrt F)$. Thus $A_{-1,x}\otimes_\Q k$ and
$A_{-1,t}\otimes_\Q k$ are domains. Since
\begin{equation}\label{eq:start-base-open-intersection}
 (U_{-1,x})_k\cap(U_{-1,t})_k
 =\Spec((A_{-1,x}\otimes_\Q k)[x^{-1}])\ne\varnothing,
\end{equation}
the curve $(C_{-1})_k$ is integral as well.

To prove smoothness, put
$P_i(x,z_0,\ldots,z_g)=z_i^2-f_i(x)$ for $0\le i\le g$.
At a geometric point $a\in U_{0,x}(k)$, at most one of the
coordinates $z_i(a)$ is zero, because the polynomials $f_i$ have
disjoint sets of roots. The partial derivatives are
\begin{equation}\label{eq:start-affine-jacobian-entries}
 \frac{\partial P_i}{\partial z_j}=2z_i\delta_{ij},\qquad
 \frac{\partial P_i}{\partial x}=-f_i'(x).
\end{equation}
If $z_i(a)\ne0$ for all $i$, the submatrix with columns
$z_0,\ldots,z_g$ has determinant $\prod_{i=0}^g2z_i(a)\ne0$.
If $z_j(a)=0$, then $f_j'(x(a))\ne0$, and the submatrix with
columns $x$ and $z_i$ for $i\ne j$ has determinant
$\pm f_j'(x(a))\prod_{i\ne j}2z_i(a)\ne0$.
The Jacobian criterion therefore proves that $U_{0,x}$ is smooth.
At a point of $U_{0,t}\setminus U_{0,x}$, one has $t=0$ and
$v_i^2=1$ for every $i$. For
$R_i(t,v_0,\ldots,v_g)=v_i^2-h_i(t)$,
\begin{equation}\label{eq:start-infinity-jacobian-determinant}
 \det\left(\frac{\partial R_i}{\partial v_j}\right)_{0\le i,j\le g}
       =\prod_{i=0}^g2v_i\ne0\qquad\text{at }t=0.
\end{equation}
It follows that $C_0$ is smooth over $\Q$.

For $P(x,y)=y^2-F(x)$, the equalities
$\partial P/\partial y=2y$ and $\partial P/\partial x=-F'(x)$
show that the two partial derivatives cannot vanish at a point of
$U_{-1,x}$, because $F$ has no repeated root.
For $R(t,w)=w^2-\prod_{i=0}^g h_i(t)$, the equations $t=0$ and
$R(t,w)=0$ give $w=1$ or $w=-1$, and hence
$(\partial R/\partial w)(0,w)=2w\ne0$.
The Jacobian criterion proves that $C_{-1}$ is smooth over $\Q$.

By the definitions of $\rho_{-1}$ and $\rho_0$,
\begin{equation}\label{eq:start-inverseimages-affine-opens}
 \begin{array}{ll}
 \rho_{-1}^{-1}(B_x)=\Spec A_{-1,x},&
 \rho_{-1}^{-1}(B_t)=\Spec A_{-1,t},\\
 \rho_0^{-1}(B_x)=\Spec A_{0,x},&
 \rho_0^{-1}(B_t)=\Spec A_{0,t}.
 \end{array}
\end{equation}
The rings $A_{-1,x}$ and $A_{-1,t}$ are free of rank two over
$\Q[x]$ and $\Q[t]$, respectively. The rings $A_{0,x}$ and
$A_{0,t}$ are free of rank $2^{g+1}$ over $\Q[x]$ and $\Q[t]$,
respectively. Thus $\rho_{-1}$ and $\rho_0$ are finite morphisms,
and $C_{-1}$ and $C_0$ are projective over $\Q$.
Finally, the field degrees give
\begin{equation}\label{eq:start-rho-field-degrees}
 \begin{aligned}
 \deg\rho_{-1}&=[\Q(x)(\sqrt F):\Q(x)]=2,\\
 \deg\rho_0&=[\Q(x)(\sqrt{f_0},\ldots,\sqrt{f_g}):\Q(x)]
            =2^{g+1},
 \end{aligned}
\end{equation}
as desired.
\end{proof}

For $\boldsymbol\epsilon=(\epsilon_0,\ldots,\epsilon_g)
\in\{1,-1\}^{g+1}$, define automorphisms of $A_{0,x}$ and
$A_{0,t}$ by
\begin{equation}\label{eq:start-automorphism-generator-images}
 \begin{aligned}
 \sigma_{\boldsymbol\epsilon,x}^*(x)&=x,&
 \sigma_{\boldsymbol\epsilon,x}^*(z_i)&=\epsilon_i z_i,\\
 \sigma_{\boldsymbol\epsilon,t}^*(t)&=t,&
 \sigma_{\boldsymbol\epsilon,t}^*(v_i)&=\epsilon_i v_i.
 \end{aligned}
\end{equation}
After localization, the isomorphism $\beta_0$ satisfies
\begin{equation}\label{eq:start-automorphism-gluing-compatibility}
 \beta_0\circ\sigma_{\boldsymbol\epsilon,t}^*[t^{-1}]
 =\sigma_{\boldsymbol\epsilon,x}^*[x^{-1}]\circ\beta_0.
\end{equation}
Consequently, $\Spec\sigma_{\boldsymbol\epsilon,x}^*$ and
$\Spec\sigma_{\boldsymbol\epsilon,t}^*$ define an automorphism
$\sigma_{\boldsymbol\epsilon}:C_0\to C_0$
with $\rho_0\circ\sigma_{\boldsymbol\epsilon}=\rho_0$.
Every $\Q(x)$-automorphism of $\Q(C_0)$ sends $z_i$ to one of the
two roots $z_i,-z_i$ of $T^2-f_i(x)$. Hence
\begin{equation}\label{eq:start-rho-automorphism-group}
 \operatorname{Aut}_{\mathbb P^1_\Q}(C_0)
       =\{\sigma_{\boldsymbol\epsilon}:
                       \boldsymbol\epsilon\in\{1,-1\}^{g+1}\}.
\end{equation}
The identity $\pi_0\circ\sigma_{\boldsymbol\epsilon}=\pi_0$
holds if and only if
$\sigma_{\boldsymbol\epsilon}^*(\prod_i z_i)=\prod_i z_i$,
or equivalently $\prod_i\epsilon_i=1$.
We therefore put
\begin{equation}\label{eq:start-explicit-sign-group}
 G_0=\left\{(\epsilon_0,\ldots,\epsilon_g)\in\{1,-1\}^{g+1}:
              \prod_{i=0}^{g}\epsilon_i=1\right\}.
\end{equation}
The isomorphism $\mathbb F_2^g\to G_0$ is
\begin{equation}\label{eq:start-sign-group-isomorphism}
 (a_0,\ldots,a_{g-1})\longmapsto
 ((-1)^{a_0},\ldots,(-1)^{a_{g-1}},(-1)^{\sum_{i=0}^{g-1}a_i}).
\end{equation}

Put $K=\Q(C_{-1})=\Q(x,y)$ and define
\begin{equation}\label{eq:start-explicit-function-field-L}
 L=\Q(C_0)=\mathbb Q(x)(z_0,\ldots,z_g),\qquad z_i^2=f_i(x).
\end{equation}

\begin{proposition}\label{prop:start-galois}
The morphism $\pi_0:C_0\to C_{-1}$ is a finite Galois cover of degree
$2^g$ with group $G_0\simeq\mathbb F_2^g$. The action of
$\boldsymbol\epsilon\in G_0$ on $C_0$ is the automorphism
$\sigma_{\boldsymbol\epsilon}$ defined above.
\end{proposition}

\begin{proof}
The extension $L/\Q(x)$ is Galois of degree $2^{g+1}$, with
Galois group $\{\sigma_{\boldsymbol\epsilon}^*:
\boldsymbol\epsilon\in\{1,-1\}^{g+1}\}$.
The subgroup fixing $y=\prod_i z_i$ is $G_0$, and
\begin{equation}\label{eq:start-fixed-field-identities}
 L^{G_0}=K,\qquad
 [L:K]=2^g,\qquad \operatorname{Gal}(L/K)=G_0.
\end{equation}
The inverse images of the affine opens of $C_{-1}$ are
$\pi_0^{-1}(U_{-1,x})=U_{0,x}$ and
$\pi_0^{-1}(U_{-1,t})=U_{0,t}$.
The $A_{-1,x}$-algebra $A_{0,x}$ is generated by $z_0,\ldots,z_g$,
which satisfy the monic equations $z_i^2-f_i(x)=0$.
The $A_{-1,t}$-algebra $A_{0,t}$ is generated by $v_0,\ldots,v_g$,
which satisfy the monic equations $v_i^2-h_i(t)=0$.
Thus $A_{0,x}$ is finite over $A_{-1,x}$ and $A_{0,t}$ is finite
over $A_{-1,t}$. Consequently, $\pi_0$ is finite, and
\begin{equation}\label{eq:start-pi-degree-group}
 \deg\pi_0=[L:K]=2^g,\qquad
 \operatorname{Aut}_{C_{-1}}(C_0)=G_0.
\end{equation}
\end{proof}

For $0\leq a\leq 2g+1$, let $W_a\in U_{-1,x}(\mathbb Q)$ be the
point with $x(W_a)=a$ and $y(W_a)=0$.
Let $P_+,P_-\in U_{-1,t}(\Q)$ be defined by
\begin{equation}\label{eq:start-base-infinity-coordinates}
 t(P_+)=t(P_-)=0,\qquad w(P_+)=1,\qquad w(P_-)=-1.
\end{equation}

\begin{proposition}\label{prop:start-etale}
The morphism $\pi_0:C_0\to C_{-1}$ is \'{e}tale.
\end{proposition}

\begin{proof}
For $0\le i\le g$, the divisor of the rational function $f_i$ on
$C_{-1}$ is
\begin{equation}\label{eq:start-explicit-divisors}
 \operatorname{div}(f_i)=2W_i+2W_{i+\ell}-2P_+-2P_-.
\end{equation}
Let $P$ be a closed point of $C_{-1}$, put
$A=\mathcal O_{C_{-1},P}$, and choose a uniformizer $\varpi$ of
$A$. For $0\le i\le g$, equation
\eqref{eq:start-explicit-divisors} gives integers $n_i$ with
$v_P(f_i)=2n_i$. Define
\begin{equation}\label{eq:start-local-etale-algebras}
 q_i=\varpi^{-2n_i}f_i\in A^\times,\qquad
 B_i=A[T_i]/(T_i^2-q_i).
\end{equation}
The element $T_i$ is a unit of $B_i$, since $T_i^2=q_i\in A^\times$.
Also $2\in A^\times$, so
$\partial(T_i^2-q_i)/\partial T_i=2T_i\in B_i^\times$.
It follows that $B_i$ is finite \'etale over $A$; see \st{02GH}.
The tensor product $B=\bigotimes_{i=0}^g B_i$ is consequently
finite \'etale over $A$.
Define an $A$-algebra homomorphism
\begin{equation}\label{eq:start-local-algebra-map}
 \varphi:B\longrightarrow L,\qquad
 T_i\longmapsto\varpi^{-n_i}z_i,
\end{equation}
and put $B_0=\varphi(B)$. A finite \'etale algebra over the normal
ring $A$ is a product of normal domains. The kernel of $\varphi$
is a prime ideal whose intersection with $A$ is zero, since
$\varphi|_A$ is the inclusion $A\subset L$.
Because $B$ is integral over $A$, this kernel is a minimal prime
of $B$. Thus $B_0=B/\ker\varphi$ is one of the normal domain
factors of $B$, and is finite \'etale over $A$.
Furthermore,
$\operatorname{Frac}(B_0)=K(z_0,\ldots,z_g)=L$.
Thus $B_0$ is the integral closure of $A$ in $L$.
The curve $C_0$ is normal and finite over $C_{-1}$, with function
field $L$. Hence $C_0$ is the normalization of $C_{-1}$ in $L$,
and
\begin{equation}\label{eq:start-local-normalization-identity}
 C_0\times_{C_{-1}}\Spec A=\Spec B_0.
\end{equation}
Since $B_0$ is finite \'etale over $A$, $\pi_0$ is \'etale at every
point above $P$. This holds for every closed point $P$ of $C_{-1}$,
and $L/K$ is separable, so $\pi_0$ is \'etale everywhere.
\end{proof}

\begin{proposition}\label{prop:start-genera}
The curves $C_{-1}$ and $C_0$ satisfy
\begin{equation}\label{eq:start-explicit-genera}
 g(C_{-1})=g,\qquad g(C_0)=1+2^g(g-1).
\end{equation}
\end{proposition}

\begin{proof}
The degree-two morphism $\rho_{-1}:C_{-1}\to\mathbb P^1_\Q$ is
ramified precisely at $W_0,\ldots,W_{2g+1}$, with ramification
index two at each $W_a$. At $P_+$ and $P_-$, the nonvanishing
partial derivative
$\partial(w^2-\prod_i h_i(t))/\partial w=2w$
shows that $\rho_{-1}$ is unramified.
The Riemann--Hurwitz formula for $\rho_{-1}$
(\cite[Section~7.4.2]{Liu}, \st{0C1B}) gives
\begin{equation}\label{eq:start-rho-riemann-hurwitz}
 2g(C_{-1})-2=-4+(2g+2)=2g-2.
\end{equation}
Applying the Riemann--Hurwitz formula to the \'etale morphism
$\pi_0:C_0\to C_{-1}$ gives
\begin{equation}\label{eq:start-pi-riemann-hurwitz}
 2g(C_0)-2=2^g(2g(C_{-1})-2)=2^g(2g-2).
\end{equation}
Thus $g(C_{-1})=g$ and $g(C_0)=1+2^g(g-1)$.
\end{proof}

Define $Q_+\in U_{0,t}(\mathbb Q)$ to be the point with coordinates
\begin{equation}\label{eq:start-explicit-rational-point}
 t(Q_+)=0,\qquad v_i(Q_+)=1\quad(0\le i\le g).
\end{equation}
The equations $v_i(Q_+)^2=h_i(t(Q_+))=1$ verify that $Q_+$
belongs to $U_{0,t}(\mathbb Q)$.
Equation \eqref{eq:start-explicit-ring-maps} gives
$\pi_0(Q_+)=P_+$.

\begin{proposition}\label{prop:start-infinity}
Write $P_1=P_+$ and $P_{-1}=P_-$. For $\delta\in\{1,-1\}$, define
\begin{equation}\label{eq:start-infinity-sign-index-set}
 E_\delta=\left\{\boldsymbol\epsilon\in\{1,-1\}^{g+1}:
                                  \prod_{i=0}^g\epsilon_i=\delta\right\}.
\end{equation}
For $\boldsymbol\epsilon\in E_\delta$, let
$Q_{\boldsymbol\epsilon}\in U_{0,t}(\Q)$ be the point defined by
\begin{equation}\label{eq:start-explicit-infinity-points}
 t(Q_{\boldsymbol\epsilon})=0,\qquad
 v_i(Q_{\boldsymbol\epsilon})=\epsilon_i\quad(0\le i\le g).
\end{equation}
Then
\begin{equation}\label{eq:start-infinity-split-fibre}
 C_0\times_{C_{-1},P_\delta}\Spec\Q
       =\coprod_{\boldsymbol\epsilon\in E_\delta}
                                      \{Q_{\boldsymbol\epsilon}\}
       \simeq\coprod_{a=1}^{2^g}\Spec\Q.
\end{equation}
The action of $G_0$ on
$\{Q_{\boldsymbol\epsilon}:\boldsymbol\epsilon\in E_\delta\}$
is simply transitive.
\end{proposition}

\begin{proof}
The fibre $C_0\times_{C_{-1},P_\delta}\Spec\Q$ has coordinate ring
\begin{equation}\label{eq:start-explicit-infinity-fibre-ring}
 \Q[v_0,\ldots,v_g]/(v_i^2-1\ (0\leq i\leq g),\,
                              \prod_i v_i-\delta).
\end{equation}
Evaluation at the tuples $\boldsymbol\epsilon\in E_\delta$ gives
an isomorphism from the ring in
\eqref{eq:start-explicit-infinity-fibre-ring} to
$\prod_{\boldsymbol\epsilon\in E_\delta}\Q$, because
$T^2-1=(T-1)(T+1)$ has distinct roots in $\Q$.
For $\boldsymbol\eta\in G_0$ and
$\boldsymbol\epsilon\in E_\delta$, the automorphism
$\sigma_{\boldsymbol\eta}$ satisfies
\begin{equation}\label{eq:start-infinity-action-formula}
 \sigma_{\boldsymbol\eta}(Q_{\boldsymbol\epsilon})
       =Q_{(\eta_i\epsilon_i)_{i=0}^g}.
\end{equation}
For $\boldsymbol\epsilon,\boldsymbol\epsilon'\in E_\delta$,
there is exactly one $\boldsymbol\eta\in G_0$ with
$\eta_i\epsilon_i=\epsilon_i'$ for all $i$, namely
$\eta_i=\epsilon_i'\epsilon_i$. Thus the action of $G_0$ on
$\{Q_{\boldsymbol\epsilon}:\boldsymbol\epsilon\in E_\delta\}$
is simply transitive.
\end{proof}

Define rational functions on $C_{-1}$ by
\begin{equation}\label{eq:start-kummer-functions}
 u_i=\frac{f_i(x)}{(x-g)^2}
     =\frac{(x-i)(x-i-\ell)}{(x-g)^2}\quad(0\leq i<g).
\end{equation}
Choose the square root $\sqrt{u_i}=z_i/(x-g)$ in $L$.
Then
\begin{equation}\label{eq:start-explicit-kummer-field}
 L=K(\sqrt{u_0},\ldots,\sqrt{u_{g-1}}),\qquad
 \sqrt{u_i}=\frac{z_i}{x-g}.
\end{equation}
Indeed, the chosen elements $\sqrt{u_i}$ satisfy
\begin{equation}\label{eq:start-explicit-kummer-generators}
 z_i=(x-g)\sqrt{u_i}\quad(0\leq i<g),\qquad
 z_g=\frac{y}{\prod_{i=0}^{g-1}z_i}.
\end{equation}
Equations \eqref{eq:start-explicit-cover-overlap} and
\eqref{eq:start-explicit-kummer-generators} give
\begin{equation}\label{eq:start-kummer-infinity-value}
 \sqrt{u_i}=\frac{v_i}{1-gt}\in A_{0,t}[(1-gt)^{-1}],
 \qquad \sqrt{u_i}(Q_+)=1\quad(0\le i<g).
\end{equation}
The evaluation at $Q_+$ is defined because $1-gt(Q_+)=1$.

\subsection{The arithmetic models}

Put
\begin{equation}\label{eq:start-explicit-bad-primes}
 N=2g+2=2\ell,\qquad R=N-1,\qquad
 \Sigma=\{p\text{ prime}:p\leq R\},\qquad
 U=\Spec\Z\bigl[1/\prod_{p\in\Sigma}p\bigr].
\end{equation}
We will construct a regular surface $P_m$ by blowing up closed points
of $\mathbb P^1_\Z$ over primes in $\Sigma$. We will then normalize
$P_m$ in $K$, resolve the singularities of this normalization over odd
primes to obtain $S_{-1}$, and define $S_0$ by normalizing
$S_{-1}$ in $L$.
The definitions of normalization, blow-up and strict transform are in
Section~\ref{sec:dictionary}; see also
\cite[Sections~4.1.2 and~8.1.1]{Liu}, \st{0BAK} and \st{01OF}.

For $0\leq a<N$, let
$\sigma_a:\Spec\Z\to\mathbb P^1_\Z$ be the section defined by
$x-a=0$ in $\Spec\Z[x]$. For each surface $T$ constructed below,
let $D_{a,T}\subset T$ denote the strict transform of
$\sigma_a(\Spec\Z)$ under $T\to\mathbb P^1_\Z$.
We will use the following observation at each step. Suppose
$b:T'\to T$ is the blow-up of a reduced union $Z$ of closed points
of $T_{\mathbb F_p}$, and let $\mathcal I_Z\subset\mathcal O_T$
be the ideal sheaf of $Z$. If $D_{a,T}$ is the image of a section
$\sigma_{a,T}:\Spec\Z\to T$, then
\begin{equation}\label{eq:start-model-section-ideal}
 \sigma_{a,T}^{-1}\mathcal I_Z\cdot\mathcal O_{\Spec\Z}
 =\begin{cases}
   p\mathcal O_{\Spec\Z},&\sigma_{a,T}(p)\in Z,\\
   \mathcal O_{\Spec\Z},&\sigma_{a,T}(p)\notin Z.
  \end{cases}
\end{equation}
In the first case the equality follows locally because the ideal of
$Z$ contains $p$, and every element of that ideal vanishes at
$\sigma_{a,T}(p)$. Both ideals on the right of
\eqref{eq:start-model-section-ideal} are invertible. The universal
property of the blow-up therefore gives a unique section
$\sigma_{a,T'}:\Spec\Z\to T'$ satisfying
$b\circ\sigma_{a,T'}=\sigma_{a,T}$. Its image is $D_{a,T'}$, so
$D_{a,T'}\simeq\Spec\Z$.

We now specify all the blow-ups used to define $P_m$.
\begin{enumerate}[label=(\roman*)]
\item Write the odd primes in $\Sigma$ as
$p_1<\cdots<p_s$, and set $T^{(0)}=\mathbb P^1_\Z$.
Suppose $T^{(r-1)}$ has been constructed, where $1\leq r\leq s$,
and put $p=p_r$ and $T_{p,0}=T^{(r-1)}$.
For $k\geq0$, define the finite reduced closed subscheme
\begin{equation}\label{eq:start-model-odd-centres}
 Z_{p,k}=
 \left(\bigcup_{0\leq a<b<N}
       (D_{a,T_{p,k}}\cap D_{b,T_{p,k}}\cap(T_{p,k})_{\mathbb F_p})
 \right)_{\mathrm{red}}.
\end{equation}
If $Z_{p,k}\ne\varnothing$, let $\mathcal I_{p,k}$ be its ideal
sheaf on $T_{p,k}$ and define
\begin{equation}\label{eq:start-model-odd-blowup}
 b_{p,k}:T_{p,k+1}=\operatorname{Bl}_{\mathcal I_{p,k}}(T_{p,k})
       \longrightarrow T_{p,k}.
\end{equation}
If $Z_{p,k}=\varnothing$, stop the recursion for $p$, put $k_p=k$,
and define $T^{(r)}=T_{p,k_p}$.

Here are affine opens proving that $k_p$ exists. Set
\begin{equation}\label{eq:start-model-odd-chart-indices}
 \begin{aligned}
 C_{p,0}&=\{0\},\\
 C_{p,k}&=\{c:0\leq c<p^k,\
       |\{a:0\leq a<N,\ a\equiv c\pmod{p^k}\}|\geq2\}
       \quad(k\geq1).
 \end{aligned}
\end{equation}
For each $c\in C_{p,k}$, the surface
$T_{p,k}\times_{\Spec\Z}\Spec\Z_{(p)}$ has an affine open
\begin{equation}\label{eq:start-model-odd-affine-open}
 U_{p,k,c}=\Spec\Z_{(p)}[X_{p,k,c}],\qquad
 X_{p,k,c}=\frac{x-c}{p^k}.
\end{equation}
For $k=0$, take $U_{p,0,0}=\Spec\Z_{(p)}[x]$.
For $k\geq0$, assume that $U_{p,k,c}$ has been constructed.
We construct the opens at stage $k+1$ as follows.
In $U_{p,k,c}$, define $z_{p,k,c,e}$ to be the closed point with ideal
$(p,X_{p,k,c}-e)$, where $0\leq e<p$.
The section with index $a$ meets $z_{p,k,c,e}$ precisely when
$a\equiv c+p^ke\pmod{p^{k+1}}$.
If $c+p^ke\in C_{p,k+1}$, then $z_{p,k,c,e}\in Z_{p,k}$.
The affine open of the blow-up of $U_{p,k,c}$ at $z_{p,k,c,e}$
corresponding to the generator $p$ has coordinate ring
\begin{equation}\label{eq:start-explicit-odd-coordinate}
 \Z_{(p)}[X_{p,k,c}]
    \left[\frac{X_{p,k,c}-e}{p}\right]
     =\Z_{(p)}[X_{p,k+1,c+p^ke}],\qquad
 X_{p,k+1,c+p^ke}=\frac{x-(c+p^ke)}{p^{k+1}}.
\end{equation}
This defines $U_{p,k+1,c+p^ke}$. Under the morphism to $T_{p,k}$,
the fibre $(U_{p,k+1,c+p^ke})_{\mathbb F_p}$ maps to
$z_{p,k,c,e}$. Thus $U_{p,k+1,c+p^ke}$ has empty inverse image
of $Z_{p,k}\setminus\{z_{p,k,c,e}\}$. The blow-up along the
inverse image of $Z_{p,k}\setminus\{z_{p,k,c,e}\}$ therefore
restricts to the identity morphism of $U_{p,k+1,c+p^ke}$.
Consequently $U_{p,k+1,c+p^ke}$ is an affine open of
$T_{p,k+1}\times_{\Spec\Z}\Spec\Z_{(p)}$.
For $a\equiv c+p^ke\pmod{p^{k+1}}$, the restriction of
$\sigma_{a,T_{p,k+1}}$ to $\Spec\Z_{(p)}$ factors through
$U_{p,k+1,c+p^ke}$ because $(a-c-p^ke)/p^{k+1}$ is an integer.
These identities prove \eqref{eq:start-model-odd-affine-open}
by induction. They also show that each point of $Z_{p,k}$ belongs
to the smooth locus of $(T_{p,k})_{\mathbb F_p}$.
Choose an integer $h$ with $p^h>R$. Distinct indices
$a,b\in\{0,\ldots,N-1\}$ satisfy $a\not\equiv b\pmod{p^h}$.
Consequently $k_p\leq h-1$, and the definition of $Z_{p,k_p}$ gives
\begin{equation}\label{eq:start-model-separated-sections}
 D_{a,T^{(r)}}\cap D_{b,T^{(r)}}\cap(T^{(r)})_{\mathbb F_p}
   =\varnothing\qquad(0\leq a<b<N).
\end{equation}

\item Set $T_{2,0}=T^{(s)}$. The morphism
$T^{(s)}\to\mathbb P^1_\Z$ is an isomorphism after base change
to $\Spec\Z_{(2)}$, since the centres in
\eqref{eq:start-model-odd-centres} lie over odd primes. Thus
$(T_{2,0})_{\mathbb F_2}=\mathbb P^1_{\mathbb F_2}$; denote this
component by $E_{0,0}$ and put $X_{0,0}=x$.
For $k=0,\ldots,m-2$, the fibre $(T_{2,k})_{\mathbb F_2}$ has
components $E_{k,c}$ indexed by $0\leq c<2^k$. The base change
$T_{2,k}\times_{\Spec\Z}\Spec\Z_{(2)}$ has affine opens
\begin{equation}\label{eq:start-explicit-dyadic-coordinate}
 U_{2,k,c}=\Spec\Z_{(2)}[X_{k,c}],\qquad
 X_{k,c}=\frac{x-c}{2^k},\qquad
 E_{k,c}\cap U_{2,k,c}=V(2).
\end{equation}
For $k=0$, this is $U_{2,0,0}=\Spec\Z_{(2)}[x]$; the opens
for $k\geq1$ are defined inductively below.
For $0\leq k\leq m-3$, define $z_{k,c,e}\in U_{2,k,c}$ to be
the closed point with ideal $(2,X_{k,c}-e)$ in
$\Z_{(2)}[X_{k,c}]$, for $e\in\{0,1\}$, and put
\begin{equation}\label{eq:start-explicit-dyadic-centres}
 Z_{2,k}=\coprod_{0\leq c<2^k,\ e\in\{0,1\}} z_{k,c,e},
 \qquad \mathcal I_{2,k}=\mathcal I_{Z_{2,k}}.
\end{equation}
The point $z_{k,c,e}$ belongs to the smooth locus of
$(T_{2,k})_{\mathbb F_2}$. Define
\begin{equation}\label{eq:start-model-dyadic-blowup}
 b_{2,k}:T_{2,k+1}=\operatorname{Bl}_{\mathcal I_{2,k}}(T_{2,k})
        \longrightarrow T_{2,k}.
\end{equation}
Let $E_{k+1,c+2^ke}$ be the exceptional component over $z_{k,c,e}$.
Define $U_{2,k+1,c+2^ke}$ to be the affine open of
$T_{2,k+1}\times_{\Spec\Z}\Spec\Z_{(2)}$ corresponding to
the generator $2$ of $(2,X_{k,c}-e)$. Its coordinate ring is
\begin{equation}\label{eq:start-explicit-dyadic-next-coordinate}
 \Z_{(2)}[X_{k,c}]\left[\frac{X_{k,c}-e}{2}\right]
 =\Z_{(2)}[X_{k+1,c+2^ke}],\qquad
 X_{k+1,c+2^ke}=\frac{x-(c+2^ke)}{2^{k+1}}.
\end{equation}
The section $\sigma_{a,T_{2,k}}$ passes through $z_{k,c,e}$ if and
only if $a\equiv c+2^ke\pmod{2^{k+1}}$.
For these indices $a$, the restriction of $\sigma_{a,T_{2,k+1}}$
to $\Spec\Z_{(2)}$ factors through $U_{2,k+1,c+2^ke}$, with
$X_{k+1,c+2^ke}=(a-c-2^ke)/2^{k+1}\in\Z_{(2)}$.
This proves \eqref{eq:start-explicit-dyadic-coordinate} at stage $k+1$.
We perform exactly the morphisms $b_{2,0},\ldots,b_{2,m-3}$, and
set $P_m=T_{2,m-2}$. For $0\leq j\leq m-2$, the notation
$E_{j,c}\subset P_m$ means the strict transform on $P_m$ of the
component created on $T_{2,j}$; we call $j$ the depth of $E_{j,c}$.
No centre of a morphism in
\eqref{eq:start-model-dyadic-blowup} belongs to a component of depth
$m-2$.
\end{enumerate}
For each morphism $b_{p,k}:T_{p,k+1}\to T_{p,k}$ in
\eqref{eq:start-model-odd-blowup} or
\eqref{eq:start-model-dyadic-blowup}, the centre is contained in
$(T_{p,k})_{\mathbb F_p}$. Hence the restriction of $b_{p,k}$ is
an isomorphism
\begin{equation}\label{eq:start-model-disjoint-prime-restriction}
 T_{p,k+1}\times_{\Spec\Z}\Spec\Z[1/p]
   \xrightarrow{\ \sim\ }
 T_{p,k}\times_{\Spec\Z}\Spec\Z[1/p].
\end{equation}
In particular, if $q\ne p$, base change of $b_{p,k}$ to
$\Spec\Z_{(q)}$ is an isomorphism. Thus the calculations over
$\Z_{(p)}$ in steps (i) and (ii) also describe
$P_m\times_{\Spec\Z}\Spec\Z_{(p)}$.
Write $\widetilde\sigma_a=\sigma_{a,P_m}$ and $D_a=D_{a,P_m}$.

Every geometric fibre of $P_m\to\Spec\Z$ is a reduced union of
projective lines. Define
\begin{equation}\label{eq:start-explicit-fibre-graph}
 \begin{aligned}
 \Gamma_p&=(V_p,\mathcal E_p,\partial_p),\\
 V_p&=\{E:E\text{ is an irreducible component of }
                   (P_m)_{\overline{\mathbb F}_p}\},\\
 \mathcal E_p&=\{z:z\text{ is an intersection point of two distinct}
                   \ E,E'\in V_p\},\\
 \partial_p(z)&=\{E,E'\}\quad\text{if }z\in E\cap E'.
 \end{aligned}
\end{equation}
These definitions agree with
\eqref{eq:bg-explicit-incidence-graph}. To justify the assertions
about the fibres and the definition of $\partial_p$, consider a
single centre $z$ of a morphism $b_{p,k}$.
The point $z$ is a smooth point of $(T_{p,k})_{\mathbb F_p}$ and
lies on one component $E$ of that fibre. If $u$ is a local parameter
along $E$ at $z$, the maximal ideal at $z$ is $(p,u)$.
The affine opens of the blow-up of $(p,u)$ have coordinates $u/p$
and $p/u$; on the latter open the equation for the fibre is
$u(p/u)=0$. It follows that the exceptional divisor
$E'_z\simeq\mathbb P^1_{\mathbb F_p}$ meets the strict transform
of $E$ at one point, transversely, and meets no other fibre
component. The fibre of the blow-up is reduced.
The fibre of $\mathbb P^1_\Z$ over $p$ has one component. Induction
on the morphisms $b_{p,k}$ proves that every geometric fibre of
$P_m$ is reduced, has projective lines as components, and has
at most two components through any point, meeting transversely.
For each stage $k$, let
$\Gamma((T_{p,k})_{\overline{\mathbb F}_p})
=(V_{p,k},\mathcal E_{p,k},\partial_{p,k})$.
For $z\in Z_{p,k}$, denote the component containing $z$ by $E_z$
and the intersection of $E'_z$ with the strict transform of $E_z$
by $e_z$. The change in the graph under $b_{p,k}$ is
\begin{equation}\label{eq:start-explicit-graph-blowup}
 \begin{aligned}
 V_{p,k+1}&=V_{p,k}\amalg\{E'_z:z\in Z_{p,k}\},\\
 \mathcal E_{p,k+1}&=\mathcal E_{p,k}\amalg\{e_z:z\in Z_{p,k}\},\\
 \partial_{p,k+1}(e_z)&=\{E_z,E'_z\}\quad(z\in Z_{p,k}),\\
 \partial_{p,k+1}(e)&=\partial_{p,k}(e)\quad(e\in\mathcal E_{p,k}).
 \end{aligned}
\end{equation}
In these equalities each component in $V_{p,k}$ is identified with
its strict transform on $T_{p,k+1}$.
Thus $\Gamma_p$ is connected and
$|\mathcal E_p|=|V_p|-1$, so $\Gamma_p$ is a tree.

The coordinates in \eqref{eq:start-explicit-odd-coordinate} and
\eqref{eq:start-explicit-dyadic-next-coordinate} show that each
$D_a$ meets $(P_m)_{\mathbb F_p}$ at a smooth point of the fibre
and transversely. For odd $p$, equation
\eqref{eq:start-model-separated-sections} also makes these points
distinct for different $a$. Consequently
$\sum_aD_a+(P_m)_{\mathbb F_p}$ is a divisor with normal crossings
near $(P_m)_{\mathbb F_p}$ when $p$ is odd, in the sense of
Definition~\ref{def:bg-normal-crossings}.
Every centre of every $b_{p,k}$ is a regular closed point of
$T_{p,k}$. Blowing up a regular surface at a regular closed point
preserves regularity. The morphisms $b_{p,k}$ are projective, so
their composition $P_m\to\mathbb P^1_\Z$ is projective.
It follows that $P_m$ is a regular projective model of
$\mathbb P^1_\Q$.
For a prime $p$, put
\begin{equation}\label{eq:start-explicit-local-base}
 P_p=P_m\times_{\Spec\Z}\Spec\Z_{(p)}.
\end{equation}
The notation $(P_p)_{\overline{\mathbb F}_p}$ denotes the geometric
special fibre of $P_p\to\Spec\Z_{(p)}$.
Below, when we write a strictly henselian local ring of $P_p$ at a
geometric point $\bar z\to P_p$, we mean
$\mathcal O^{\mathrm{sh}}_{P_p,\bar z}$.

Let $\nu_{-1}:X_{-1}\to P_m$ be the normalization of $P_m$ in
\begin{equation}\label{eq:start-explicit-function-field-K}
 K=\Q(x,y),\qquad y^2=F(x).
\end{equation}
Let $Z_{-1}\subset X_{-1}$ be the reduced closed subscheme whose
points are the singular closed points of $X_{-1}$ over odd primes,
and let $\mathcal J_{-1}$ be its ideal sheaf.
The local equations in the proof of
Proposition~\ref{prop:start-models} below show that $Z_{-1}$ is a
finite set. Define
\begin{equation}\label{eq:start-model-resolution-morphism}
 b_{-1}:S_{-1}=\operatorname{Bl}_{\mathcal J_{-1}}(X_{-1})
             \longrightarrow X_{-1}.
\end{equation}
This is one blow-up along the union $Z_{-1}$: over each
$z\in Z_{-1}$ the morphism $b_{-1}$ is the blow-up of the
maximal ideal of $\mathcal O_{X_{-1},z}$.
Since $Z_{-1}$ has no point over $2$, the restriction
\begin{equation}\label{eq:start-model-resolution-at-two}
 S_{-1}\times_{\Spec\Z}\Spec\Z_{(2)}
  \xrightarrow{\ b_{-1}\ \sim\ }
 X_{-1}\times_{\Spec\Z}\Spec\Z_{(2)}
\end{equation}
is an isomorphism. Finally, define
\begin{equation}\label{eq:start-surface-map}
 S_0=\operatorname{Norm}_{L}(S_{-1}),\qquad
 \Pi_0:S_0\longrightarrow S_{-1}.
\end{equation}
For every affine open $\Spec A\subset S_{-1}$, equation
\eqref{eq:start-surface-map} means
\begin{equation}\label{eq:start-explicit-integral-closure}
 A\longrightarrow\overline A^{\,L},
 \qquad
 \Pi_0^{-1}(\Spec A)=\Spec\overline A^{\,L},
\end{equation}
where $\overline A^{\,L}=\{u\in L:u\text{ is integral over }A\}$.
If $\Spec A$ maps to $\Spec\Z[x]\subset\mathbb P^1_\Z$, then
$x,y\in A$. In this case $\overline A^{\,L}$ is also the integral
closure in $L$ of $A[z_0,\ldots,z_g]$, where
$z_i^2=f_i(x)$ and $\prod_i z_i=y$: the elements $z_i$ are
integral over $A$, so transitivity of integrality gives this equality
of integral closures. If $\Spec A$ maps to $\Spec\Z[t]$, the
corresponding subalgebra is $A[v_0,\ldots,v_g]\subset L$, with
$v_i^2=h_i(t)$ and $\prod_i v_i=w$.
On the inverse image of $\Spec\Z[x,x^{-1}]$, the identities
$t=x^{-1}$ and $v_i=x^{-1}z_i$ give
$A[x^{-1}][z_0,\ldots,z_g]=A[x^{-1}][v_0,\ldots,v_g]$.
These are the substitutions in
\eqref{eq:start-explicit-base-overlap} and
\eqref{eq:start-explicit-cover-overlap}.
The generic fibre of $\Pi_0$ is $\pi_0:C_0\to C_{-1}$.

The normalization morphisms $\nu_{-1}:X_{-1}\to P_m$ and
$\Pi_0:S_0\to S_{-1}$ are finite: schemes of finite type over
$\Z$ are excellent, and their normalizations in finite extensions
of their function fields are finite; see
\cite[Section~8.2.3]{Liu}, \st{07QS} and \st{0AVK}.

\begin{proposition}\label{prop:start-models}
The schemes $S_{-1}$ and $S_0$ are normal integral projective flat surfaces
over $\Spec\Z$, with pure one-dimensional fibres and respective generic
fibres $C_{-1}$ and $C_0$. The surface $S_{-1}$ is regular away from its
fibre at $2$.
\end{proposition}
\begin{proof}
Each morphism $b_{p,k}$ in
\eqref{eq:start-model-odd-blowup} or
\eqref{eq:start-model-dyadic-blowup}, and the morphism
$b_{-1}$ in \eqref{eq:start-model-resolution-morphism}, is
projective because it is a blow-up. The finite morphisms
$\nu_{-1}:X_{-1}\to P_m$ and $\Pi_0:S_0\to S_{-1}$ are
projective by \st{0B3I}. Composing these morphisms with
$P_m\to\mathbb P^1_\Z\to\Spec\Z$ proves that the structure
morphisms $S_{-1}\to\Spec\Z$ and $S_0\to\Spec\Z$ are
projective; see \st{01WE} and \st{0B44}.
By the definition of normalization, $X_{-1}$ is normal and integral
with function field $K$, and $S_0$ is normal and integral with
function field $L$. The blow-up morphism
$b_{-1}:S_{-1}\to X_{-1}$ is birational, and $S_{-1}$ is
integral with function field $K$.

Fix an odd prime $p$. The support of
$\sum_a(D_a\cap P_p)+(P_p)_{\mathbb F_p}$ is a divisor with
normal crossings. Let $\bar z$ be a geometric closed point of
$(P_p)_{\mathbb F_p}$, and set
$A=\mathcal O^{\mathrm{sh}}_{P_p,\bar z}$.
Choose regular parameters $u,v$ in $A$ so that every component of
this divisor through $\bar z$ has equation $u=0$ or $v=0$.
Every unit of $A$ is a square in $A$: apply Hensel's lemma to
$T^2-a$ for $a\in A^\times$, using that $2\in A^\times$ and
that the residue field of $A$ is separably closed.
If $\bar z$ maps to the point at infinity of $\mathbb P^1_{\mathbb F_p}$,
use $\tau=x^{-1}$ and replace $F$ by
$\tau^{2g+2}F=\prod_i h_i(\tau)$; the factor
$\tau^{2g+2}$ is a square in $\Frac(A)^\times$.
Otherwise, the divisor of $F$ on $\Spec A$ is supported on
$V(u)\cup V(v)$, so there exist $b\in\Frac(A)^\times$ and
$\alpha,\beta\in\{0,1\}$ with
\begin{equation}\label{eq:start-model-odd-square-class}
 F=b^2u^\alpha v^\beta.
\end{equation}
At infinity, the corresponding equality is
$\tau^{2g+2}F=b^2u^\alpha v^\beta$.
Removing $b^2$ describes the normalization over $\Spec A$ by
$A[z]/(z^2-u^\alpha v^\beta)$ if
$(\alpha,\beta)\ne(0,0)$, and by $A\times A$ if
$(\alpha,\beta)=(0,0)$. The algebras with
$(\alpha,\beta)=(1,0)$ or $(0,1)$ are regular. For example,
the maximal ideal of $A[z]/(z^2-u)$ is generated by $z,v$.
For $(\alpha,\beta)=(1,1)$, put
\begin{equation}\label{eq:start-model-node-ring}
 B=A[z]/(z^2-uv),\qquad \mathfrak n=(u,v,z)B.
\end{equation}
The ring $B$ is normal, and $\Spec B\setminus\{\mathfrak n\}$
is regular: the hypersurface $\Spec B$ is Cohen--Macaulay,
and the only possible singular point of $\Spec B$ is
$V(u,v,z)$, so Serre's normality
criterion applies. Thus all singular points of $X_{-1}$ over odd
primes lie over intersections of components of the divisor on
$P_p$. This divisor has finitely many component intersections,
and $X_{-1}\to P_m$ is finite. Hence there are finitely many
singular points of $X_{-1}$ over each $p$.
For $p>R$, the sections $x=a$, $0\leq a<N$, meet
$\mathbb P^1_{\mathbb F_p}$ at distinct points, and
\eqref{eq:start-model-disjoint-prime-restriction} identifies
$P_p=\mathbb P^1_{\Z_{(p)}}$. Moreover $F$ has
valuation zero at the generic point of
$\mathbb P^1_{\mathbb F_p}$.
Thus, at each finite point, at most one of $\alpha,\beta$ in
\eqref{eq:start-model-odd-square-class} is one. At infinity,
$\prod_i h_i(0)=1$, so the algebra
$A[w]/(w^2-\prod_i h_i(t))$ is \'etale over $A$: the element
$w$ is a unit, and the derivative of its defining polynomial is
$2w\in(A[w]/(w^2-\prod_i h_i(t)))^\times$.
The regularity of the corresponding local algebras proved above
shows that $X_{-1}$ has no singular point over $p>R$. Therefore
$Z_{-1}$ is finite.

The blow-up $b_B:\operatorname{Bl}_{\mathfrak n}(\Spec B)
\to\Spec B$ has affine opens
\begin{equation}\label{eq:start-explicit-node-resolution}
 \begin{aligned}
 B_u&=\Spec A[q]/(v-uq^2),& z&=uq,\\
 B_v&=\Spec A[r]/(u-vr^2),& z&=vr.
 \end{aligned}
\end{equation}
These are the opens corresponding to the generators $u$ and $v$
of $\mathfrak n$. Denote the open corresponding to the generator $z$ by
$B_z=\Spec B[u/z,v/z]$. The identity $(u/z)(v/z)=1$ in its
coordinate ring makes both $u/z$ and $v/z$ units and gives
\begin{equation}\label{eq:start-model-resolution-open-inclusion}
 B_z\subseteq B_u\cap B_v.
\end{equation}
Let $\xi$ be a closed point of the exceptional divisor in $B_u$,
and let $b\in A$ lift the residue of $q$ at $\xi$.
The relation $v=uq^2$ shows that the maximal ideal of
$\mathcal O_{B_u,\xi}$ is generated by $u,q-b$.
This local ring has dimension two and is regular.
For a closed point $\xi\in B_v$ on the exceptional divisor,
the relation $u=vr^2$ gives the generators $v,r-b$, where
$b\in A$ lifts the residue of $r$; thus
$\mathcal O_{B_v,\xi}$ is also regular.
Writing $E_B=b_B^{-1}(\mathfrak n)$, the restriction of $b_B$ is
\begin{equation}\label{eq:start-model-resolution-complement}
 \operatorname{Bl}_{\mathfrak n}(\Spec B)\setminus E_B
     \xrightarrow{\ \sim\ }\Spec B\setminus\{\mathfrak n\}.
\end{equation}
The target of \eqref{eq:start-model-resolution-complement} is
regular, as proved after \eqref{eq:start-model-node-ring}.
Therefore $\operatorname{Bl}_{\mathfrak n}(\Spec B)$ is regular.
Applying this calculation at every point of $Z_{-1}$ proves that
$S_{-1}$ is regular over every odd prime. The generic fibre of
$S_{-1}$ is $C_{-1}$ and is regular. Over $2$, equation
\eqref{eq:start-model-resolution-at-two} identifies $S_{-1}$
with the normal scheme $X_{-1}$. Hence $S_{-1}$ is normal
everywhere and regular away from its fibre over $2$.

For $j\in\{-1,0\}$, every local ring of $S_j$ is a subring of
its characteristic-zero function field. These local rings are
torsion-free over $\Z$, and hence flat over $\Z$ by \st{0AUW};
see also \cite[Section~4.3.1]{Liu}.
The generic fibre of $X_{-1}$ is the normalization of
$\mathbb P^1_\Q$ in $K$, namely $C_{-1}$.
The centres of $b_{-1}$ lie in special fibres, so
$(S_{-1})_\Q=(X_{-1})_\Q=C_{-1}$.
The generic fibre of $S_0$ is the normalization of $C_{-1}$ in
$L$, namely $C_0$.
Finally apply \eqref{eq:bg-explicit-fibre-dimension} to an affine
open $\Spec A\subset S_j$ containing the generic point
$\mathfrak q$ of an irreducible component of $(S_j)_{\mathbb F_p}$.
Here $\operatorname{ht}(\mathfrak q)=1$ and
$\operatorname{trdeg}_\Q\Frac(A)=1$, so that equation gives
$\operatorname{trdeg}_{\mathbb F_p}\kappa(\mathfrak q)=1$.
Thus every irreducible component of $(S_j)_{\mathbb F_p}$ has
dimension one, as asserted; see \cite[Section~8.3.1]{Liu} and
\st{02IJ}.
\end{proof}

\begin{proposition}\label{prop:start-integral-galois}
The morphism $\Pi_0:S_0\to S_{-1}$ is finite and surjective of
degree $2^g$. The extension of function fields induced by $\Pi_0$
is the Galois extension $L/K$, with group $G_0\simeq\F^g$, and
\begin{equation}\label{eq:start-model-automorphism-group}
 \operatorname{Aut}_{S_{-1}}(S_0)=\operatorname{Gal}(L/K)=G_0.
\end{equation}
The generic fibre of $\Pi_0$ is $\pi_0:C_0\to C_{-1}$.
\end{proposition}
\begin{proof}
The morphism $\Pi_0$ is finite by finiteness of normalization over
$\Z$. For every affine open $\Spec A\subset S_{-1}$, lying over
for the integral extension $A\subset\overline A^{\,L}$ proves
that $\Spec\overline A^{\,L}\to\Spec A$ is surjective.
Consequently $\Pi_0$ is surjective, and
$\deg(\Pi_0)=[L:K]=2^g$.
The generic fibre of $\Pi_0$ is the normalization of $C_{-1}$ in
$L$, so $(\Pi_0)_\Q=\pi_0$.

Let $\sigma\in\operatorname{Gal}(L/K)$ and
$u\in\overline A^{\,L}$.
If $u^d+a_1u^{d-1}+\cdots+a_d=0$ with $a_i\in A$, then
$\sigma$ fixes each $a_i$ and gives
$\sigma(u)^d+a_1\sigma(u)^{d-1}+\cdots+a_d=0$.
Thus $\sigma(\overline A^{\,L})=\overline A^{\,L}$.
Restriction of $\sigma$ defines an automorphism of
$\Spec\overline A^{\,L}$ over $\Spec A$.
For $f\in A\setminus\{0\}$, this automorphism restricts to
the same automorphism of
$\Spec(\overline A^{\,L}[f^{-1}])$ as the action of $\sigma$
on the integral closure of $A[f^{-1}]$ in $L$.
Hence these affine automorphisms define an automorphism of $S_0$
over $S_{-1}$. Conversely an automorphism of $S_0$ over $S_{-1}$
acts on $L$ and fixes $K$, and its action on $L$ determines its
action on every subring $\overline A^{\,L}\subset L$.
This proves \eqref{eq:start-model-automorphism-group}.
\end{proof}

Thus $\Pi_0$ is a finite Galois morphism according to the definition
in Section~\ref{sec:dictionary}: $S_0$ is the normalization of the
normal integral scheme $S_{-1}$ in the finite Galois extension
$L/K$. This assertion does not assert that $\Pi_0$ is \'etale
in every special fibre.
The quotient of $S_0$ by $G_0$ is $S_{-1}$. Indeed, for every
affine open $\Spec A\subset S_{-1}$,
\begin{equation}\label{eq:start-model-invariant-ring}
 (\overline A^{\,L})^{G_0}
  =\overline A^{\,L}\cap L^{G_0}
  =\overline A^{\,L}\cap K=A,
\end{equation}
since $A$ is integrally closed in $K$.

\begin{proposition}\label{prop:start-good-fibres}
The morphisms $S_{-1}\times_{\Spec\Z}U\to U$ and
$S_0\times_{\Spec\Z}U\to U$ are smooth and projective, with
geometrically integral one-dimensional fibres. The morphism
$\Pi_0\times_{\Spec\Z}U$ is a finite \'etale $G_0$-torsor.
For every prime $p>R$, the fibre $(S_{-1})_{\mathbb F_p}$ is
obtained by gluing
\begin{equation}\label{eq:start-explicit-good-fibre-base}
 \Spec\mathbb F_p[x,y]/(y^2-F(x))\quad\text{and}\quad
 \Spec\mathbb F_p[t,w]/\bigl(w^2-\prod_i h_i(t)\bigr),
\end{equation}
and the fibre $(S_0)_{\mathbb F_p}$ is obtained by gluing
\begin{equation}\label{eq:start-explicit-good-fibre-cover}
 \Spec\mathbb F_p[x,z_0,\ldots,z_g]/(z_i^2-f_i(x))_i
 \quad\text{and}\quad
 \Spec\mathbb F_p[t,v_0,\ldots,v_g]/(v_i^2-h_i(t))_i.
\end{equation}
For \eqref{eq:start-explicit-good-fibre-base}, identify the open
where $x\ne0$ in the left-hand scheme with the open where $t\ne0$
in the right-hand scheme by $t=x^{-1}$ and $w=x^{-(g+1)}y$.
For \eqref{eq:start-explicit-good-fibre-cover}, identify the open
where $x\ne0$ in the left-hand scheme with the open where $t\ne0$
in the right-hand scheme by $t=x^{-1}$ and $v_i=x^{-1}z_i$.
The morphism $(\Pi_0)_{\mathbb F_p}$ is induced by
$y\mapsto\prod_i z_i$ on the affine opens with coordinate $x$
and by $w\mapsto\prod_i v_i$ on the affine opens with coordinate $t$.
\end{proposition}
\begin{proof}
Set $\Lambda=\Z[1/\prod_{p\in\Sigma}p]$, so $U=\Spec\Lambda$,
and define
\begin{equation}\label{eq:start-model-good-rings}
 \begin{aligned}
 B_{-1,x}&=\Lambda[x,y]/(y^2-F(x)),&
 B_{-1,t}&=\Lambda[t,w]/\bigl(w^2-\prod_i h_i(t)\bigr),\\
 B_{0,x}&=\Lambda[x,z_0,\ldots,z_g]/(z_i^2-f_i(x))_i,&
 B_{0,t}&=\Lambda[t,v_0,\ldots,v_g]/(v_i^2-h_i(t))_i.
 \end{aligned}
\end{equation}
The substitutions in \eqref{eq:start-explicit-base-overlap} and
\eqref{eq:start-explicit-cover-overlap} are defined over $\Lambda$
and give isomorphisms
\begin{equation}\label{eq:start-model-good-localizations}
 \begin{aligned}
 B_{-1,t}[t^{-1}]&\xrightarrow{\ \sim\ }B_{-1,x}[x^{-1}],&
 t&\longmapsto x^{-1},&w&\longmapsto x^{-(g+1)}y,\\
 B_{0,t}[t^{-1}]&\xrightarrow{\ \sim\ }B_{0,x}[x^{-1}],&
 t&\longmapsto x^{-1},&v_i&\longmapsto x^{-1}z_i.
 \end{aligned}
\end{equation}
For $j\in\{-1,0\}$, let $Y_j$ be the scheme obtained by gluing
$\Spec B_{j,x}$ to $\Spec B_{j,t}$ along
$\Spec B_{j,x}[x^{-1}]$ and $\Spec B_{j,t}[t^{-1}]$ using the
corresponding isomorphism in
\eqref{eq:start-model-good-localizations}.
The inclusions of $\Lambda[x]$ and $\Lambda[t]$ into these rings
define finite morphisms $Y_j\to\mathbb P^1_U$.
All four rings in \eqref{eq:start-model-good-rings} are free over
the corresponding polynomial ring, with the monomial bases used in
Proposition~\ref{prop:start-curves}; in particular, $Y_j\to U$
is flat.

Let $k$ be an algebraically closed residue field of a point of $U$.
The images of $0,\ldots,2g+1$ in $k$ are distinct and
$2\in k^\times$.
At a point of $\Spec(B_{0,x}\otimes_\Lambda k)$ with every
$z_i\ne0$, the partial derivatives
$\partial(z_i^2-f_i(x))/\partial z_i=2z_i$ give a nonzero
diagonal $(g+1)\times(g+1)$ minor of the Jacobian matrix.
If $z_j=0$, then $z_i\ne0$ for $i\ne j$ and
$f'_j(x)\ne0$; the columns indexed by $x$ and the $z_i$ with
$i\ne j$ have determinant $\pm f'_j(x)\prod_{i\ne j}2z_i\ne0$.
At a point with $t=0$ in $\Spec(B_{0,t}\otimes_\Lambda k)$,
one has $v_i^2=h_i(0)=1$, so
$\partial(v_i^2-h_i(t))/\partial v_i=2v_i\ne0$ for every $i$.
For $\Spec(B_{-1,x}\otimes_\Lambda k)$, the derivatives of
$y^2-F(x)$ are $2y$ and $-F'(x)$; they cannot both vanish
because $F$ has distinct roots in $k$.
At $t=0$ in $\Spec(B_{-1,t}\otimes_\Lambda k)$,
$w^2=1$, so
$\partial(w^2-\prod_i h_i(t))/\partial w=2w\ne0$.
These derivatives prove that $(Y_{-1})_k$ and $(Y_0)_k$ are smooth.
The valuation at $x=i$ of $f_i$ is one, while the valuations of
$f_j$ for $j\ne i$ are zero. Thus the classes of the $f_i$ are
independent in $k(x)^\times/k(x)^{\times2}$, and
$F\notin k(x)^{\times2}$.
Define
\begin{equation}\label{eq:start-model-good-geometric-fields}
 K_k=k(x)(\sqrt F),\qquad
 L_k=k(x)(\sqrt{f_0},\ldots,\sqrt{f_g}).
\end{equation}
The free $k[x]$-modules $B_{-1,x}\otimes_\Lambda k$ and
$B_{0,x}\otimes_\Lambda k$ embed in $K_k$ and $L_k$, respectively,
by localizing at $k[x]\setminus\{0\}$. Likewise the free
$k[t]$-modules $B_{-1,t}\otimes_\Lambda k$ and
$B_{0,t}\otimes_\Lambda k$ embed in $K_k$ and $L_k$ using
$t=x^{-1}$ and the substitutions in
\eqref{eq:start-model-good-localizations}.
In particular, all four rings are domains, and for $j=-1,0$,
\begin{equation}\label{eq:start-model-good-nonempty-intersection}
 \Spec(B_{j,x}\otimes_\Lambda k)\cap
 \Spec(B_{j,t}\otimes_\Lambda k)
 =\Spec((B_{j,x}\otimes_\Lambda k)[x^{-1}])\ne\varnothing
 \quad\text{in }(Y_j)_k.
\end{equation}
Thus $(Y_{-1})_k$ and $(Y_0)_k$ are integral.

It follows that $Y_{-1}\to U$ and $Y_0\to U$ are smooth with
geometrically integral fibres. Since $U$ is regular, $Y_{-1}$
and $Y_0$ are regular and normal. Their function fields are $K$
and $L$, respectively. The finite morphisms
$Y_{-1}\to\mathbb P^1_U$ and $Y_0\to\mathbb P^1_U$ therefore
identify $Y_{-1}$ and $Y_0$ with the normalizations of
$\mathbb P^1_U$ in $K$ and $L$.
Equation~\eqref{eq:start-model-disjoint-prime-restriction} gives
$P_m\times_{\Spec\Z}U=\mathbb P^1_U$, so
$X_{-1}\times_{\Spec\Z}U=Y_{-1}$.
The scheme $Y_{-1}$ is regular, hence $Z_{-1}\times_{\Spec\Z}U$
is empty and $S_{-1}\times_{\Spec\Z}U=Y_{-1}$.
Normalizing this last scheme in $L$ gives
$S_0\times_{\Spec\Z}U=Y_0$.

For $0\leq j\leq g$, define the open subschemes
\begin{equation}\label{eq:start-model-etale-opens}
 \begin{aligned}
 W_{x,j}&=\Spec B_{-1,x}\bigl[(\prod_{i\ne j}f_i(x))^{-1}\bigr],\\
 W_{t,j}&=\Spec B_{-1,t}\bigl[(\prod_{i\ne j}h_i(t))^{-1}\bigr].
 \end{aligned}
\end{equation}
At every point of $\Spec B_{-1,x}$ at most one $f_i(x)$ vanishes,
so $\bigcup_jW_{x,j}=\Spec B_{-1,x}$.
Likewise $\bigcup_jW_{t,j}=\Spec B_{-1,t}$; at $t=0$ all
$h_i(0)=1$.
Write $A_{x,j}=\Gamma(W_{x,j},\mathcal O_{W_{x,j}})$.
Define the $A_{x,j}$-algebra $B_{x,j}$ by
\begin{equation}\label{eq:start-model-etale-algebra}
 \begin{aligned}
 B_{x,j}&=A_{x,j}[z_i:i\ne j]/(z_i^2-f_i(x))_{i\ne j},\\
 \Pi_0^{-1}(W_{x,j})&=\Spec B_{x,j},\qquad
 z_j=y\Big/\prod_{i\ne j}z_i.
 \end{aligned}
\end{equation}
For each $i\ne j$, we have $z_i\in B_{x,j}^{\times}$ and
$\left.\partial(Z^2-f_i(x))/\partial Z\right|_{Z=z_i}
=2z_i\in B_{x,j}^{\times}$. Thus $B_{x,j}$ is finite \'etale
of rank $2^g$ over $A_{x,j}$.
Put $A_{t,j}=\Gamma(W_{t,j},\mathcal O_{W_{t,j}})$ and define
\begin{equation}\label{eq:start-model-etale-infinity-algebra}
 \begin{aligned}
 B_{t,j}&=A_{t,j}[v_i:i\ne j]/(v_i^2-h_i(t))_{i\ne j},\\
 \Pi_0^{-1}(W_{t,j})&=\Spec B_{t,j},\qquad
 v_j=w\Big/\prod_{i\ne j}v_i.
 \end{aligned}
\end{equation}
For each $i\ne j$, we have $v_i\in B_{t,j}^{\times}$ and
$\left.\partial(V^2-h_i(t))/\partial V\right|_{V=v_i}
=2v_i\in B_{t,j}^{\times}$. Thus $B_{t,j}$ is finite \'etale
of rank $2^g$ over $A_{t,j}$.
Consequently $\Pi_0\times_{\Spec\Z}U$ is finite \'etale
of degree $2^g$.
Fix a geometric point $a$ of $W_{x,j}$ and two points $q,q'$
of the fibre $\Pi_0^{-1}(a)$. Define
\begin{equation}\label{eq:start-model-torsor-signs}
 \epsilon_i=\frac{z_i(q')}{z_i(q)}\in\{1,-1\}\quad(i\ne j),
 \qquad \epsilon_j=\prod_{i\ne j}\epsilon_i.
\end{equation}
Then $\epsilon=(\epsilon_0,\ldots,\epsilon_g)\in G_0$.
The identity $z_j=y/\prod_{i\ne j}z_i$ gives
$z_j(q')=\epsilon_jz_j(q)$, so the automorphism
$\sigma_\epsilon$ defined by $\sigma_\epsilon^*(z_i)=\epsilon_i z_i$
satisfies $\sigma_\epsilon(q)=q'$.
Every element of $G_0$ with this property must have the coordinates
in \eqref{eq:start-model-torsor-signs}; hence it equals $\epsilon$.
For $a\in W_{t,j}$ and $q,q'\in\Pi_0^{-1}(a)$, use instead
$\epsilon_i=v_i(q')/v_i(q)$ for $i\ne j$ and
$\epsilon_j=\prod_{i\ne j}\epsilon_i$. The identity
$v_j=w/\prod_{i\ne j}v_i$ again gives the unique automorphism
$\sigma_\epsilon$ with $\sigma_\epsilon(q)=q'$.
Therefore
$\Pi_0\times_{\Spec\Z}U$ is a $G_0$-torsor, as asserted.
\end{proof}

% Mathematical audit fragment: notation P=P_m, S_{-1}=C_m, S_0=T_m.
% All statements below use ell=2^m, g=ell-1, N=2ell.
\subsection{Regularity of \texorpdfstring{$S_0$}{S0}}

Recall the functions and fields
\begin{equation}\label{eq:fibres-explicit-functions}
 f_i=(x-i)(x-i-\ell)\quad(0\leq i\leq g),
 \qquad
 L=\mathbb Q(x)(\sqrt{f_0},\ldots,\sqrt{f_g}),
 \qquad y=\prod_{i=0}^g\sqrt{f_i}.
\end{equation}
The classes of $f_0,\ldots,f_g$ in
$\overline{\mathbb Q}(x)^\times/\overline{\mathbb Q}(x)^{\times2}$
are linearly independent over $\mathbb F_2$: for each $i$, the
valuation at $x=i$ is odd on $f_i$ and zero on $f_j$ for $j\ne i$.
Identify $H=\Gal(L/\mathbb Q(x))$ with $\mathbb F_2^{g+1}$ by sending
$(a_0,\ldots,a_g)$ to the automorphism fixing $x$ and satisfying
$\sqrt{f_i}\mapsto(-1)^{a_i}\sqrt{f_i}$ for $0\leq i\leq g$.
The subgroup $G_0=\Gal(L/K)$, where $K=\mathbb Q(x,y)$, is the kernel of
\begin{equation}\label{eq:fibres-explicit-parity-character}
 \lambda:H\longrightarrow\mathbb F_2,
 \qquad (a_0,\ldots,a_g)\longmapsto\sum_{i=0}^g a_i.
\end{equation}
We will compute the local rings of $S_0$ after strict henselization.
Regularity descends from a strict henselization, and normalization
commutes with strict henselization (\st{081J}, \st{0CBM}).
Finite \'etale algebras over a strictly henselian local ring are
products of copies of that ring; see \st{04GK}.
We also use the descriptions of surface singularities in
\cite[Sections~8.3.4 and~9.2.4]{Liu}.
Put $X_0=\operatorname{Norm}_L(P_m)$ and
$X_{-1}=\operatorname{Norm}_K(P_m)$.
By \eqref{eq:start-surface-map}, $S_{-1}\to X_{-1}$ is the blow-up
at the reduced union of the singular closed points of $X_{-1}$
above odd primes, and $S_0=\operatorname{Norm}_L(S_{-1})$.

\begin{proposition}\label{prop:starting-regularity}
The normalization $S_0$ of $S_{-1}$ in $L$ is regular.
\end{proposition}

\begin{proof}
Proposition~\ref{prop:start-good-fibres} proves that
$(S_0)_U=S_0\times_{\Spec\mathbb Z}U$ is regular.
We first prove regularity of
$S_0\times_{\Spec\mathbb Z}\Spec\mathbb Z_{(p)}$ for an odd prime
$p\in\Sigma$.
For an irreducible component $E$ of $(P_p)_{\mathbb F_p}$,
let $v_E$ be the normalized valuation of the discrete valuation ring
$\mathcal O_{P_p,\eta_E}$ at the generic point $\eta_E$ of $E$, and put
\begin{equation}\label{eq:fibres-explicit-component-inertia}
 b_E=(v_E(f_i)\bmod 2)_{0\leq i\leq g}\in H.
\end{equation}
Write $D_a\subset P_p$ for the restriction to $P_p$ of the strict
transform of $\sigma_a(\Spec\mathbb Z)$ on $P_m$.
For $a=i$ or $a=i+\ell$, put $b_{D_a}=e_i$, where
$e_i=(0,\ldots,0,1,0,\ldots,0)\in H$ has $i$th entry $1$.
For $D=D_a$, let $v_D$ be the normalized valuation of
$\mathcal O_{P_p,\eta_D}$.
If $D$ is a component of $(P_p)_{\mathbb F_p}$ or a divisor $D_a$,
and $w$ is an extension of $v_D$ to $L$, the inertia subgroup $I_w\subset H$ is
\begin{equation}\label{eq:fibres-explicit-inertia-subgroup}
 I_w=\langle b_D\rangle.
\end{equation}
To verify \eqref{eq:fibres-explicit-inertia-subgroup}, let
$R_D^{\mathrm{sh}}$ be a strict henselization of the discrete
valuation ring $\mathcal O_{P_p,\eta_D}$, let $\pi$ be a
uniformizer, and put $F_D=\operatorname{Frac}(R_D^{\mathrm{sh}})$.
Write $f_i=\pi^{v_D(f_i)}u_i$ with
$u_i\in(R_D^{\mathrm{sh}})^\times$.
Since $p\ne2$, every $u_i$ has a square root in
$R_D^{\mathrm{sh}}$.
Thus every field factor of $L\otimes_{\mathbb Q(x)}F_D$ is
$F_D$ when $b_D=0$, and is $F_D(\sqrt\pi)$ when $b_D\ne0$.
In the second case, the automorphism $\sqrt\pi\mapsto-\sqrt\pi$
acts on $\sqrt{f_i}$ by $(-1)^{v_D(f_i)}$.
Consequently $I_w=\{0\}$ when $b_D=0$, and otherwise $b_D$
is the unique nonidentity element of $I_w$.

Let $z$ be a geometric closed point of $(P_p)_{\mathbb F_p}$, and
let $(A,\mathfrak m_A)$ be the strict henselization of
$\mathcal O_{P_p,z}$, using the geometric point to choose the residue
field. The reduced divisor formed by $(P_p)_{\mathbb F_p}$ and
$D_0,\ldots,D_{N-1}$ has strict normal crossings.
Choose regular parameters $u,v$ of $A$ such that every component of
this divisor through $z$ is defined by $u=0$ or $v=0$.
Choose integers $b_i,c_i\in\{0,1\}$ satisfying
\begin{equation}\label{eq:fibres-explicit-local-parity-vectors}
 b_i\equiv\operatorname{ord}_{(u)}(f_i)\pmod2,\qquad
 c_i\equiv\operatorname{ord}_{(v)}(f_i)\pmod2
 \quad(0\leq i\leq g).
\end{equation}
Put $b=(b_i\bmod2)_{i=0}^g\in H$ and
$c=(c_i\bmod2)_{i=0}^g\in H$.
Here $\operatorname{ord}_{(u)}$ and $\operatorname{ord}_{(v)}$ are the
normalized valuations of $A_{(u)}$ and $A_{(v)}$.
When $u=0$ defines a component $D$ of the divisor, $b=b_D$;
when $u=0$ defines no component, $b=0$.
When $v=0$ defines a component $D$, $c=b_D$; otherwise $c=0$.
Thus \eqref{eq:fibres-explicit-inertia-subgroup} identifies the inertia
subgroups at $u=0$ and $v=0$ as $\langle b\rangle$ and
$\langle c\rangle$.

At a point above $x=\infty$, replace $f_i$ by
$\tau^2f_i=h_i(\tau)$ with $\tau=x^{-1}$; multiplication by
$\tau^2$ does not change the class in
$\operatorname{Frac}(A)^\times/\operatorname{Frac}(A)^{\times2}$.
Factoring $f_i$ in $\operatorname{Frac}(A)$, using that $A$ is factorial, gives
$f_i=\xi_i u^{b_i}v^{c_i}r_i^2$ with
$\xi_i\in A^\times$ and $r_i\in\operatorname{Frac}(A)^\times$.
Since $2\in A^\times$ and $A$ is strictly henselian, every
$\xi_i$ has a square root in $A$.
Let $\widetilde A_L$ be the integral closure of $A$ in
$L\otimes_{\mathbb Q(x)}\operatorname{Frac}(A)$, so that
$X_0\times_{P_m}\Spec A=\Spec\widetilde A_L$.
If $b,c$ are linearly independent, $\widetilde A_L$ is a product
of copies of
\begin{equation}\label{eq:fibres-explicit-independent-normalization}
 B_{u,v}=A[s,t]/(s^2-u,t^2-v).
\end{equation}
The maximal ideal of $B_{u,v}$ is $(s,t)$, and
$\dim B_{u,v}=2$, so $B_{u,v}$ is regular.
If $b\ne0$ and $c=0$, the factors of $\widetilde A_L$ are
$A[s]/(s^2-u)$, which are regular with parameters $s,v$.
The case $b=0$, $c\ne0$ follows by interchanging $u$ and $v$.
If $b=c=0$, then $\widetilde A_L$ is a product of copies of $A$.
If $b=c\ne0$, the factors are $A[z_1]/(z_1^2-uv)$.

Suppose $z\in E\cap F$ for two distinct irreducible components
$E,F$ of $(P_p)_{\mathbb F_p}$. We claim that $b_E\ne b_F$.
By the construction of $P_p$, one of the components, say $F$, was
created by blowing up a smooth point of the strict transform of $E$.
Let $B\subset\{0,\ldots,N-1\}$ consist of the indices $a$ for
which the strict transform of $\sigma_a(\Spec\mathbb Z)$ passed
through the centre of that blow-up. The set $B$ is nonempty.
For $a\in B$, $v_F(x-a)=v_E(x-a)+1$; for $a\notin B$,
$v_F(x-a)=v_E(x-a)$.
Consequently
\begin{equation}\label{eq:fibres-explicit-inertia-after-blowup}
 b_F-b_E=\sum_{a\in B}e_{\operatorname{pair}(a)},
 \qquad
 \operatorname{pair}(a)=
 \begin{cases}a,&0\leq a\leq g,\\a-\ell,&\ell\leq a<N.\end{cases}
\end{equation}
Since $p\nmid\ell$, the strict transforms of
$\sigma_i(\Spec\mathbb Z)$ and $\sigma_{i+\ell}(\Spec\mathbb Z)$
cannot pass through the same point above $p$.
Thus the map $\operatorname{pair}:B\to\{0,\ldots,g\}$ is
injective. The sum in
\eqref{eq:fibres-explicit-inertia-after-blowup} is a nonempty sum
of distinct basis vectors, so $b_F-b_E\ne0$.

Suppose instead that $z\in D_a\cap E$, where $E$ is a component
of $(P_p)_{\mathbb F_p}$ and $a=i$ or $a=i+\ell$.
If $b=c\ne0$ in \eqref{eq:fibres-explicit-local-parity-vectors},
then $b=c=e_i$, so $\lambda(b)=\lambda(c)=1$.
In this case put $R=A[z_1]/(z_1^2-uv)$.
Then $X_{-1}\times_{P_m}\Spec A=\Spec R$, and
$X_0\times_{P_m}\Spec A$ is a disjoint union of copies of $\Spec R$.
The restriction of $S_{-1}\to X_{-1}$ to $\Spec R$ is the blow-up
of the ideal $(u,v,z_1)$.
The computation in \eqref{eq:start-explicit-node-resolution} proves
that $\operatorname{Bl}_{(u,v,z_1)}(\Spec R)$ is regular.
The normalization in each of the identical field factors therefore
leaves this blow-up unchanged. Hence $S_0$ is regular above $z$ in
this case.

It remains at odd primes to consider $b,c$ linearly independent
with $\lambda(b)=\lambda(c)=1$.
Again put $R=A[z_1]/(z_1^2-uv)$; now choose $z_1=st$ in
$B_{u,v}$ from \eqref{eq:fibres-explicit-independent-normalization}.
The blow-up $\operatorname{Bl}_{(u,v,z_1)}(\Spec R)$ contains the
affine open
\begin{equation}\label{eq:fibres-explicit-quadratic-u-open}
 U_u=\Spec R_u,\qquad R_u=A[q]/(v-uq^2),\qquad
 q=z_1/u,\quad z_1=uq.
\end{equation}
The integral closure of $R_u$ in the field of fractions of
$B_{u,v}$ is
\begin{equation}\label{eq:fibres-explicit-normalized-u-open}
 B_u=A[s,q]/(u-s^2,\ v-s^2q^2),\qquad t=sq.
\end{equation}
The ring $B_u$ is finite over $R_u$, and the formulas
$s\mapsto s$, $q\mapsto t/s$ identify
$\operatorname{Frac}(B_u)=\operatorname{Frac}(B_{u,v})$.
To check normality, we prove that $B_u$ is regular: at a closed point $\xi$ of
$\Spec B_u$ with $s(\xi)=0$, choose $\beta\in A$ lifting $q(\xi)$.
The maximal ideal of $(B_u)_\xi$ is $(s,q-\beta)$ and
$\dim(B_u)_\xi=2$.
The affine open associated with the generator $v$ is
$U_v=\Spec R_v$, where
$R_v=A[q']/(u-v(q')^2)$ and $q'=z_1/v$.
Its integral closure in the field of fractions of $B_{u,v}$ is
\begin{equation}\label{eq:fibres-explicit-normalized-v-open}
 B_v=A[t,q']/(v-t^2,\ u-t^2(q')^2),\qquad s=tq'.
\end{equation}
At a closed point $\xi'$ of $\Spec B_v$ with $t(\xi')=0$,
the maximal ideal of $(B_v)_{\xi'}$ is $(t,q'-\beta')$, where
$\beta'\in A$ lifts $q'(\xi')$. Thus $(B_v)_{\xi'}$ is regular.
Outside $V(s)\subset\Spec B_u$ and $V(t)\subset\Spec B_v$,
respectively, the homomorphisms
$B_{u,v}\to B_u$, $t\mapsto sq$, and
$B_{u,v}\to B_v$, $s\mapsto tq'$, give the isomorphisms
\begin{equation}\label{eq:fibres-mechanical-step-1}
 B_{u,v}[s^{-1}]\simeq B_u[s^{-1}],\qquad
 B_{u,v}[t^{-1}]\simeq B_v[t^{-1}].
\end{equation}
Consequently $B_u$ and $B_v$ are regular everywhere.
The affine open associated with the generator $z_1$ is
\begin{equation}\label{eq:fibres-explicit-quadratic-z-open}
 U_{z_1}=\Spec A[z_1,\alpha,\beta]/
 (u-z_1\alpha,v-z_1\beta,\alpha\beta-1),\qquad
 \alpha=u/z_1,\quad\beta=v/z_1.
\end{equation}
Since $\alpha\beta=1$, one has
$U_{z_1}=D(q)\subset U_u$ and
$U_{z_1}=D(q')\subset U_v$, with $q'=q^{-1}$.
Thus $U_u\cup U_v$ covers the blow-up, and
\eqref{eq:fibres-explicit-normalized-u-open}--
\eqref{eq:fibres-explicit-normalized-v-open} prove that $S_0$ is
regular above $z$.
If $\lambda(b)$ and $\lambda(c)$ are not both $1$, then
$X_{-1}$ is regular above $z$ and the morphism $S_{-1}\to X_{-1}$
is an isomorphism there. The regularity of $B_{u,v}$ and of the
other factors already considered then proves regularity of $S_0$
above $z$.
This proves regularity over $\Spec\mathbb Z_{(p)}$ for every odd
$p\in\Sigma$.

We now consider the prime $2$. Put $k=m-2$ and
\begin{equation}\label{eq:fibres-explicit-dyadic-functions}
 A_i=-\frac{2^{m-2}}{x-i},\qquad
 \frac{f_i}{(x-i)^2}=1+4A_i\quad(0\leq i\leq g).
\end{equation}
At the generic point of $E_{j,c}\subset(P_2)_{\mathbb F_2}$,
where $j\leq k$, the valuation $v_{E_{j,c}}$ satisfies
$v_{E_{j,c}}(2)=1$ and $v_{E_{j,c}}(x-i)\leq j$.
It follows that $v_{E_{j,c}}(A_i)\geq0$.
Among the horizontal prime divisors of $P_2$, only $D_i$ can have
negative valuation on $A_i$.
The local rings of $P_2$ are regular and hence factorial
(\st{0AG0}). Therefore
\begin{equation}\label{eq:fibres-mechanical-step-2}
 A_i\in\mathcal O_{P_2,z}
 \quad\text{for every }z\in(P_2)_{\mathbb F_2}\setminus D_i.
\end{equation}
The intersection of $D_i$ with $(P_2)_{\mathbb F_2}$ lies on
$E_{k,c}$ with $c\equiv i\pmod{2^k}$.

Let $z$ be a geometric point of $(P_2)_{\mathbb F_2}$ outside
$\bigcup_a D_a$, and let $A$ be the strict henselization of
$\mathcal O_{P_2,z}$. For each $i$ define
\begin{equation}\label{eq:fibres-explicit-dyadic-substitution}
 W_i=\frac{\sqrt{f_i}/(x-i)-1}{2},\qquad
 P_i(W)=W^2+W-A_i,\qquad P_i(W_i)=0.
\end{equation}
Here $A_i\in A$. Let $w_i$ be the image of the indeterminate
$W$ in $A[W]/(P_i(W))$. Then
$P_i'(w_i)=2w_i+1\in(A[W]/(P_i(W)))^\times$.
Thus $A[W]/(P_i(W))$ is finite \'etale over $A$, and strict
henselianity gives $A[W]/(P_i(W))\simeq A\times A$.
Consequently
\begin{equation}\label{eq:fibres-mechanical-step-3}
 \widetilde A_L\simeq\prod_{j=1}^{2^{g+1}}A.
\end{equation}
In particular, $X_0$ is regular at every point above $z$.
This calculation applies when $z$ is an intersection of two
components of $(P_2)_{\mathbb F_2}$, since $D_a\cap E\cap F=\varnothing$ for every $a$ and
every pair of distinct fibre components $E,F$.

For $0\leq j<2^{m-1}$, let $\zeta_j$ be the point of
$E_{k,c}$ through which the divisors $D_a$ with
$a\equiv j\pmod{2^{m-1}}$ pass, where $c\equiv j\pmod{2^k}$.
These are precisely the points of
$(P_2)_{\mathbb F_2}\cap\bigcup_aD_a$.
Let $A$ be the strict henselization of $\mathcal O_{P_2,\zeta_j}$
and put $t=(x-j)/2^{m-2}$, so $\mathfrak m_A=(2,t)$.
For $i\notin\{j,j+2^{m-1}\}$, the polynomial
$P_i(W)$ from \eqref{eq:fibres-explicit-dyadic-substitution}
has coefficients in $A$ and defines $A\times A$ as above.
For the remaining indices put
$U=\sqrt{f_j}/2^{m-2}$ and
$V=\sqrt{f_{j+2^{m-1}}}/2^{m-2}$. Then
\begin{equation}\label{eq:fibres-explicit-dyadic-two-extensions}
 U^2=t(t-4),\qquad V^2=(t-2)(t-6).
\end{equation}
Define $a,b,d$ by
\begin{equation}\label{eq:fibres-explicit-dyadic-integral-substitutions}
 U=t+2a,\qquad V=t-2+2b,\qquad d=b-a.
\end{equation}
Substitution in \eqref{eq:fibres-explicit-dyadic-two-extensions}
gives the finite free $A$-algebra of rank four
\begin{equation}\label{eq:fibres-explicit-dyadic-rank-four-algebra}
 B=A[a,b]/(a^2+ta+t,\ b^2+(t-2)b+t-2).
\end{equation}
Writing $F_A=\operatorname{Frac}(A)$, one has
\begin{equation}\label{eq:fibres-mechanical-step-4}
 B\otimes_A F_A\simeq
 F_A[U,V]/(U^2-t(t-4),V^2-(t-2)(t-6)).
\end{equation}
The quotient $B/(2,t)B=\kappa(A)[a,b]/(a^2,b^2)$ has one
prime ideal. Thus the unique maximal ideal of $B$ above
$\mathfrak m_A$ is $\mathfrak n=(2,t,a,b)$.
In $B_{\mathfrak n}$ the defining relations give
\begin{equation}\label{eq:fibres-explicit-dyadic-parameters}
 t(a+1)=-a^2,\qquad
 d^2+td+2\bigl((a-1)d-(a+1)\bigr)=0.
\end{equation}
The residues of $a+1$ and $(a-1)d-(a+1)$ in
$B_{\mathfrak n}/\mathfrak nB_{\mathfrak n}$ are nonzero.
Hence \eqref{eq:fibres-explicit-dyadic-parameters} gives
$t,2\in(a,d)B_{\mathfrak n}$, and therefore
$\mathfrak nB_{\mathfrak n}=(a,d)$.
Since $\dim B_{\mathfrak n}=2$, the ring $B_{\mathfrak n}$ is
regular and normal.
The algebra $B$ is finite over the local ring $A$ and has only the
maximal ideal $\mathfrak n$, so $B=B_{\mathfrak n}$.
Thus $B$ is the integral closure of $A$ in $B\otimes_A F_A$.
The equations for $i\notin\{j,j+2^{m-1}\}$ split over $A$, and
consequently
\begin{equation}\label{eq:fibres-mechanical-step-5}
 \widetilde A_L\simeq\prod_{h=1}^{2^{g-1}}B.
\end{equation}
This proves that $X_0$ is regular above every $\zeta_j$.
Finally,
$S_{-1}\times_{\Spec\mathbb Z}\Spec\mathbb Z_{(2)}
 \simeq X_{-1}\times_{\Spec\mathbb Z}\Spec\mathbb Z_{(2)}$.
Normalization in $L$ then gives
$S_0\times_{\Spec\mathbb Z}\Spec\mathbb Z_{(2)}
 \simeq X_0\times_{\Spec\mathbb Z}\Spec\mathbb Z_{(2)}$. Hence
$S_0\times_{\Spec\mathbb Z}\Spec\mathbb Z_{(2)}$ is regular,
completing the proof.
\end{proof}

\subsection{The section}

Recall the point $Q_+\in U_{0,t}(\mathbb Q)\subset C_0(\mathbb Q)$
with $t(Q_+)=0$ and $v_i(Q_+)=1$ for $0\leq i\leq g$.
We will extend $Q_+:\Spec\mathbb Q\to C_0$ to a section of
$S_0\to\Spec\mathbb Z$.
The extension uses the valuative criterion for properness and
extension of a morphism from a local scheme to an open
neighbourhood; see \cite[Section~3.3.2]{Liu}, \st{0BX5} and \st{0BX6}.

\begin{proposition}\label{prop:start-integral-section}
The point $Q_+\in C_0(\Q)$ extends uniquely to a section
\begin{equation}\label{eq:fibres-explicit-section}
 s_0:\Spec\Z\longrightarrow S_0.
\end{equation}
The composite $s_{-1}=\Pi_0\circ s_0:\Spec\mathbb Z\to S_{-1}$
is the unique section extending $P_+\in C_{-1}(\mathbb Q)$.
The morphism $S_0\to\Spec\mathbb Z$ is smooth at every point
of $s_0(\Spec\mathbb Z)$, and
$S_{-1}\to\Spec\mathbb Z$ is smooth at every point of
$s_{-1}(\Spec\mathbb Z)$.
Moreover,
\begin{equation}\label{eq:fibres-explicit-split-section-fibre}
 S_0\times_{S_{-1},s_{-1}}\Spec\Z
 \simeq\coprod_{j=1}^{2^g}\Spec\Z.
\end{equation}
\end{proposition}
\begin{proof}
For every prime $p$, properness gives a unique extension
$s_{0,p}:\Spec\mathbb Z_{(p)}\to S_0$ of $Q_+$ over
$\Spec\mathbb Z$.
Since $S_0\to\Spec\mathbb Z$ is of finite presentation, there
are an open neighbourhood $V_p\subset\Spec\mathbb Z$ of $(p)$
and a morphism $s_{0,V_p}:V_p\to S_0$ whose restriction to
$\Spec\mathbb Z_{(p)}$ is $s_{0,p}$.
For any primes $p,q$, the morphisms
$s_{0,V_p}|_{V_p\cap V_q}$ and $s_{0,V_q}|_{V_p\cap V_q}$
agree at $\Spec\mathbb Q$. Separatedness of
$S_0\to\Spec\mathbb Z$ gives
\begin{equation}\label{eq:fibres-mechanical-step-6}
 s_{0,V_p}|_{V_p\cap V_q}=s_{0,V_q}|_{V_p\cap V_q}.
\end{equation}
Therefore the morphisms $s_{0,V_p}$ define a section $s_0$ on
$\bigcup_pV_p=\Spec\mathbb Z$.
The uniqueness in the valuative criterion proves uniqueness of $s_0$.
Applying the same argument to $S_{-1}\to\Spec\mathbb Z$ and
$P_+$ gives a unique section $s_{-1}$; since
$\pi_0(Q_+)=P_+$, uniqueness gives $\Pi_0\circ s_0=s_{-1}$.

We next check the morphisms near the sections explicitly.
Put $T=x^{-1}$ and
$H(T)=\prod_{i=0}^g(1-iT)(1-(i+\ell)T)$.
For every prime $p$ let
\begin{equation}\label{eq:fibres-mechanical-step-7}
 R_p=\mathbb Z_{(p)}[T,H(T)^{-1}],\qquad
 \mathcal V_p=\Spec R_p\subset\mathbb P^1_{\mathbb Z_{(p)}}.
\end{equation}
The section $T=0$ lies in $\mathcal V_p$.
The centres of the blow-ups defining $P_m$ lie on the images of
$\sigma_0,\ldots,\sigma_{N-1}$, which are disjoint from
$\mathcal V_p$. Hence
$P_m\times_{\mathbb P^1_{\mathbb Z}}\mathcal V_p\simeq\mathcal V_p$.
For odd $p$ the elements
\begin{equation}\label{eq:fibres-explicit-infinity-odd-functions}
 h_i(T)=T^2f_i=(1-iT)(1-(i+\ell)T)\in R_p^\times
 \quad(0\leq i\leq g)
\end{equation}
are units. Define
$B_p=R_p[v_0,\ldots,v_g]/(v_i^2-h_i(T):0\leq i\leq g)$.
For $P_i(V)=V^2-h_i(T)$ one has
$P_i'(v_i)=2v_i\in B_p^\times$.
Thus $\Spec B_p\to\mathcal V_p$ is finite \'etale.
Since $R_p$ is regular, $B_p$ is regular.
The fraction field of $B_p$ is $L$, so $\Spec B_p$ is the
restriction of $X_0$ to $\mathcal V_p$.
Likewise, $R_p[w]/(w^2-H(T))$ is finite \'etale over $R_p$
and is the integral closure of $R_p$ in $K$.
Therefore the restriction of $X_{-1}$ to $\mathcal V_p$ is
regular, and the blow-up $S_{-1}\to X_{-1}$ is an isomorphism
on this restriction.

For $p=2$, the functions
\begin{equation}\label{eq:fibres-explicit-infinity-dyadic-functions}
 A_i=-\frac{2^{m-2}}{x-i}
     =-\frac{2^{m-2}T}{1-iT}\in R_2
 \quad(0\leq i\leq g)
\end{equation}
are defined because $1-iT\in R_2^\times$.
Put
$B_2=R_2[W_0,\ldots,W_g]/(W_i^2+W_i-A_i:0\leq i\leq g)$.
For $P_i(W)=W^2+W-A_i$ one has
\begin{equation}\label{eq:fibres-mechanical-step-8}
 P_i'(W_i)=2W_i+1,\qquad
 (2W_i+1)^2=1+4A_i=\frac{1-(i+\ell)T}{1-iT}\in R_2^\times.
\end{equation}
Consequently $\Spec B_2\to\mathcal V_2$ is finite \'etale,
and $B_2$ is the integral closure of $R_2$ in $L$.
The integral closure of $R_2$ in the intermediate field $K$
is also finite \'etale over $R_2$.
The identification with the functions on $C_0$ is
\begin{equation}\label{eq:fibres-explicit-infinity-substitution}
 \frac{\sqrt{f_i}}{x-i}=1+2W_i,\qquad
 v_i=(1-iT)(1+2W_i).
\end{equation}
The restriction to $T=0$ has the explicit presentation
\begin{equation}\label{eq:fibres-mechanical-step-9}
 B_2/TB_2=
 \mathbb Z_{(2)}[W_0,\ldots,W_g]/
 (W_i(W_i+1):0\leq i\leq g).
\end{equation}
The points with $v_i=1$ and $v_i=-1$ on the generic fibre of
$T=0$ extend through $W_i=0$ and $W_i=-1$, respectively.
Since $0\ne-1$ in $\mathbb F_2$, these sections of
$\Spec B_2\to\Spec\mathbb Z_{(2)}$ are distinct after reduction.

For each $p$, the restrictions of $S_0$ and $S_{-1}$ to
$\mathcal V_p$ are finite \'etale over $\mathcal V_p$.
Since $\mathcal V_p\to\Spec\mathbb Z_{(p)}$ is smooth,
$S_0\to\Spec\mathbb Z$ and $S_{-1}\to\Spec\mathbb Z$ are
smooth at the points above $T=0$.
The induced morphism
$S_0\times_{P_m}\mathcal V_p\to S_{-1}\times_{P_m}\mathcal V_p$
is a morphism between finite \'etale $\mathcal V_p$-schemes,
and is therefore \'etale.

For $\epsilon=(\epsilon_i)_{i=0}^g\in\{1,-1\}^{g+1}$ with
$\prod_i\epsilon_i=1$, let $Q_\epsilon\in C_0(\mathbb Q)$
be the point $T=0$, $v_i=\epsilon_i$.
Then $\pi_0^{-1}(P_+)=\{Q_\epsilon:\prod_i\epsilon_i=1\}$.
Properness extends each $Q_\epsilon$ uniquely to a section
$s_\epsilon:\Spec\mathbb Z\to S_0$, and
$\Pi_0\circ s_\epsilon=s_{-1}$.
For odd $p$, the sections $s_\epsilon$ have distinct tuples
$(v_i)=(\epsilon_i)$ in the fibre at $p$.
For $p=2$, the tuples
$(W_i)=((\epsilon_i-1)/2)$ are distinct in that fibre.
Thus the morphism
\begin{equation}\label{eq:fibres-mechanical-step-10}
 \coprod_{\epsilon:\prod_i\epsilon_i=1}\Spec\mathbb Z
 \longrightarrow S_0\times_{S_{-1},s_{-1}}\Spec\mathbb Z
\end{equation}
induced by the sections $s_\epsilon$ is an isomorphism on every
geometric fibre. Both schemes are finite \'etale of degree $2^g$
over $\Spec\mathbb Z$, so the morphism is an isomorphism.
This proves \eqref{eq:fibres-explicit-split-section-fibre}.
\end{proof}

\begin{proposition}\label{prop:start-bost-charles-model}
For every $m\geq4$, the surface $S_0$ defined in
\eqref{eq:start-surface-map} is an integral regular projective flat
arithmetic surface over $\Spec\Z$, with smooth geometrically integral
generic fibre $C_0$ and the section $s_0$.
\end{proposition}
\begin{proof}
Proposition~\ref{prop:start-models} proves that $S_0$ is integral,
projective and flat over $\Spec\mathbb Z$, with generic fibre $C_0$.
Proposition~\ref{prop:start-curves} proves that $C_0$ is smooth
and geometrically integral.
Proposition~\ref{prop:starting-regularity} proves that $S_0$ is
regular, and Proposition~\ref{prop:start-integral-section} constructs
the section $s_0$.
\end{proof}

\subsection{The bad fibres}

For $p\in\Sigma$, we describe
$S_{-1}\times_{\Spec\mathbb Z}\Spec\overline{\mathbb F}_p$
and $S_0\times_{\Spec\mathbb Z}\Spec\overline{\mathbb F}_p$.
The schemes $S_{-1}$ and $S_0$ need not be minimal models.
The composite $S_0\to S_{-1}\to P_m$ factors uniquely through
the normalization $X_0=\operatorname{Norm}_L(P_m)$.
Denote the resulting morphism by $\mu_0:S_0\to X_0$.
The morphism $\mu_0$ is proper and birational and is an
isomorphism above every codimension-one point of $X_0$.
For an irreducible curve $C\subset X_0$, its strict transform
on $S_0$ means the closure of $\mu_0^{-1}(\eta_C)$, where
$\eta_C$ is the generic point of $C$.
For a component $C\subset(X_0)_{\overline{\mathbb F}_p}$, we use
the closure in $(S_0)_{\overline{\mathbb F}_p}$ of
$(\mu_0)_{\overline{\mathbb F}_p}^{-1}(\eta_C)$.
We keep track of the multiplicity of each irreducible component in
the scheme-theoretic fibre as well as the reduced fibre; see
\cite[Sections~8.3.1 and~10.1.1--10.1.4]{Liu} and \st{0CBN}.
The reduced fibres at odd primes are nodal, although some components
have multiplicity greater than one in the scheme-theoretic fibres.
The fibres at $2$ are reduced, but at some singular points the two
branches have intersection multiplicity two.

\begin{proposition}\label{prop:starting-odd-fibres}
Let $p\in\Sigma$ be odd and put $k_p=\overline{\mathbb F}_p$.
The reduced curves $(S_{-1})_{k_p,\mathrm{red}}$ and
$(S_0)_{k_p,\mathrm{red}}$ are nodal.
Their irreducible components, multiplicities and local rings can be
computed as follows from $(P_p)_{k_p}$, the graph $\Gamma_p$ in
\eqref{eq:start-explicit-fibre-graph}, and the divisors $D_a$.
Recall that $X_{-1}=\operatorname{Norm}_K(P_m)$,
$X_0=\operatorname{Norm}_L(P_m)$,
$S_{-1}\to X_{-1}$ is the blow-up at the singular closed points
above odd primes, and $S_0=\operatorname{Norm}_L(S_{-1})$.
Items (1) and (2) describe the components of $(X_0)_{k_p}$ and
$(X_{-1})_{k_p}$, respectively. Items (3) and (4) describe the
changes under $S_{-1}\to X_{-1}$ and normalization in $L$.
\begin{enumerate}
\item Let $E$ be an irreducible component of $(P_p)_{k_p}$.
The component $E$ is obtained by base change from a component
$E^\circ\subset(P_p)_{\mathbb F_p}$.
Let $v_E$ be the normalized valuation of
$\mathcal O_{P_p,\eta_{E^\circ}}$, and choose a uniformizer $q$
of this discrete valuation ring.
Put $m_i=v_E(f_i)$,
$b_E=(m_i\bmod2)_{i=0}^g\in H$, and
\begin{equation}\label{eq:fibres-mechanical-step-11}
 b_E^\perp=\{(a_i)_{i=0}^g\in H:\textstyle\sum_i a_i(b_E)_i=0\}.
\end{equation}
For $a=(a_i)\in b_E^\perp$, define $\xi_E(a)$ to be the class
\begin{equation}\label{eq:fibres-explicit-residual-square-classes}
 \xi_E(a)=\left[
 \overline{q^{-\sum_i a_i m_i}\prod_i f_i^{a_i}}
 \right]\in k_p(E)^\times/k_p(E)^{\times2}.
\end{equation}
In \eqref{eq:fibres-explicit-residual-square-classes}, each $a_i$
is represented by $0$ or $1$, and the bar denotes reduction in
the residue field at $\eta_{E^\circ}$ followed by extension to
$k_p(E)$.
If $E$ was created in the construction of $P_p$ with coordinate
$X=(x-c)/p^k$, then
\begin{equation}\label{eq:fibres-explicit-component-valuations}
 m_i=\min\{k,v_p(c-i)\}+\min\{k,v_p(c-i-\ell)\},
 \qquad v_p(0)=+\infty.
\end{equation}
For the strict transform of
$\mathbb P^1_{\mathbb F_p}\subset\mathbb P^1_{\mathbb Z}$,
take $k=0$ in \eqref{eq:fibres-explicit-component-valuations}.
Set
\begin{equation}\label{eq:fibres-mechanical-step-12}
 r_E=\dim_{\mathbb F_2}\langle\xi_E(a):a\in b_E^\perp\rangle,
 \qquad
 \epsilon_E=\begin{cases}0,&b_E=0,\\1,&b_E\ne0.\end{cases}
\end{equation}
There are exactly $2^{g+1-\epsilon_E-r_E}$ irreducible
components $C\subset(X_0)_{k_p}$ whose image in $(P_p)_{k_p}$
is $E$. Each component $C$ mapping onto $E$ has multiplicity
$2^{\epsilon_E}$ in
$(X_0)_{k_p}$.
Choose representatives $u_a\in k_p(E)^\times$ of $\xi_E(a)$.
The normalization $\widetilde C$ of $C$ is the smooth projective
curve with function field
$k_p(\widetilde C)=k_p(E)(\sqrt{u_a}:a\in b_E^\perp)$.
The strict transform of $C$ on $S_0$ has the same multiplicity
$2^{\epsilon_E}$.

For the genus calculation, let $\mathcal D_E$ be the set consisting
of the components $D\ne E$ of $(P_p)_{k_p}$ with
$D\cap E\ne\varnothing$, together with the divisors $D_a$
with $D_a\cap E\ne\varnothing$.
For a component $D$ use the vector $b_D$ from
\eqref{eq:fibres-explicit-component-inertia}; for $D_a$ put
$b_{D_a}=e_i$ when $a=i$ or $a=i+\ell$.
Then
\begin{equation}\label{eq:fibres-explicit-residual-rank}
 r_E=\dim_{\F}\langle b_E,b_D:D\in\mathcal D_E\rangle-\epsilon_E.
\end{equation}
Put $B_E=\#\{D\in\mathcal D_E:b_D\notin\langle b_E\rangle\}$.
For every component $C\subset(X_0)_{k_p}$ mapping onto $E$,
the genus of $\widetilde C$ is zero if $r_E=0$, and otherwise
\begin{equation}\label{eq:fibres-explicit-residual-genus}
 g(\widetilde C)=1+2^{r_E-2}(B_E-4).
\end{equation}
\item Put $F=\prod_{i=0}^g f_i$.
If $\sum_i m_i$ is odd, there is exactly one component
$C\subset(X_{-1})_{k_p}$ mapping onto $E$;
The multiplicity of $C$ is two and $k_p(C)=k_p(E)$.
If $\sum_i m_i$ is even, put
\begin{equation}\label{eq:fibres-explicit-quadratic-residue}
 \overline F_E=\overline{q^{-\sum_i m_i}F}\in k_p(E)^\times.
\end{equation}
If $\overline F_E\notin k_p(E)^{\times2}$, there is exactly
one component $C\subset(X_{-1})_{k_p}$ mapping onto $E$, and
$k_p(C)=k_p(E)(\sqrt{\overline F_E})$ has degree two over
$k_p(E)$.
If $\overline F_E\in k_p(E)^{\times2}$, there are exactly two
components $C\subset(X_{-1})_{k_p}$ mapping onto $E$, and
$k_p(C)=k_p(E)$ for each component.
When $\sum_i m_i$ is even, every component mapping onto $E$
has multiplicity one.
\item Let $z\in D\cap D'$ be an intersection of two distinct
members of the union of the fibre components and
$D_0,\ldots,D_{N-1}$.
Let $A$ be the strict henselization of $\mathcal O_{P_p,z}$,
and choose regular parameters $u,v$ of $A$ so that
$D\times_{P_p}\Spec A=V(u)$ and
$D'\times_{P_p}\Spec A=V(v)$.
Put $b=b_D$ and $c=b_{D'}$. Equivalently,
$b_i=\operatorname{ord}_{(u)}(f_i)\bmod2$ and
$c_i=\operatorname{ord}_{(v)}(f_i)\bmod2$.
The inertia subgroups at $D$ and $D'$ are
$\langle b\rangle$ and $\langle c\rangle$, as in
\eqref{eq:fibres-explicit-inertia-subgroup}.
After interchanging $(D,u,b)$ and $(D',v,c)$ if necessary,
the integral closure of $A$ in
$L\otimes_{\mathbb Q(x)}\operatorname{Frac}(A)$ is a product
of copies of the following ring:
\begin{equation}\label{eq:fibres-explicit-odd-local-rings}
\begin{array}{c|c}
 \text{condition on }b,c&\text{factor of the integral closure}\\\hline
 b=c=0&A\\
 b\ne0,\ c=0&A[s]/(s^2-u)\\
 b=c\ne0&A[z_1]/(z_1^2-uv)\\
 b,c\text{ linearly independent}&A[s,t]/(s^2-u,t^2-v).
\end{array}
\end{equation}
The scheme $X_{-1}\times_{P_m}\Spec A$ has a singular point
if and only if $\lambda(b)=\lambda(c)=1$.
In this case
$X_{-1}\times_{P_m}\Spec A=\Spec A[z_1]/(z_1^2-uv)$,
and $S_{-1}\to X_{-1}$ restricts to the blow-up at
$(u,v,z_1)$.
The exceptional curve is $\mathbb P^1_{k_p}$.
Its multiplicity in $(S_{-1})_{k_p}$ is one if one of $D,D'$
is a divisor $D_a$, and two if $D,D'$ are both components
of $(P_p)_{k_p}$.
\item In the notation of (3), suppose $\lambda(b)=\lambda(c)=1$.
If $b,c$ are linearly independent, the normalization of
$\operatorname{Bl}_{(u,v,z_1)}(\Spec A[z_1]/(z_1^2-uv))$
in each field factor of
$L\otimes_{\mathbb Q(x)}\operatorname{Frac}(A)$ is
$\operatorname{Bl}_{(s,t)}(\Spec A[s,t]/(s^2-u,t^2-v))$.
There are $2^{g-1}$ exceptional components of $(S_0)_{k_p}$
mapping to $z$, each isomorphic to $\mathbb P^1_{k_p}$.
The multiplicity of each exceptional component is two if one
of $D,D'$ is a divisor $D_a$, and four if $D,D'$ are both
components of $(P_p)_{k_p}$.
If $b=c$, then exactly one of $D,D'$ is a divisor $D_a$.
In this case there are $2^g$ exceptional components of
$(S_0)_{k_p}$ mapping to $z$, each isomorphic to
$\mathbb P^1_{k_p}$ and each of multiplicity one.
\end{enumerate}
In particular, every component of $(S_{-1})_{k_p}$ has
multiplicity at most two, and every component of $(S_0)_{k_p}$
has multiplicity at most four.
\end{proposition}

\begin{proof}
Fix a component $E\subset(P_p)_{k_p}$ as in (1).
By \eqref{eq:fibres-explicit-inertia-subgroup}, the inertia
subgroup of an extension of $v_E$ to $L$ is
$\langle b_E\rangle$, of order $2^{\epsilon_E}$.
After extension of the constant field to $k_p$, the residue
field extension is
$k_p(E)(\sqrt{u_a}:a\in b_E^\perp)/k_p(E)$, of degree
$2^{r_E}$.
Thus, for every component $C\subset(X_0)_{k_p}$ mapping onto $E$,
\begin{equation}\label{eq:fibres-mechanical-step-13}
 [k_p(C):k_p(E)]=2^{r_E},\qquad e(C/E)=2^{\epsilon_E}.
\end{equation}
The equality $[L:\mathbb Q(x)]=2^{g+1}$ therefore gives
\begin{equation}\label{eq:fibres-mechanical-step-14}
 \#\{C\subset(X_0)_{k_p}:C\text{ irreducible},\ C\mapsto E\}
 =\frac{2^{g+1}}{2^{r_E}2^{\epsilon_E}}
 =2^{g+1-\epsilon_E-r_E}.
\end{equation}
Since $(P_p)_{k_p}$ is reduced, the multiplicity of $C$ in
$(X_0)_{k_p}$ equals $e(C/E)=2^{\epsilon_E}$.
For $K=\mathbb Q(x)(\sqrt F)$, the same valuation calculation
has ramification index two when $v_E(F)=\sum_i m_i$ is odd.
When $\sum_i m_i$ is even, the ramification index is one and
the residue algebra is
$k_p(E)[w]/(w^2-\overline F_E)$.
This algebra is a quadratic field if
$\overline F_E\notin k_p(E)^{\times2}$ and is
$k_p(E)\times k_p(E)$ otherwise. This proves (2).

To prove \eqref{eq:fibres-explicit-residual-rank}, fix a component
$C\subset(X_0)_{k_p}$ mapping onto $E$, and denote the induced
morphism of smooth projective curves by
$\varphi_C:\widetilde C\to E$.
The extension $k_p(\widetilde C)/k_p(E)$ is Galois, with
$\Gal(k_p(\widetilde C)/k_p(E))$ a subgroup of $H/\langle b_E\rangle$.
For $D\in\mathcal D_E$, the inertia subgroup at
$D\cap E$ is generated by the image
$b_D+\langle b_E\rangle$ in $H/\langle b_E\rangle$.
A square class in the function field of
$E\simeq\mathbb P^1_{k_p}$ is determined by the parities of
its valuations at the points of $E$.
Consequently
\begin{equation}\label{eq:fibres-mechanical-step-15}
 \Gal(k_p(\widetilde C)/k_p(E))
 =\langle b_E,b_D:D\in\mathcal D_E\rangle/\langle b_E\rangle.
\end{equation}
Taking dimensions proves
\eqref{eq:fibres-explicit-residual-rank}.
The morphism $\varphi_C$ is ramified at exactly the points
$D\cap E$ with $b_D\notin\langle b_E\rangle$.
There are $B_E$ of these points, and every nontrivial inertia
subgroup has order two.
Riemann--Hurwitz for $\varphi_C$ gives
\begin{equation}\label{eq:fibres-explicit-odd-riemann-hurwitz}
 2g(\widetilde C)-2=2^{r_E}(-2+B_E/2).
\end{equation}
For $r_E>0$, solving
\eqref{eq:fibres-explicit-odd-riemann-hurwitz} gives
\eqref{eq:fibres-explicit-residual-genus}.
For $r_E=0$, $\varphi_C$ is an isomorphism and
$g(\widetilde C)=g(E)=0$.

The local rings in \eqref{eq:fibres-explicit-odd-local-rings}
were computed in Proposition~\ref{prop:starting-regularity}.
For the multiplicities in (3) and (4), retain the notation
$z\in D\cap D'$, $A,u,v,b,c$ of (3), ordering the divisors so
that $D$ is a component of $(P_p)_{k_p}$.
Since $(P_p)_{k_p}$ is reduced, one has
\begin{equation}\label{eq:fibres-explicit-base-fibre-equation}
 p=\eta u^\alpha v^\beta,\qquad \eta\in A^\times,\qquad
 (\alpha,\beta)=
 \begin{cases}
 (1,0),&D'=D_a\text{ for some }a,\\
 (1,1),&D'\text{ is a component of }(P_p)_{k_p}.
 \end{cases}
\end{equation}
Suppose $\lambda(b)=\lambda(c)=1$ and put
$R=A[z_1]/(z_1^2-uv)$.
On $U_u=\Spec R_u$, where $R_u=A[q]/(v-uq^2)$ as in
\eqref{eq:fibres-explicit-quadratic-u-open}, the image of $p$ is
\begin{equation}\label{eq:fibres-explicit-resolved-fibre-equation}
 p=\eta u^{\alpha+\beta}q^{2\beta}\quad\text{in }R_u.
\end{equation}
The exceptional divisor in $U_u$ is $V(u)$, and $q$ is a unit
at the generic point of $V(u)$. Thus the multiplicity of $V(u)$
in the fibre of $S_{-1}$ is $\alpha+\beta$, equal to one or two as asserted
in (3).
If $b,c$ are linearly independent, the integral closure of
$R_u$ in each field factor is $B_u$ from
\eqref{eq:fibres-explicit-normalized-u-open}.
Substituting $u=s^2$ gives
\begin{equation}\label{eq:fibres-explicit-normalized-fibre-equation}
 p=\eta s^{2(\alpha+\beta)}q^{2\beta}\quad\text{in }B_u.
\end{equation}
The exceptional divisor in $\Spec B_u$ is $V(s)$.
Its multiplicity is $2(\alpha+\beta)$, equal to two when
$D'=D_a$ and four when $D'$ is a fibre component.
On $U_v=\Spec R_v$, where $R_v=A[q']/(u-v(q')^2)$,
and on $\Spec B_v$ from
\eqref{eq:fibres-explicit-normalized-v-open}, one has
\begin{equation}\label{eq:fibres-explicit-v-open-fibre-equations}
 \begin{aligned}
 p&=\eta v^{\alpha+\beta}(q')^{2\alpha} &&\text{in }R_v,\\
 p&=\eta t^{2(\alpha+\beta)}(q')^{2\alpha} &&\text{in }B_v.
 \end{aligned}
\end{equation}
The open $U_{z_1}$ of
\eqref{eq:fibres-explicit-quadratic-z-open} is
$D(q)\subset U_u$ and $D(q')\subset U_v$.
Thus $U_u\cup U_v$ is the entire blow-up, so
\eqref{eq:fibres-explicit-resolved-fibre-equation}--
\eqref{eq:fibres-explicit-v-open-fibre-equations} compute every
exceptional component.
There are $2^{g-1}$ field factors when $b,c$ are independent,
and each factor gives one exceptional projective line.
This proves the count and multiplicities in the first case of (4).
If $b=c$, Proposition~\ref{prop:starting-regularity} shows that
$D'$ is a divisor $D_a$.
There are $2^g$ factors of the integral closure, all equal to
$R=A[z_1]/(z_1^2-uv)$.
The normalization of the blow-up in each field factor is therefore
the blow-up of $\Spec R$ itself.
Equation~\eqref{eq:fibres-explicit-resolved-fibre-equation} with
$(\alpha,\beta)=(1,0)$ gives multiplicity one for each of the
$2^g$ exceptional components. This proves the second case of (4).

We check the assertions about singular points of the reduced fibres.
In the regular rings $R_u,R_v,B_u,B_v$, the reduced fibre is
cut out by one parameter or the product of two distinct parameters,
as follows from
\eqref{eq:fibres-explicit-resolved-fibre-equation}--
\eqref{eq:fibres-explicit-v-open-fibre-equations}.
At points where $S_{-1}\to X_{-1}$ is an isomorphism, the
regular factors in \eqref{eq:fibres-explicit-odd-local-rings}
give the same descriptions: the reduced fibre is cut out by
$u$, $v$, $uv$, $s$, $v$, $sv$, or $st$, as appropriate.
Hence every singular point of
$(S_{-1})_{k_p,\mathrm{red}}$ or $(S_0)_{k_p,\mathrm{red}}$
has completed local ring $k_p[[z_1,z_2]]/(z_1z_2)$.
This proves that $(S_{-1})_{k_p,\mathrm{red}}$ and
$(S_0)_{k_p,\mathrm{red}}$ are nodal.

For completeness, the identifications of the local rings are
compatible with the global normalizations.
If $\Spec A'\subset P_m$ is affine and $f\in A'$, then
\begin{equation}\label{eq:fibres-mechanical-step-16}
 (\overline{A'}^{\,L})_f=\overline{A'_f}^{\,L}.
\end{equation}
The same identity holds for an affine open of $S_{-1}$.
Thus the descriptions of $X_0$ and $S_0$ on affine opens restrict
to identical integral closures on every smaller affine open.
The restrictions of $\sqrt{f_i}\in L$ agree under these
identifications. The automorphisms
$\sqrt{f_i}\mapsto(-1)^{a_i}\sqrt{f_i}$ commute with localization,
so the identifications also respect the actions of $H$ on $X_0$
and $G_0$ on $S_0$.

We finally prove the stronger assertion needed later: every
irreducible component of
$Y_p=(S_{-1})_{k_p,\mathrm{red}}$ is smooth, and the two
branches at a node of $Y_p$ belong to distinct irreducible
components of $Y_p$.
At a point where $S_{-1}\to X_{-1}$ is an isomorphism,
a node of $Y_p$ can occur only above an intersection $E\cap F$
of two distinct components of $(P_p)_{k_p}$, by
\eqref{eq:fibres-explicit-odd-local-rings}.
The two branches map onto $E$ and $F$, respectively.
An irreducible component of $Y_p$ cannot map onto two distinct
curves, so the branches belong to different components of $Y_p$.
Above $D_a\cap E$, the reduced fibre of $X_{-1}$ has only one
branch.

At a singular point of $X_{-1}$ with local ring
$R=A[z_1]/(z_1^2-uv)$, consider the blow-up
$\operatorname{Bl}_{(u,v,z_1)}(\Spec R)$.
In $U_u=\Spec R_u$, the reduced fibre is $V(u)$ if
$\beta=0$ and $V(uq)$ if $\beta=1$, by
\eqref{eq:fibres-explicit-resolved-fibre-equation}.
If $\beta=1$, $V(u)$ is the exceptional curve and $V(q)$ is
the strict transform of the fibre component of $\Spec R$
whose image on $\Spec A$ is $V(v)$.
The elements $u,q$ are regular parameters at $V(u,q)$, so
$V(u)$ and $V(q)$ are smooth there and have intersection
multiplicity one.
In $U_v=\Spec R_v$, the reduced fibre is $V(vq')$ by
\eqref{eq:fibres-explicit-v-open-fibre-equations}.
Here $V(v)$ is the exceptional curve and $V(q')$ is the strict
transform of the fibre component mapping onto $V(u)\subset\Spec A$.
The elements $v,q'$ are regular parameters at $V(v,q')$, so
$V(v)$ and $V(q')$ are smooth there and have intersection
multiplicity one.
The exceptional curve is contracted to the point $(u,v,z_1)$,
whereas either strict transform maps onto a curve of $\Spec R$.
Thus the exceptional curve and either strict transform are
contained in distinct irreducible components of $Y_p$.
All branches of every component of $Y_p$ are smooth by the local
rings above, and no component contains two branches at a node.
Therefore every component of $Y_p$ is smooth, and $Y_p$ is a
strict normal-crossing curve.
\end{proof}

\Needspace{16\baselineskip}
\begin{proposition}\label{prop:starting-dyadic-fibres}
Put $k_2=\overline{\mathbb F}_2$,
$r=2^{m-2}$ and $n=2r-1$.
The fibre $(P_2)_{k_2}$ has $n$ irreducible components in total,
namely $E_{k,c}$ for $0\leq k\leq m-2$ and $0\leq c<2^k$.
In particular, there are $r$ components $E_{m-2,c}$.

The curves $(S_{-1})_{k_2}$ and $(S_0)_{k_2}$ are reduced.
The curve $(S_{-1})_{k_2}$ has $2n$ irreducible components,
each with normalization isomorphic to $\mathbb P^1_{k_2}$.
The singular points of $(S_{-1})_{k_2}$ consist of
$2(n-1)$ ordinary nodes and $2r$ points at which two smooth
branches have intersection multiplicity two.

For $k<m-2$, there are $2^{g+1}$ irreducible components of
$(S_0)_{k_2}$ mapping onto $E_{k,c}$, each with normalization
isomorphic to $\mathbb P^1_{k_2}$.
For $k=m-2$, there are $2^{g-1}$ irreducible components of
$(S_0)_{k_2}$ mapping onto $E_{m-2,c}$, each with normalization
of genus one.
Thus, writing $c_0$ for the number of irreducible components
of $(S_0)_{k_2}$ and $a_0$ for the sum of the genera of their
normalizations, one has
\begin{equation}\label{eq:fibres-explicit-dyadic-component-counts}
 c_0=(r-1)2^{g+1}+r2^{g-1},\qquad a_0=r2^{g-1}.
\end{equation}
The singular points of $(S_0)_{k_2}$ consist of
$2^{g+1}(n-1)$ ordinary nodes and $r2^g$ points at which two
smooth branches have intersection multiplicity two.
\end{proposition}

\begin{proof}
For each component $E_{k,c}$ with $k<m-2$, the construction of
$P_2$ blows up two points of $E_{k,c}$, creating the components
$E_{k+1,c}$ and $E_{k+1,c+2^k}$.
The number of components indexed by a given $k$ is therefore
$2^k$, and the total number is
\begin{equation}\label{eq:fibres-mechanical-step-17}
 \sum_{k=0}^{m-2}2^k=2^{m-1}-1=2r-1=n.
\end{equation}
The graph $\Gamma_2$ defined in
\eqref{eq:start-explicit-fibre-graph} is a tree with $n$ vertices
and hence $n-1$ edges. Each edge of $\Gamma_2$ corresponds to
one node of $(P_2)_{k_2}$, so $(P_2)_{k_2}$ has $n-1$ nodes.

Let $\eta_{k,c}$ be the generic point of $E_{k,c}$.
For $k<m-2$, the residue of
$A_i=-2^{m-2}/(x-i)$ at $\eta_{k,c}$ is zero for every $i$.
Hence the polynomial $W^2+W-A_i$ has residue $W(W+1)$.
The residue algebra at $\eta_{k,c}$ of the normalization in $L$
is consequently
\begin{equation}\label{eq:fibres-mechanical-step-18}
 \prod_{h=1}^{2^{g+1}}k_2(E_{k,c}).
\end{equation}
Thus $(X_0)_{k_2}$ has $2^{g+1}$ components mapping onto
$E_{k,c}$, each with function field $k_2(E_{k,c})\simeq k_2(X)$.

Now fix $0\leq c<2^{m-2}$ and consider $E_{m-2,c}$.
Use the coordinate $X=(x-c)/2^{m-2}$ on $E_{m-2,c}$.
For $i=c+j2^{m-2}$ with $j=0,1,2,3$, one has
$A_i=-1/(X-j)$.
The residues of $A_{c}$ and $A_{c+2^{m-1}}$ are $1/X$,
and the residues of $A_{c+2^{m-2}}$ and
$A_{c+3\cdot2^{m-2}}$ are $1/(X-1)$.
For every other index $i$, the residue of $A_i$ is zero.
Define the additive subgroup
$\wp(k_2(X))=\{h^2+h:h\in k_2(X)\}\subset k_2(X)$.
The residues of $A_0,\ldots,A_g$ span in
$k_2(X)/\wp(k_2(X))$ the two classes represented by
\begin{equation}\label{eq:fibres-explicit-dyadic-residue-classes}
 \frac1X,\qquad\frac1{X-1}.
\end{equation}
Every nonzero $\mathbb F_2$-linear combination of $1/X$ and
$1/(X-1)$ has a pole of order one.
If $h\in k_2(X)$ has a pole of order $d>0$, then $h^2+h$ has
a pole of order $2d$ at the same point.
Therefore the classes of $1/X$ and $1/(X-1)$ in
$k_2(X)/\wp(k_2(X))$ are linearly independent.
For a component $C\subset(X_0)_{k_2}$ mapping onto
$E_{m-2,c}$, the normalization $\widetilde C$ consequently
has function field
\begin{equation}\label{eq:fibres-explicit-dyadic-residue-field}
 k_2(\widetilde C)=k_2(X)(w_0,w_1),\qquad
 w_0^2+w_0=1/X,\quad w_1^2+w_1=1/(X-1).
\end{equation}
The degree of $k_2(\widetilde C)/k_2(X)$ is four.
Denote the corresponding morphism by
$\varphi_C:\widetilde C\to\mathbb P^1_{k_2}$.
Over each of $X=0$ and $X=1$ there are two points of
$\widetilde C$, each of ramification index two.
To compute the different at a point above $X=0$, put $t=X$
and $w=w_0$; at a point above $X=1$, put $t=X-1$ and $w=w_1$.
The equations in \eqref{eq:fibres-explicit-dyadic-residue-field}
then give
\begin{equation}\label{eq:fibres-explicit-dyadic-ramified-parameter}
 w^2+w=1/t,\qquad a=tw,\qquad
 a^2+ta+t=0,\qquad t=\frac{a^2}{a+1}.
\end{equation}
The element $a$ is a uniformizer at the chosen point of
$\widetilde C$, and in characteristic two
\begin{equation}\label{eq:fibres-mechanical-step-19}
 \frac{dt}{da}=\frac{a^2}{(a+1)^2},\qquad
 \operatorname{ord}_a\left(\frac{dt}{da}\right)=2.
\end{equation}
Thus the different exponent is two at each of the four points
of $\varphi_C^{-1}(\{0,1\})$.
The morphism $\varphi_C$ is unramified outside $\{0,1\}$.
Riemann--Hurwitz gives
\begin{equation}\label{eq:fibres-explicit-dyadic-riemann-hurwitz}
 2g(\widetilde C)-2=4(-2)+4\cdot2=0.
\end{equation}
Hence $g(\widetilde C)=1$.
There are $2^{g+1}/4=2^{g-1}$ components of $(X_0)_{k_2}$
mapping onto $E_{m-2,c}$.

We now compute the components of $(X_{-1})_{k_2}$.
For each $E_{k,c}$, the sum of the residues of $A_0,\ldots,A_g$
is zero: all residues vanish when $k<m-2$, and each of the two
nonzero residues in
\eqref{eq:fibres-explicit-dyadic-residue-classes} occurs twice
when $k=m-2$.
Define
\begin{equation}\label{eq:fibres-explicit-dyadic-product-substitution}
 V=\frac{y}{\prod_{i=0}^g(x-i)}
   =\prod_{i=0}^g(1+2W_i),\qquad W=\frac{V-1}{2}.
\end{equation}
The relation $V^2=\prod_i(1+4A_i)$ gives
\begin{equation}\label{eq:fibres-explicit-dyadic-product-equation}
 W^2+W=\frac{\prod_{i=0}^g(1+4A_i)-1}{4}.
\end{equation}
At the generic point of every $E_{k,c}$, the right side of
\eqref{eq:fibres-explicit-dyadic-product-equation} has residue
$\sum_i\overline A_i=0$.
Thus the residue algebra of the normalization in $K$ is
$k_2(E_{k,c})[W]/(W(W+1))
 \simeq k_2(E_{k,c})\times k_2(E_{k,c})$.
There are exactly two components of $(X_{-1})_{k_2}$ mapping
onto each $E_{k,c}$, each with function field $k_2(E_{k,c})$.

At every $\eta_{k,c}$, the polynomial $P_i(W)=W^2+W-A_i$
has derivative $P_i'(W)=2W+1$, which is a unit in the integral
quadratic algebra over the discrete valuation ring at
$\eta_{k,c}$.
The extensions of these discrete valuation rings in $L$ and
$K$ are therefore unramified.
Every component of $(X_0)_{k_2}$ and $(X_{-1})_{k_2}$ has
multiplicity one.
The normal surfaces $X_0$ and $X_{-1}$ have Cohen--Macaulay
local rings, so their fibres have no embedded associated points.
The multiplicities equal to one therefore imply that
$(X_0)_{k_2}$ and $(X_{-1})_{k_2}$ are reduced.
The morphism $S_{-1}\to X_{-1}$ is an isomorphism over
$\Spec\mathbb Z_{(2)}$, and normalization in $L$ gives
\begin{equation}\label{eq:fibres-mechanical-step-20}
 S_0\times_{\Spec\mathbb Z}\Spec\mathbb Z_{(2)}
 \simeq X_0\times_{\Spec\mathbb Z}\Spec\mathbb Z_{(2)}.
\end{equation}
Consequently $(S_{-1})_{k_2}$ and $(S_0)_{k_2}$ are reduced,
and the component counts give
\eqref{eq:fibres-explicit-dyadic-component-counts}.

Let $z$ be a node of $(P_2)_{k_2}$ and let $A$ be the strict
henselization of the local ring of $P_2$ at $z$.
The calculation in Proposition~\ref{prop:starting-regularity}
gives
\begin{equation}\label{eq:fibres-mechanical-step-21}
 X_{-1}\times_{P_m}\Spec A\simeq\coprod_{h=1}^2\Spec A,
 \qquad
 X_0\times_{P_m}\Spec A\simeq
 \coprod_{h=1}^{2^{g+1}}\Spec A.
\end{equation}
Thus there are exactly two nodes of $(S_{-1})_{k_2}$ and
$2^{g+1}$ nodes of $(S_0)_{k_2}$ mapping to $z$.
Since $(P_2)_{k_2}$ has $n-1$ nodes, this gives
$2(n-1)$ nodes on $(S_{-1})_{k_2}$ and
$2^{g+1}(n-1)$ nodes on $(S_0)_{k_2}$.

It remains to consider the points where $D_a$ meets
$E_{m-2,c}$.
Fix $0\leq j<2^{m-1}$, put $c\equiv j\pmod{2^{m-2}}$
with $0\leq c<2^{m-2}$, and let $\zeta_j\in E_{m-2,c}$ be the
point through which the divisors $D_a$ with
$a\equiv j\pmod{2^{m-1}}$ pass.
Let $A$ be the strict henselization of $\mathcal O_{P_2,\zeta_j}$
and put $t=(x-j)/2^{m-2}$, so $\mathfrak m_A=(2,t)$.
Use the elements $U,V$ of
\eqref{eq:fibres-explicit-dyadic-two-extensions}, and put
\begin{equation}\label{eq:fibres-explicit-dyadic-quadratic-substitution}
 h=t^2-6t+4,\qquad Z=\frac{UV-h}{2}.
\end{equation}
The identity
\begin{equation}\label{eq:fibres-explicit-dyadic-product-identity}
 t(t-2)(t-4)(t-6)=h^2-16
\end{equation}
gives $Z^2+hZ+4=0$.
The integral closure of $A$ in
$K\otimes_{\mathbb Q(x)}\operatorname{Frac}(A)$ is
\begin{equation}\label{eq:fibres-explicit-dyadic-quadratic-ring}
 R_{-1}=A[Z]/(Z^2+hZ+4).
\end{equation}
Indeed, the map $Z\mapsto(UV-h)/2$ identifies
$R_{-1}\otimes_A\operatorname{Frac}(A)$ with
$K\otimes_{\mathbb Q(x)}\operatorname{Frac}(A)$.
To prove that $R_{-1}$ is normal, we verify Serre's criterion.
The ring $R_{-1}$ is a hypersurface in the regular ring $A[Z]$,
so $R_{-1}$ is Cohen--Macaulay.
After inverting $2$, the substitution $T_1=2Z+h$ gives
\begin{equation}\label{eq:fibres-mechanical-horizontal-quadratic-ring}
 R_{-1}[1/2]\simeq
 A[1/2][T_1]/(T_1^2-t(t-2)(t-4)(t-6)).
\end{equation}
The polynomial $t(t-2)(t-4)(t-6)$ has distinct roots in
characteristic zero. The equation in
\eqref{eq:fibres-mechanical-horizontal-quadratic-ring}
therefore defines a regular curve after the \'etale base changes
used to form $A$.
Hence every localization of $R_{-1}$ at a horizontal
height-one prime is regular.
For $Q(Z)=Z^2+hZ+4$, the derivative $Q'(Z)=2Z+h$ has
residue $t^2$ at each generic point of the fibre at $2$.
Since $t^2\ne0$ in $k_2(t)$, $Q'(Z)$ is a unit at every
height-one prime of $R_{-1}$ containing $2$.
For every height-one prime $\mathfrak p\subset R_{-1}$
containing $2$, the ring $(R_{-1})_{\mathfrak p}$ is therefore \'etale over a
discrete valuation ring and are regular.
Serre's criterion proves normality of $R_{-1}$.
Since $S_{-1}\to X_{-1}$ is an isomorphism over
$\Spec\mathbb Z_{(2)}$, one has
$S_{-1}\times_{P_m}\Spec A=\Spec R_{-1}$.

Reduction of \eqref{eq:fibres-explicit-dyadic-quadratic-ring}
modulo $2$ gives
\begin{equation}\label{eq:fibres-explicit-dyadic-quadratic-fibre}
 R_{-1}/2R_{-1}=(A/2A)[Z]/(Z(Z+t^2)).
\end{equation}
There is exactly one point of $\Spec(R_{-1}/2R_{-1})$
above $\zeta_j$, namely the point with maximal ideal $(t,Z)$.
The completion of $(R_{-1}/2R_{-1})_{(t,Z)}$ is
\begin{equation}\label{eq:fibres-mechanical-step-22}
 k_2[[t,Z]]/(Z(Z+t^2)).
\end{equation}
The quotient by either minimal prime $(Z)$ or $(Z+t^2)$ is
$k_2[[t]]$, so the branches defined by $(Z)$ and $(Z+t^2)$
are smooth.
The intersection multiplicity of the branches $(Z)$ and
$(Z+t^2)$ is
\begin{equation}\label{eq:fibres-mechanical-step-23}
 \dim_{k_2}k_2[[t,Z]]/(Z,Z+t^2)
 =\dim_{k_2}k_2[[t]]/(t^2)=2.
\end{equation}
For each $c$, the points $\zeta_j$ on $E_{m-2,c}$ are exactly
$X=0$ and $X=1$.
There are $r$ choices of $c$, so this gives $2r$ points of
$(S_{-1})_{k_2}$ with two smooth branches of intersection
multiplicity two.

For $S_0$, use the regular ring $B=B_{\mathfrak n}$ from
\eqref{eq:fibres-explicit-dyadic-rank-four-algebra}.
Reducing \eqref{eq:fibres-explicit-dyadic-parameters} modulo
$2$ gives
\begin{equation}\label{eq:fibres-explicit-dyadic-rank-four-fibre}
 d(d+t)=0,\qquad t=\frac{a^2}{a+1}\quad\text{in }B/2B.
\end{equation}
Thus the completed local ring of $B/2B$ at $(a,d)$ is
\begin{equation}\label{eq:fibres-mechanical-step-24}
 k_2[[a,d]]/
 \left(d\left(d+\frac{a^2}{a+1}\right)\right).
\end{equation}
The quotient by either minimal prime $(d)$ or
$(d+a^2/(a+1))$ is $k_2[[a]]$.
The branches defined by $(d)$ and $(d+a^2/(a+1))$ are smooth,
and the intersection multiplicity of these branches is
\begin{equation}\label{eq:fibres-mechanical-step-25}
 \dim_{k_2}k_2[[a,d]]/(d,d+a^2/(a+1))
 =\dim_{k_2}k_2[[a]]/(a^2)=2.
\end{equation}
Substituting $U=t+2a$, $V=t-2+2b$ into
\eqref{eq:fibres-explicit-dyadic-quadratic-substitution} gives
\begin{equation}\label{eq:fibres-mechanical-step-26}
 Z=2t+t(a+b)-2a+2ab-2\equiv td\pmod{2B}.
\end{equation}
Hence the homomorphism on the reduced fibres induced by $\Pi_0$
is
\begin{equation}\label{eq:fibres-mechanical-step-27}
 R_{-1}/2R_{-1}\longrightarrow B/2B,
 \qquad t\longmapsto\frac{a^2}{a+1},\quad Z\longmapsto td.
\end{equation}
The branch $d=0$ maps to $Z=0$, and the branch $d=t$ maps
to $Z=t^2$; on either branch the map of parameters is
$t\mapsto a^2/(a+1)$.

By Proposition~\ref{prop:starting-regularity}, the integral
closure of $A$ in
$L\otimes_{\mathbb Q(x)}\operatorname{Frac}(A)$ is a product
of $2^{g-1}$ copies of $B$.
Thus there are $2^{g-1}$ points of $(S_0)_{k_2}$ mapping to
$\zeta_j$, each with completed local ring
$k_2[[a,d]]/(d(d+a^2/(a+1)))$.
There are $2r$ choices of $\zeta_j$, giving
$2r\cdot2^{g-1}=r2^g$ points of $(S_0)_{k_2}$ where two
smooth branches have intersection multiplicity two.
Finally, let $z\in(P_2)_{k_2}$ be a point which is neither
a node nor one of the $\zeta_j$.
Proposition~\ref{prop:starting-regularity} identifies the
strict henselian local rings of $S_{-1}$ and $S_0$ above $z$
with the strict henselian local ring of $P_2$ at $z$.
Since $(P_2)_{k_2}$ is smooth at $z$, the curves
$(S_{-1})_{k_2}$ and $(S_0)_{k_2}$ are smooth at every point
mapping to $z$.
This proves that the nodes and the points above the $\zeta_j$
listed in the proposition exhaust the singular points of
$(S_{-1})_{k_2}$ and $(S_0)_{k_2}$.
\end{proof}

The ring $R_{-1}$ in
\eqref{eq:fibres-explicit-dyadic-quadratic-ring} also describes
the singularities of the surface $S_{-1}$ itself.
At the point with maximal ideal $\mathfrak m=(2,t,Z)$,
\begin{equation}\label{eq:fibres-mechanical-step-28}
 \dim R_{-1,\mathfrak m}=2,\qquad
 \dim_{k_2}\mathfrak m/\mathfrak m^2=3.
\end{equation}
Thus $S_{-1}$ is singular at every point above a $\zeta_j$.
Every point of $S_0$ above a $\zeta_j$ has regular
local ring $B_{\mathfrak n}$ with parameters $a,d$, by
\eqref{eq:fibres-explicit-dyadic-parameters}.
\section{Realizable extensions}\label{sec:realizable}

Let $\widetilde S\to S_0$ be a connected finite \'etale cover, and
suppose that $\widetilde S\to S_{-1}$ is Galois with group $G$.
We ask which central extensions of $G$ by $\F$ occur as
$\Gal(T/S_{-1})$ for a finite \'etale double cover
$T\to\widetilde S$. We call a class in $H^2(G,\F)$ realizable when
the associated extension can be obtained from a cover in the precise
sense defined below. We will bound the codimension of the realizable
classes in terms of the morphism $q:S_0\to S_{-1}$.

\subsection{The group of realizable extensions}\label{sec:realizable-extensions}

Throughout this section, $g\ge15$ and $N=2g+2$.
We allow $q:S_0\to S_{-1}$ to be any morphism satisfying the conditions
below. Proposition~\ref{prop:starting-anchor} will verify the conditions
for the morphism constructed in Section~\ref{sec:starting-surface}.

We call a diagram of finite surjective morphisms of arithmetic surfaces
\begin{equation}\label{eq:anchored-diagram}
 \widetilde S\xrightarrow{\ h\ } S_0\xrightarrow{\ q\ }S_{-1}
\end{equation}
an \emph{anchored diagram} if the following conditions hold.
\begin{enumerate}[label=(\roman*)]
\item The surface $S_{-1}$ is normal and projective over $\Spec\Z$.
The generic fibre $C=(S_{-1})_{\Q}$ is smooth and geometrically
integral, with $g(C)=g$. The surface $S_0$ is regular and projective
over $\Spec\Z$, the curve $(S_0)_{\Q}$ is smooth and geometrically
integral, and there is a section $s_0:\Spec\Z\to S_0$.
In particular, $S_{-1}$ and $S_0$ are integral and flat over $\Spec\Z$,
as required by the definition of an arithmetic surface in
Section~\ref{sec:dictionary}.
\item The morphism $q$ is Galois with group
$A=\Gal(S_0/S_{-1})\simeq\F^g$, and
$q_{\Q}:(S_0)_{\Q}\to C$ is \'etale.
The scheme $\widetilde S$ is connected, the morphism $h$ is \'etale,
and $q\circ h:\widetilde S\to S_{-1}$ is Galois, with group
$G=\Gal(\widetilde S/S_{-1})$.
The term Galois refers to the normalization in a Galois extension of
function fields, as defined in Section~\ref{sec:dictionary}.
\item There is a set
$\Sigma\subseteq\{p:p\text{ prime},\ p\le2g+1\}$, with $2\in\Sigma$,
such that, for
\begin{equation}\label{eq:real-explicit-good-base}
 U=\Spec\Z[1/\prod_{p\in\Sigma}p],
\end{equation}
the morphism $(S_{-1})_U\to U$ is smooth with geometrically connected
fibres and $q_U:(S_0)_U\to(S_{-1})_U$ is an \'etale $A$-torsor.
Furthermore,
\begin{equation}\label{eq:real-explicit-rational-two-torsion}
 J(C)(\overline\Q)[2]=J(C)(\Q)[2].
\end{equation}
\item For each odd $p\in\Sigma$, the surface $S_{-1}$ is regular on
an open neighbourhood of $(S_{-1})_{\mathbb F_p}$. The curve
$Y_p=(S_{-1})_{\overline{\mathbb F}_p,\mathrm{red}}$ has strict normal
crossings, in the sense of Section~\ref{sec:dictionary}.
Let $c_p$ be the number of irreducible components of $Y_p$, let
$s_p=|\operatorname{Sing}(Y_p)|$, and write
$\nu_p:\coprod_{i\in I_p}\widetilde Y_{p,i}\to Y_p$ for the
normalization, with $\widetilde Y_{p,i}$ connected. Put
$a_p=\sum_{i\in I_p}g(\widetilde Y_{p,i})$. Then
\begin{equation}\label{eq:real-explicit-fibre-bounds}
 c_p\le7N,\qquad s_p\le8N,\qquad a_p\le2N.
\end{equation}
\item The curve
$Y_2=(S_{-1})_{\overline{\mathbb F}_2,\mathrm{red}}$ is connected.
Every connected component of the normalization of $Y_2$ is isomorphic
to $\mathbb P^1_{\overline{\mathbb F}_2}$, and every singular point
of $Y_2$ has two branches. For the graph $\Gamma(Y_2)$ of
\eqref{eq:bg-explicit-incidence-graph},
\begin{equation}\label{eq:real-explicit-dyadic-hypotheses}
 b_1(\Gamma(Y_2))\le\frac{g-1}{2}.
\end{equation}
There is a set $B_2\subseteq Y_2(\overline{\mathbb F}_2)$ with
$|B_2|\le(g+1)/2$ such that the stabilizer
$A_{\bar x}=\{a\in A:a\bar x=\bar x\}$ of each
$\bar x\in S_0(\overline{\mathbb F}_2)$ satisfies
\begin{equation}\label{eq:real-explicit-stabilizer-hypotheses}
 A_{\bar x}=\{0\}\quad\text{if }q(\bar x)\notin B_2,
 \qquad |A_{\bar x}|\le2\quad\text{if }q(\bar x)\in B_2.
\end{equation}
\end{enumerate}
The first Betti number in \eqref{eq:real-explicit-dyadic-hypotheses} is
as defined in \eqref{eq:bg-explicit-graph-betti}.
Except for the conditions on $\widetilde S$ and $h$ in (ii), the
conditions (i)--(v) concern $q:S_0\to S_{-1}$ alone.

Fix an anchored diagram
\begin{equation}\label{eq:realization-given-tower}
 \widetilde S\xrightarrow{\ h\ }S_0\xrightarrow{\ q\ }S_{-1},
 \qquad G=\Gal(\widetilde S/S_{-1}),
\end{equation}
and give $\F$ the trivial $G$-action. A class
$\theta\in H^2(G,\F)$ corresponds to a central extension
\begin{equation}\label{eq:realization-central-extension}
 1\longrightarrow\F\xrightarrow{\ j_\theta\ }E_\theta
   \xrightarrow{\ \rho_\theta\ }G\longrightarrow1.
\end{equation}
We call $\theta$ \emph{realizable} if there is a connected finite
\'etale cover $r:T\to\widetilde S$ such that
$q\circ h\circ r:T\to S_{-1}$ is Galois and, on putting
$H=\Gal(T/S_{-1})$, there is an injective homomorphism
\begin{equation}\label{eq:real-explicit-realization-embedding}
 \iota:H\hookrightarrow E_\theta,
 \qquad
 \rho_\theta\circ\iota
   =\operatorname{res}_{\widetilde S}:H\longrightarrow G.
\end{equation}
Here $\operatorname{res}_{\widetilde S}$ is restriction of automorphisms
from $\Q(T)$ to $\Q(\widetilde S)$.
We write
\begin{equation}\label{eq:real-explicit-realizable-classes}
 \mathscr E(\widetilde S,S_{-1})
   =\{\theta\in H^2(G,\F):\theta\text{ is realizable}\}.
\end{equation}
The morphisms $r$ and $q\circ h\circ r$ are surjective, and
$\operatorname{res}_{\widetilde S}(H)=G$. We do not require
$\iota(H)=E_\theta$. In particular, the zero class is realizable by
\begin{equation}\label{eq:real-explicit-zero-realization}
 E_0=G\times\F,\qquad T=\widetilde S,\qquad
 r=\mathrm{id}_{\widetilde S},\qquad \iota(g)=(g,0).
\end{equation}

If $\theta\ne0$, every homomorphism $\iota$ in
\eqref{eq:real-explicit-realization-embedding} is surjective.
Indeed, put $H'=\iota(H)$ and $K_\theta=\ker\rho_\theta=j_\theta(\F)$.
Since $|K_\theta|=2$, either $H'\cap K_\theta=K_\theta$ or
$H'\cap K_\theta=\{1\}$. If $H'\cap K_\theta=K_\theta$, then
$\rho_\theta(H')=G$ implies $H'=E_\theta$. If
$H'\cap K_\theta=\{1\}$, then
$\rho_\theta|_{H'}:H'\to G$ is an isomorphism; the inverse
$G\to H'\subseteq E_\theta$ splits
\eqref{eq:realization-central-extension}, contrary to $\theta\ne0$.
Consequently,
\begin{equation}\label{eq:real-explicit-realization-degree}
 [\Q(T):\Q(\widetilde S)]\in\{1,2\},\qquad
 [\Q(T):\Q(\widetilde S)]=2\quad\text{if }\theta\ne0.
\end{equation}
The definition of $\mathscr E(\widetilde S,S_{-1})$ depends only on
$\widetilde S\to S_{-1}$. For every realizing cover $r:T\to\widetilde S$,
the morphism $h\circ r:T\to S_0$ is a connected finite \'etale cover.
Proposition~\ref{prop:realizing-surface-properties} therefore shows that
$T$ is an integral regular projective flat arithmetic surface,
$T_{\Q}$ is smooth and geometrically integral, and
$T\to\Spec\Z$ has a section.

Choose a normalized cocycle $c_\theta:G\times G\to\F$ representing
$\theta$. Thus, for every $g,h,k\in G$,
\begin{equation}\label{eq:real-explicit-cocycle-identities}
 c_\theta(1,g)=c_\theta(g,1)=0,\qquad
 c_\theta(g,h)+c_\theta(gh,k)
   =c_\theta(h,k)+c_\theta(g,hk).
\end{equation}
We may take $E_\theta=G\times\F$ with multiplication
\begin{equation}\label{eq:realization-cocycle-product}
 (g,a)(h,b)=(gh,a+b+c_\theta(g,h)),
 \qquad j_\theta(a)=(1,a),\quad \rho_\theta(g,a)=g.
\end{equation}
For $b:G\to\F$ with $b(1)=0$, put
$\delta b(g,h)=b(g)+b(h)+b(gh)$.
The map
\begin{equation}\label{eq:real-explicit-change-cocycle}
 (G\times\F,*_{c_\theta})\longrightarrow
 (G\times\F,*_{c_\theta+\delta b}),\qquad
 (g,a)\longmapsto(g,a+b(g)),
\end{equation}
is a group isomorphism inducing the identity on $\F$ and on $G$.
Composing $\iota$ with \eqref{eq:real-explicit-change-cocycle} shows
that realizability depends only on the class $\theta$, and not on
the chosen cocycle $c_\theta$.

\begin{proposition}\label{prop:realizable-subspace}
Let $\widetilde S\to S_0\to S_{-1}$ be an anchored diagram as in
\eqref{eq:realization-given-tower}, with
$G=\Gal(\widetilde S/S_{-1})$. Then
$\mathscr E(\widetilde S,S_{-1})$ is an $\F$-linear subspace of
$H^2(G,\F)$.
\end{proposition}

\begin{proof}
Equation~\eqref{eq:real-explicit-zero-realization} gives
$0\in\mathscr E(\widetilde S,S_{-1})$. Let
$\theta_1,\theta_2\in\mathscr E(\widetilde S,S_{-1})$.
If $\theta_1=0$ or $\theta_2=0$, then
$\theta_1+\theta_2\in\{\theta_1,\theta_2\}$ is realizable.
We may therefore assume $\theta_1\ne0$ and $\theta_2\ne0$.
For $i\in\{1,2\}$ choose a realizing cover
$r_i:T_i\to\widetilde S$, identify $\Gal(T_i/S_{-1})$ with
$E_i=E_{\theta_i}$, and write
$j_i:\F\to E_i$ and $\rho_i:E_i\to G$ for the homomorphisms in
\eqref{eq:realization-central-extension}.

Define the fibre product $W$ and the morphism $\pi$ by
\begin{equation}\label{eq:real-explicit-cover-fibre-product}
 W=T_1\times_{\widetilde S}T_2,\qquad
 \pi=r_1\circ\operatorname{pr}_{T_1}
    =r_2\circ\operatorname{pr}_{T_2}:W\to\widetilde S.
\end{equation}
Since $r_1$ and $r_2$ are finite \'etale of degree two, $\pi$ is
finite \'etale of degree four. The group
\begin{equation}\label{eq:real-explicit-group-fibre-product}
 F=E_1\times_G E_2
   =\{(e_1,e_2)\in E_1\times E_2:
                       \rho_1(e_1)=\rho_2(e_2)\}
\end{equation}
acts on $W$ by $(e_1,e_2)(t_1,t_2)=(e_1t_1,e_2t_2)$.
Indeed, if $r_1(t_1)=r_2(t_2)$ and $(e_1,e_2)\in F$, then
\begin{equation}\label{eq:real-explicit-action-fibre-product}
 r_1(e_1t_1)=\rho_1(e_1)r_1(t_1)
           =\rho_2(e_2)r_2(t_2)=r_2(e_2t_2).
\end{equation}
The subgroup
\begin{equation}\label{eq:real-explicit-diagonal-subgroup}
 \Delta=\{(j_1(a),j_2(a)):a\in\F\}\subseteq F
\end{equation}
is central. The nonidentity element of $\Delta$ has no fixed
geometric point of $W$, because $j_i(1)$ has no fixed geometric
point of the \'etale double cover $r_i:T_i\to\widetilde S$.
Define
\begin{equation}\label{eq:real-explicit-quotient-cover}
 Z=W/\Delta,\qquad p_Z:Z\to\widetilde S,
\end{equation}
by the relative spectrum
\begin{equation}\label{eq:real-explicit-quotient-algebra}
 Z=\operatorname{Spec}_{\widetilde S}
       \bigl((\pi_*\mathcal O_W)^\Delta\bigr).
\end{equation}
We verify that $p_Z$ is finite \'etale of degree two.

Choose an \'etale covering $\{V_\alpha\to\widetilde S\}_\alpha$
and isomorphisms
$T_i\times_{\widetilde S}V_\alpha\simeq V_\alpha\times\F$
for $i=1,2$, compatible with the actions of $j_i(\F)$.
Here $V_\alpha\times\F$ denotes the disjoint union of two copies
of $V_\alpha$. Then
\begin{equation}\label{eq:real-explicit-quotient-trivializations}
 W\times_{\widetilde S}V_\alpha
       \simeq V_\alpha\times(\F\times\F),\qquad
 (j_1(1),j_2(1)):(a,b)\longmapsto(a+1,b+1).
\end{equation}
The two orbits of the action in
\eqref{eq:real-explicit-quotient-trivializations} are indexed by
$a+b\in\F$. Consequently,
\begin{equation}\label{eq:real-explicit-local-invariant-algebra}
 ((\pi_*\mathcal O_W)^\Delta)|_{V_\alpha}
       \simeq\mathcal O_{V_\alpha}\times\mathcal O_{V_\alpha},
 \qquad Z\times_{\widetilde S}V_\alpha\simeq V_\alpha\times\F.
\end{equation}
For completeness, on
$V_{\alpha\beta}=V_\alpha\times_{\widetilde S}V_\beta$, the
transition maps for $T_i$ have the form
$a\mapsto a+\epsilon_{i,\alpha\beta}$, where
$\epsilon_{i,\alpha\beta}:V_{\alpha\beta}\to\F$ is locally constant.
The induced transition map for $a+b$ is
\begin{equation}\label{eq:real-explicit-quotient-transition}
 u\longmapsto u+\epsilon_{1,\alpha\beta}
                  +\epsilon_{2,\alpha\beta}.
\end{equation}
On $V_\alpha\times_{\widetilde S}V_\beta
\times_{\widetilde S}V_\gamma$, the transition maps in
\eqref{eq:real-explicit-quotient-transition} satisfy the cocycle
identity because the transition maps of $T_1$ and $T_2$ do.
Thus \eqref{eq:real-explicit-local-invariant-algebra} identifies
$p_Z$ \'etale locally with the projection
$V_\alpha\times\F\to V_\alpha$, proving that $p_Z$ is finite
\'etale of degree two (\st{03RV}). The argument applies in every
residue characteristic.

Since $\Delta$ is central in $F$, the action of $F$ on $W$ induces
an action on $Z$ of
\begin{equation}\label{eq:real-explicit-quotient-group}
 E=F/\Delta.
\end{equation}
The homomorphism
$\rho:E\to G$, defined by
$\rho([(e_1,e_2)])=\rho_1(e_1)=\rho_2(e_2)$, has central kernel
\begin{equation}\label{eq:real-explicit-baer-kernel}
 \ker\rho=(j_1(\F)\times j_2(\F))/\Delta\simeq\F,
 \qquad [(j_1(a),j_2(b))]\longmapsto a+b.
\end{equation}
To compute the class of the extension, choose cocycles
$c_i=c_{\theta_i}$ and use
\eqref{eq:realization-cocycle-product} to identify $F$ with
$G\times\F\times\F$. The multiplication in $F$ is
\begin{equation}\label{eq:real-explicit-fibre-product-multiplication}
 (g,a,b)(h,a',b')
   =(gh,a+a'+c_1(g,h),b+b'+c_2(g,h)).
\end{equation}
The map $(g,a,b)\mapsto(g,a+b)$ identifies $E$ with
$G\times\F$ equipped with multiplication
\begin{equation}\label{eq:real-explicit-baer-multiplication}
 (g,u)(h,v)=(gh,u+v+c_1(g,h)+c_2(g,h)).
\end{equation}
It follows that the extension
\begin{equation}\label{eq:real-explicit-baer-extension}
 1\longrightarrow\F\longrightarrow E\xrightarrow{\ \rho\ }G
   \longrightarrow1
\end{equation}
represents the Baer sum $\theta_1+\theta_2$.

Let $\bar\eta$ be a geometric generic point of $S_{-1}$.
The action of $E_i$ on $(T_i)_{\bar\eta}$ is simply transitive.
By the definition of $F$ in
\eqref{eq:real-explicit-group-fibre-product}, the action of $F$ on
\begin{equation}\label{eq:real-explicit-generic-fibre-pairs}
 W_{\bar\eta}
   =\{(t_1,t_2)\in(T_1)_{\bar\eta}\times(T_2)_{\bar\eta}:
                                      r_1(t_1)=r_2(t_2)\}
\end{equation}
is simply transitive. Therefore the action of $E=F/\Delta$ on
$Z_{\bar\eta}=W_{\bar\eta}/\Delta$ is simply transitive.

If $\theta_1+\theta_2=0$, equation
\eqref{eq:real-explicit-zero-realization} already realizes the sum.
Suppose $\theta_1+\theta_2\ne0$. We claim that $Z$ is connected.
If $Z$ were disconnected, then, because $p_Z$ is finite \'etale of
degree two and $\widetilde S$ is connected, we could write
\begin{equation}\label{eq:real-explicit-disconnected-quotient}
 Z=Z_1\amalg Z_2,\qquad
 p_Z|_{Z_i}:Z_i\xrightarrow{\sim}\widetilde S\quad(i=1,2).
\end{equation}
Let $\tau$ be the nonidentity element of $\ker\rho$.
Equation~\eqref{eq:real-explicit-local-invariant-algebra} shows that
$\tau$ acts on each fibre of $p_Z$ by interchanging the two points.
Hence $\tau(Z_1)=Z_2$ and $\tau(Z_2)=Z_1$.
Put $E_{Z_1}=\{e\in E:e(Z_1)=Z_1\}$.
Given $g\in G$, choose $e\in E$ with $\rho(e)=g$. Either
$e(Z_1)=Z_1$ or $(\tau e)(Z_1)=Z_1$, so
$\rho(E_{Z_1})=G$. Moreover,
$E_{Z_1}\cap\ker\rho=\{1\}$, since $\tau(Z_1)=Z_2$.
Thus $\rho|_{E_{Z_1}}:E_{Z_1}\to G$ is an isomorphism, and
\begin{equation}\label{eq:real-explicit-forbidden-splitting}
 G\xrightarrow{(\rho|_{E_{Z_1}})^{-1}}E_{Z_1}\hookrightarrow E
\end{equation}
splits \eqref{eq:real-explicit-baer-extension}. This contradicts
$\theta_1+\theta_2\ne0$ and proves that $Z$ is connected.

The scheme $\widetilde S$ is regular because $h:\widetilde S\to S_0$
is \'etale and $S_0$ is regular. The scheme $Z$ is regular because
$p_Z$ is \'etale. Since $Z$ is connected and regular, $Z$ is integral.
The simply transitive action on $Z_{\bar\eta}$ identifies
$\Gal(\Q(Z)/\Q(S_{-1}))$ with $E$.
Furthermore, $Z\to S_{-1}$ is finite, and $Z$ is normal; hence $Z$
is the normalization of $S_{-1}$ in $\Q(Z)$.
The cover $p_Z:Z\to\widetilde S$ therefore realizes
$\theta_1+\theta_2$. We have proved that
$\mathscr E(\widetilde S,S_{-1})$ contains zero and is closed under
addition. Since $\F=\mathbb F_2$, it is an $\F$-linear subspace.
\end{proof}

We define the \emph{obstruction group} by
\begin{equation}\label{eq:realizable-obstruction-space}
 \mathcal O(\widetilde S,S_{-1})
   =H^2(G,\F)/\mathscr E(\widetilde S,S_{-1}).
\end{equation}

We next prove that a connected finite \'etale cover $T\to S_0$
satisfies the geometric properties required of a realizing surface.
We will use $s_0:\Spec\Z\to S_0$ to construct a section of
$T\to\Spec\Z$. We will then deduce the geometric integrality of
$T_{\Q}$ from the existence of a point in $T_{\Q}(\Q)$.

\begin{proposition}\label{prop:realizing-surface-properties}
Let $S_0\to\Spec\Z$ be an integral regular projective flat arithmetic
surface with smooth geometrically integral generic fibre and a section
$s_0$. Let $f:T\to S_0$ be a connected finite \'etale cover of degree $d$.
Then $T$ is an integral regular projective flat arithmetic surface with
smooth geometrically integral generic fibre. Moreover,
\begin{equation}\label{eq:real-explicit-section-pullback}
 T\times_{S_0,s_0}\Spec\Z
    \simeq\coprod_{i=1}^{d}\Spec\Z.
\end{equation}
In particular $T\to\Spec\Z$ has a section.
\end{proposition}

\begin{proof}
Since $S_0$ is regular and $f$ is \'etale, $T$ is regular.
The irreducible components of a regular noetherian scheme are disjoint
and open. The connectedness of $T$ therefore implies that $T$ is
integral. The finite morphism $f:T\to S_0$ is projective; composing
$f$ with the projective morphism $S_0\to\Spec\Z$ proves that
$T\to\Spec\Z$ is projective. Since $f$ and $S_0\to\Spec\Z$ are flat,
the composite $T\to\Spec\Z$ is flat. For every prime $p$, the map
$f_{\mathbb F_p}:T_{\mathbb F_p}\to(S_0)_{\mathbb F_p}$ is finite
\'etale and surjective, so every irreducible component of
$T_{\mathbb F_p}$ has dimension one. Also,
$f_{\Q}:T_{\Q}\to(S_0)_{\Q}$ is \'etale and $(S_0)_{\Q}$ is smooth.
It follows that $T_{\Q}$ is an integral smooth projective curve.

Put $P=T\times_{S_0,s_0}\Spec\Z$ and denote the projections by
\begin{equation}\label{eq:real-explicit-section-projections}
 \operatorname{pr}_T:P\to T,\qquad
 \operatorname{pr}_{\Z}:P\to\Spec\Z.
\end{equation}
The morphism $\operatorname{pr}_{\Z}$ is finite \'etale of degree $d$.
Write $P=\coprod_{i\in I}P_i$ for the decomposition into connected
components. Each $P_i$ is isomorphic to $\Spec\mathcal O_{K_i}$
for a number field $K_i$ unramified at every finite prime.
Consequently $|\operatorname{disc}(K_i)|=1$, and Minkowski's
discriminant bound gives $K_i=\Q$
\cite[Theorem~4.9 and its proof]{MilneANT}.
Thus $\operatorname{pr}_{\Z}|_{P_i}:P_i\to\Spec\Z$ is an isomorphism
for every $i\in I$, and $|I|=d$. This proves
\eqref{eq:real-explicit-section-pullback}.
For each $i\in I$, define
\begin{equation}\label{eq:real-explicit-lifted-section}
 s_i=\operatorname{pr}_T|_{P_i}\circ
       (\operatorname{pr}_{\Z}|_{P_i})^{-1}:\Spec\Z\to T.
\end{equation}
If $\tau:T\to\Spec\Z$ is the structure morphism, then
$\tau\circ s_i=\mathrm{id}_{\Spec\Z}$ and $f\circ s_i=s_0$.
In particular, $(s_i)_{\Q}$ is a point of $T_{\Q}(\Q)$.

Let $\mathcal C$ be the set of connected components of
$T_{\overline\Q}$. Since $T_{\Q}$ is connected,
$\Gal(\overline\Q/\Q)$ acts transitively on $\mathcal C$.
Choose $i\in I$, and let $C_i\in\mathcal C$ be the component
containing the geometric point defined by $(s_i)_{\Q}$.
The rationality of $(s_i)_{\Q}$ gives
$\sigma(C_i)=C_i$ for every $\sigma\in\Gal(\overline\Q/\Q)$.
Transitivity now implies $\mathcal C=\{C_i\}$.
Since $T_{\overline\Q}$ is smooth, connectedness implies integrality.
Thus $T_{\Q}$ is geometrically integral.
\end{proof}

The proposition applies to $h:\widetilde S\to S_0$ and to
$h\circ r:T\to S_0$ for every realizing cover $r:T\to\widetilde S$.

We now verify conditions (i)--(v) for the morphism
$q:S_0\to S_{-1}$ of Section~\ref{sec:starting-surface}.
In particular, we must prove
$J(C_{-1})(\overline\Q)[2]=J(C_{-1})(\Q)[2]$, bound
$c_p,s_p,a_p$ for odd $p\in\Sigma$, and bound
$b_1(\Gamma(Y_2))$ and the stabilizers $A_{\bar x}$ at $p=2$.

\begin{proposition}\label{prop:starting-anchor}
Let $S_0\to S_{-1}$ be the morphism of
Section~\ref{sec:starting-surface}, with $g=2^m-1$ and $m\ge4$.
If $\widetilde S\to S_0$ is a connected finite \'etale cover and
$\widetilde S\to S_{-1}$ is Galois, then
$\widetilde S\to S_0\to S_{-1}$ is an anchored diagram.
\end{proposition}

\begin{proof}
Propositions~\ref{prop:start-models},
\ref{prop:start-integral-galois}, \ref{prop:starting-regularity} and
\ref{prop:start-integral-section} prove (i), construct the section
$s_0:\Spec\Z\to S_0$, and show that
$q:S_0\to S_{-1}$ is Galois with group $A=G_0\simeq\F^g$ and
$q_{\Q}$ is \'etale. The hypotheses on $\widetilde S\to S_0$ and
$\widetilde S\to S_{-1}$ give the remaining assertions in (ii).
Take $\Sigma=\{p:p\text{ prime},\ p\le2g+1\}$ and let $U$ be as in
\eqref{eq:real-explicit-good-base}.
Proposition~\ref{prop:start-good-fibres} shows that
$(S_{-1})_U\to U$ is smooth with geometrically connected fibres and
$q_U:(S_0)_U\to(S_{-1})_U$ is an \'etale $A$-torsor.

For $0\le a\le2g+1$, put $W_a=(a,0)\in C_{-1}(\Q)$.
The points $W_a$ are the ramification points of
$\rho_{-1}:C_{-1}\to\mathbb P^1_{\Q}$.
We will prove that the classes
$\delta_a=[W_a-W_{2g+1}]$, for $0\le a\le2g$, generate
$J(C_{-1})(\overline\Q)[2]$.
First, $2\delta_a=0$, because
\begin{equation}\label{eq:real-explicit-branch-divisor}
 \operatorname{div}\!\left(\frac{x-a}{x-(2g+1)}\right)
     =2W_a-2W_{2g+1}.
\end{equation}
To determine the relations among the $\delta_a$, let
$I\subseteq\{0,\ldots,2g\}$ and suppose
$\sum_{a\in I}\delta_a=0$.
Put
\begin{equation}\label{eq:real-explicit-subset-product}
 R_I(x)=\prod_{a\in I}\frac{x-a}{x-(2g+1)}.
\end{equation}
The assumed relation gives $h\in\overline\Q(C_{-1})^\times$ with
$\operatorname{div}(h)=\sum_{a\in I}(W_a-W_{2g+1})$.
Equation~\eqref{eq:real-explicit-branch-divisor} implies
$\operatorname{div}(R_I/h^2)=0$, so
$R_I/h^2\in\overline\Q^\times$ and
$R_I\in\overline\Q(C_{-1})^{\times2}$.
Recall that
$\overline\Q(C_{-1})=\overline\Q(x)(\sqrt F)$, where
$F=\prod_{a=0}^{2g+1}(x-a)$.
If $(u+v\sqrt F)^2\in\overline\Q(x)$, with
$u,v\in\overline\Q(x)$, the coefficient of $\sqrt F$ is $2uv$;
hence $u=0$ or $v=0$. It follows that
\begin{equation}\label{eq:real-explicit-square-intersection}
 \overline\Q(x)^\times\cap\overline\Q(C_{-1})^{\times2}
   =\overline\Q(x)^{\times2}\ \cup\ F\overline\Q(x)^{\times2}.
\end{equation}
For $0\le a\le2g$, the order of vanishing of $R_I$ at $x=a$ is
$1$ if $a\in I$ and $0$ if $a\notin I$.
If $R_I\in\overline\Q(x)^{\times2}$, then
$\operatorname{ord}_{x=a}(R_I)\equiv0\pmod2$ for every
$0\le a\le2g$, so $I=\varnothing$. If
$R_I\in F\overline\Q(x)^{\times2}$, then
$\operatorname{ord}_{x=a}(R_I)\equiv1\pmod2$ for every
$0\le a\le2g$, so $I=\{0,\ldots,2g\}$.
Finally,
\begin{equation}\label{eq:real-explicit-full-relation-divisor}
 \operatorname{div}\!\left(\frac{y}{(x-(2g+1))^{g+1}}\right)
   =\sum_{a=0}^{2g}(W_a-W_{2g+1})
\end{equation}
gives $\sum_{a=0}^{2g}\delta_a=0$.
Thus the $2g+1$ classes $\delta_a$ have precisely one nonzero
$\F$-linear relation, and
\begin{equation}\label{eq:real-explicit-two-torsion-span}
 \dim_{\F}\langle\delta_0,\ldots,\delta_{2g}\rangle=2g
    =\dim_{\F}J(C_{-1})(\overline\Q)[2].
\end{equation}
Since every $\delta_a$ is rational over $\Q$, equation
\eqref{eq:real-explicit-two-torsion-span} proves
\eqref{eq:real-explicit-rational-two-torsion} and completes (iii).

Fix an odd prime $p\in\Sigma$.
Proposition~\ref{prop:starting-odd-fibres} shows that
$Y_p=(S_{-1})_{\overline{\mathbb F}_p,\mathrm{red}}$ has strict normal
crossings. Let
$n_p=|V_p|$ be the number of irreducible components of
$(P_p)_{\overline{\mathbb F}_p}$, where
$\Gamma_p=(V_p,\mathcal E_p,\partial_p)$ is defined in
\eqref{eq:start-explicit-fibre-graph}.
In the construction of $P_m$, a blow-up at stage $k\ge1$ over $p$
has centre indexed by a residue class modulo $p^k$ containing at
least two integers from $\{0,\ldots,N-1\}$.
At stage $k$ there are at most $p^k$ centres, and a centre can occur
only if $p^k\le N-1$. Each blow-up adds one irreducible component to
the geometric fibre. Therefore, on putting
$M_p=\lfloor\log_p(N-1)\rfloor$, we obtain
\begin{equation}\label{eq:real-explicit-base-component-bound}
 n_p\le1+\sum_{k=1}^{M_p}p^k
       \le1+\frac{p}{p-1}(N-1)\le1+\frac32N.
\end{equation}
The divisors $D_a$, for $0\le a<N$, meet
$(P_p)_{\overline{\mathbb F}_p}$ at $N$ distinct smooth points.
Also, $\Gamma_p$ is a tree, so
$|\mathcal E_p|=n_p-1$; equivalently,
$(P_p)_{\overline{\mathbb F}_p}$ has $n_p-1$ nodes.
Define the finite set
\begin{equation}\label{eq:real-explicit-base-intersections}
 \mathcal I_p=
 \operatorname{Sing}((P_p)_{\overline{\mathbb F}_p})
 \cup\bigcup_{a=0}^{N-1}
       (D_a\cap(P_p)_{\overline{\mathbb F}_p}),
 \qquad |\mathcal I_p|=N+n_p-1.
\end{equation}
The branch divisor of $X_{-1}\to P_m$ is supported on the union of
the $D_a$ and the components of the special fibres of $P_m$.
The local equations of Proposition~\ref{prop:starting-odd-fibres}
show that every singular point of $X_{-1}$ over $p$ maps to a point
of $\mathcal I_p$, and that at most one singular point maps to each
element of $\mathcal I_p$.
The curve $(X_{-1})_{\overline{\mathbb F}_p,\mathrm{red}}$ has at
most $2n_p$ irreducible components.
In the construction of $S_{-1}\to X_{-1}$, the blow-up at each
singular point of $X_{-1}$ over $p$ introduces one rational component.
Proposition~\ref{prop:starting-odd-fibres} also proves that $Y_p$ has
at most four nodes above each node of
$(P_p)_{\overline{\mathbb F}_p}$, and at most two nodes above each
point of $D_a\cap(P_p)_{\overline{\mathbb F}_p}$.
Hence
\begin{equation}\label{eq:real-explicit-odd-component-node-bound}
 c_p\le2n_p+(N+n_p-1)=3n_p+N-1\le7N,\qquad
 s_p\le4(n_p-1)+2N\le8N.
\end{equation}

For each irreducible component $E$ of
$(P_p)_{\overline{\mathbb F}_p}$, let $\mathcal C_E$ be the set of
components of $(X_{-1})_{\overline{\mathbb F}_p,\mathrm{red}}$
whose image in $P_p$ is $E$.
For $Z\in\mathcal C_E$, let $\widetilde Z$ be the normalization
of $Z$ and let $\phi_Z:\widetilde Z\to E$ be the induced morphism.
Proposition~\ref{prop:starting-odd-fibres} gives
$\deg\phi_Z\in\{1,2\}$.
For a degree-two morphism $\phi_Z$, let $b_Z$ be the number of
branch points on $E$.
A point of $D_a\cap(P_p)_{\overline{\mathbb F}_p}$ contributes at
most one branch point, and a node of
$(P_p)_{\overline{\mathbb F}_p}$ contributes at most one branch
point on each of the two components containing that node. Therefore
\begin{equation}\label{eq:real-explicit-branch-count}
 \sum_{E\in V_p}\sum_{\substack{Z\in\mathcal C_E\\\deg\phi_Z=2}}
 b_Z\le N+2(n_p-1)\le4N.
\end{equation}
For $\deg\phi_Z=2$, Riemann--Hurwitz gives
$2g(\widetilde Z)-2=-4+b_Z$, because
$E\simeq\mathbb P^1_{\overline{\mathbb F}_p}$ and $p\ne2$.
In particular, $g(\widetilde Z)\le b_Z/2$.
For $\deg\phi_Z=1$, the curve $\widetilde Z$ is rational.
The exceptional components of $S_{-1}\to X_{-1}$ are also rational.
Summing the genera of the normalization components of $Y_p$ and
using \eqref{eq:real-explicit-branch-count}, we get $a_p\le2N$.
Together with \eqref{eq:real-explicit-odd-component-node-bound},
this proves (iv).

For the prime two, put $r=(g+1)/4$ and $n=2r-1$.
Proposition~\ref{prop:starting-dyadic-fibres} gives $2n$ rational
normalization components of $Y_2$, $2(n-1)$ nodes, and $2r$
points where two branches have intersection multiplicity two.
Write $\Gamma(Y_2)=(V_Y,\mathcal E_Y,\partial_Y)$, and let
$\mathcal T_2\subseteq\mathcal E_Y=\operatorname{Sing}(Y_2)$
be the set of the $2r$ points where the two branches have
intersection multiplicity two. Define
\begin{equation}\label{eq:real-explicit-dyadic-subgraph}
 \Gamma'=(V_Y,\mathcal E_Y\setminus\mathcal T_2,
                   \partial_Y|_{\mathcal E_Y\setminus\mathcal T_2}).
\end{equation}
The local descriptions in Proposition~\ref{prop:starting-dyadic-fibres}
identify $\Gamma'$ with the disjoint union of two copies of the tree
$\Gamma_2$ from \eqref{eq:start-explicit-fibre-graph}.
Denote the vertex sets of the two connected components of $\Gamma'$
by $V_Y^{(0)}$ and $V_Y^{(1)}$.
For each $E_{m-2,c}$, the union of the divisors $D_a$ meets
$E_{m-2,c}$ at two points. The points of $\mathcal T_2$ map to
$\bigcup_{a,c}(D_a\cap E_{m-2,c})$.
The two branches at every $y\in\mathcal T_2$ lie on the two
components mapping to the same $E_{m-2,c}$, so
\begin{equation}\label{eq:real-explicit-dyadic-connecting-edges}
 |\partial_Y(y)\cap V_Y^{(0)}|
   =|\partial_Y(y)\cap V_Y^{(1)}|=1
 \qquad(y\in\mathcal T_2).
\end{equation}
Since $|\mathcal T_2|=2r>0$, the graph $\Gamma(Y_2)$ is connected.
Thus $Y_2$ is connected. By \eqref{eq:bg-explicit-graph-betti},
\begin{equation}\label{eq:real-explicit-dyadic-betti}
 b_1(\Gamma(Y_2))=2(n-1)+2r-2n+1=\frac{g-1}{2}.
\end{equation}

It remains to verify \eqref{eq:real-explicit-stabilizer-hypotheses}.
Recall that
\begin{equation}\label{eq:real-explicit-sign-group}
 H=\Gal(L/\Q(x))=\F^{g+1},\qquad
 A=\ker\!\left(\lambda:H\to\F\right),\qquad
 \lambda((a_i)_{i=0}^g)=\sum_{i=0}^g a_i.
\end{equation}
Under the identification in \eqref{eq:real-explicit-sign-group}, the
standard basis vector $e_i\in\F^{g+1}$ represents the field
automorphism
\begin{equation}\label{eq:real-explicit-basis-automorphism}
 e_i(x)=x,\qquad e_i(z_i)=-z_i,\qquad
 e_i(z_h)=z_h\quad(0\le h\le g,\ h\ne i).
\end{equation}
The morphism $S_{-1}\to X_{-1}$ is an isomorphism after base change
to $\Spec\Z_{(2)}$. Therefore $(S_0)_{\Z_{(2)}}$ is the
normalization of $P_2$ in $L$; denote the normalization morphism by
$\psi:(S_0)_{\Z_{(2)}}\to P_2$.

Let $\bar q\in(P_2)_{\overline{\mathbb F}_2}$ satisfy
$\bar q\notin\bigcup_{a=0}^{N-1}D_a$, and put
$\mathcal A=\mathcal O^{\mathrm{sh}}_{P_2,\bar q}$.
For $0\le i\le g$, the rational function
$A_i=-2^{m-2}/(x-i)$ belongs to $\mathcal A$.
Define
\begin{equation}\label{eq:real-explicit-dyadic-split-equations}
 W_i=\frac{z_i/(x-i)-1}{2},\qquad
 W_i^2+W_i=A_i.
\end{equation}
By Proposition~\ref{prop:starting-regularity},
\begin{equation}\label{eq:real-explicit-dyadic-split-algebras}
 \mathcal A[w_i]/(w_i^2+w_i-A_i)\simeq\mathcal A\times\mathcal A,
 \qquad
 \bigotimes_{i=0}^g\mathcal A[w_i]/(w_i^2+w_i-A_i)
      \simeq\mathcal A^{\,2^{g+1}}.
\end{equation}
The tensor product in \eqref{eq:real-explicit-dyadic-split-algebras}
is the integral closure of $\mathcal A$ in
$L\otimes_{\Q(x)}\Frac(\mathcal A)$.
Hence the action of $H$ on the geometric points of
$\psi^{-1}(\bar q)$ is simply transitive. In particular,
\begin{equation}\label{eq:real-explicit-trivial-stabilizers}
 H_{\bar x}=\{0\},\qquad A_{\bar x}=\{0\}
 \quad\text{for every }\bar x\in\psi^{-1}(\bar q).
\end{equation}

Put $k=m-2$. For $0\le j<2^{m-1}$, write
$j=c+2^ke$, with $0\le c<2^k$ and $e\in\{0,1\}$, and let
$\bar q_j$ be the point of $E_{k,c}$ with
$X_{k,c}=\overline e$.
The divisors $D_a$ containing $\bar q_j$ are exactly those with
$a\equiv j\pmod{2^{m-1}}$.
Let $\mathcal A_j=\mathcal O^{\mathrm{sh}}_{P_2,\bar q_j}$ and put
\begin{equation}\label{eq:real-explicit-dyadic-scaled-roots}
 t=\frac{x-j}{2^{m-2}},\qquad
 U=\frac{z_j}{2^{m-2}},\qquad
 V=\frac{z_{j+2^{m-1}}}{2^{m-2}}.
\end{equation}
For each $i\notin\{j,j+2^{m-1}\}$, the integral closure of
$\mathcal A_j$ in
$\Frac(\mathcal A_j)[z_i]/(z_i^2-f_i)$ is
$\mathcal A_j\times\mathcal A_j$, by
Proposition~\ref{prop:starting-regularity}.
For the indices $j$ and $j+2^{m-1}$, the equations become
\begin{equation}\label{eq:real-explicit-dyadic-pair-equations}
 U^2=t(t-4),\qquad V^2=(t-2)(t-6).
\end{equation}
By Proposition~\ref{prop:starting-regularity}, the integral closure
of $\mathcal A_j$ in
$\Frac(\mathcal A_j)[U,V]/(U^2-t(t-4),V^2-(t-2)(t-6))$
is the algebra
\begin{equation}\label{eq:real-explicit-dyadic-normalization}
 \mathcal B_j=\mathcal A_j[a,b]/
       (a^2+ta+t,\ b^2+(t-2)b+t-2),
 \qquad U=t+2a,\quad V=t-2+2b.
\end{equation}
The algebra $\mathcal B_j$ is free of rank four over $\mathcal A_j$
and has one closed point. The remaining $g-1$ quadratic factors
split over $\mathcal A_j$. Thus the integral closure of
$\mathcal A_j$ in $L\otimes_{\Q(x)}\Frac(\mathcal A_j)$ is a product
of $2^{g-1}$ copies of $\mathcal B_j$.
For every $\bar x\in\psi^{-1}(\bar q_j)$, the stabilizers are
therefore
\begin{equation}\label{eq:real-explicit-nontrivial-stabilizers}
 H_{\bar x}=\langle e_j,e_{j+2^{m-1}}\rangle\simeq\F^2,
 \qquad
 A_{\bar x}=H_{\bar x}\cap\ker\lambda
           =\langle e_j+e_{j+2^{m-1}}\rangle\simeq\F.
\end{equation}
Proposition~\ref{prop:starting-dyadic-fibres} gives exactly one
point $y_j\in Y_2$ mapping to $\bar q_j$.
Put $B_2=\{y_j:0\le j<2^{m-1}\}$. The points $\bar q_j$ are
distinct, so
\begin{equation}\label{eq:real-explicit-exceptional-point-count}
 |B_2|=2^{m-1}=2r=\frac{g+1}{2}.
\end{equation}
Equations~\eqref{eq:real-explicit-trivial-stabilizers} and
\eqref{eq:real-explicit-nontrivial-stabilizers} prove
\eqref{eq:real-explicit-stabilizer-hypotheses} for this set $B_2$.
Together with \eqref{eq:real-explicit-dyadic-betti} and the description
of the normalization components and singular points of $Y_2$, this
proves (v).
\end{proof}
\subsection{Estimating the obstruction group}\label{sec:estimating-obstruction}

Let $r:\widetilde S\to S_0$ and $q:S_0\to S_{-1}$ form an
anchored diagram in the sense of \eqref{eq:anchored-diagram}. Write
$C=(S_{-1})_{\Q}$, $\widetilde C=\widetilde S_{\Q}$ and
\begin{equation}\label{eq:obs-explicit-bases}
 B_p=\Spec\Z_{(p)}^h,\qquad K_p=\Frac(\Z_{(p)}^h),\qquad
 X=(S_{-1})_U.
\end{equation}
For each $\theta\in H^2(G,\F)$, choose a representative extension
$1\to\F\xrightarrow{j_\theta}E_\theta
\xrightarrow{\rho_\theta}G\to1$ as in
\eqref{eq:realization-central-extension}.

For $V=\Spec\Z$, $\Spec\Q$, $U$, $B_p$ or $\Spec K_p$, a
\emph{central lift} of $\theta$ over $V$ is a finite \'etale
$\F$-torsor $P_V\to\widetilde S_V$ with an $E_\theta$-action over
$(S_{-1})_V$.  The morphism $\pi_V:P_V\to\widetilde S_V$
must satisfy $\pi_V(ez)=\rho_\theta(e)\pi_V(z)$ for $e\in E_\theta$;
$j_\theta(a)$ acts on $P_V$ by translation by $a\in\F$. Thus over $\Spec\Q$ the
torsor is $P_{\Q}\to\widetilde C$, and over $\Spec\Z$ it is a
torsor $P_{\Z}\to\widetilde S$.
The scheme $P_V$ is allowed to be disconnected.  For the split extension
$E_0=G\times\F$, the trivial torsor $\widetilde S_V\times\F$
is a central lift, with $(g,a)(z,b)=(gz,a+b)$.  For $\theta\ne0$, the scheme $P_V$ is connected: a connected
component of a disconnected $P_V$ would give a subgroup of
$E_\theta$ mapping isomorphically to $G$, as proved in
Subsection~\ref{sec:realizable-extensions}. Therefore
\begin{equation}\label{eq:obs-mechanical-integral-realizability}
 \mathscr E(\widetilde S,S_{-1})
 =\{\theta\in H^2(G,\F):\theta\text{ has a central lift over }\Spec\Z\}.
\end{equation} An isomorphism of central lifts over $V$ is an
$E_\theta$-equivariant isomorphism over the identity of
$\widetilde S_V$.

Let $P\to\widetilde C$ be a central lift of $\theta$, and let
$D\to C$ be a finite \'etale $\F$-torsor. Define
\begin{equation}\label{eq:obs-explicit-twist-quotient}
 P\star D:=\bigl(P\times_{\widetilde C}
       (\widetilde C\times_C D)\bigr)/\Delta\F.
\end{equation}
The scheme $P\star D$ is the contracted product of the
$\F$-torsors $P\to\widetilde C$ and
$\widetilde C\times_C D\to\widetilde C$. More precisely, if
$z\in P$ and $(x,d)\in\widetilde C\times_C D$ have the same
image in $\widetilde C$, the diagonal action is
\begin{equation}\label{eq:obs-mechanical-diagonal-twist-action}
 a\cdot(z,(x,d))=(z+a,(x,d+a)),\qquad a\in\F.
\end{equation}
After an \'etale base change trivializing the torsors, the addition
map $(a,b)\mapsto a+b$ identifies
$(\F\times\F)/\Delta\F$ with $\F$. Descent, as in
Proposition~\ref{prop:realizable-subspace}, therefore constructs
$P\star D\to\widetilde C$ as a finite \'etale $\F$-torsor.
The formula
\begin{equation}\label{eq:obs-mechanical-extension-twist-action}
 e\cdot(z,(x,d))=(ez,(\rho_\theta(e)x,d)),\qquad e\in E_\theta,
\end{equation}
defines an $E_\theta$-action commuting with
\eqref{eq:obs-mechanical-diagonal-twist-action}. Consequently
\eqref{eq:obs-mechanical-extension-twist-action} induces an action
on $P\star D$. For $a\in\F$, the element $j_\theta(a)$ acts
on $P\star D$ by the torsor translation by $a$. Hence
$P\star D\to\widetilde C$ is a central lift of $\theta$.  For $d\in\Q^\times$, put
$D_d=\Spec(\Q[t]/(t^2-d))$, a quadratic torsor over $\Spec\Q$;
in particular, $D_1=\Spec\Q\coprod\Spec\Q$.  Twisting by
$C\times_{\Q}D_d$ replaces a quadratic function-field equation
$z^2=a$ by $z^2=da$.

\begin{proposition}\label{prop:good-prime-twist}
Let $\widetilde S\to S_0\to S_{-1}$ be an anchored diagram.  Every central
lift $P_{\Q}\to\widetilde C$ has a twist by $C\times_{\Q}D_d$,
for a positive squarefree integer $d$ whose prime divisors lie outside
$\Sigma$ (with $d=1$ allowed), that extends to a central lift
$P_U\to\widetilde S_U$.  Consequently an extension class has a central lift
over $\Q$ if and only if it has one over $U$.
\end{proposition}

\begin{proof}
On the trivial torsor $\widetilde C\times\F$, each $e\in E_\theta$
acts by $(z,b)\mapsto(\rho_\theta(e)z,b+\lambda(e))$ for some
$\lambda(e)\in\F$.  The formula $(z,b)\mapsto(\rho_\theta(e)z,b+\lambda(e))$
defines the action of $e$ on $\widetilde S_U\times\F$ as well.  Suppose therefore that
$P_{\Q}$ is connected, with quadratic function field
$\Q(\widetilde C)(\sqrt a)$ for
$a\in\Q(\widetilde C)^\times$.  Let $n_a:N_a\to\widetilde S_U$ be the normalization of
$\widetilde S_U$ in $\Q(\widetilde C)(\sqrt a)$.
The morphism $n_a$ is finite because $\widetilde S_U$ is excellent.
Let $Z_a\subseteq N_a$ be the closed subset where $n_a$ is not \'etale.
The image of $Z_a$ under the proper morphism $N_a\to U$ is closed.
Since $(N_a)_{\Q}\to\widetilde C$ is the \'etale morphism
$P_{\Q}\to\widetilde C$, that image does not contain the generic
point of $U$. Thus
\begin{equation}\label{eq:obs-mechanical-ramification-prime-set}
 \Sigma_a:=\operatorname{im}(Z_a\to U)
 \quad\text{is a finite subset of }\{p:p\notin\Sigma\}.
\end{equation}

For $p\notin\Sigma$, the map $\widetilde S_U\to U$ is smooth, so every
irreducible component of $(\widetilde S_U)_p$ has multiplicity one.
For an irreducible component $D$ of $(\widetilde S_U)_p$, let
$\eta_D$ be the generic point of $D$ and let $v_D$ be the
normalized discrete valuation of the ring
$\mathcal O_{\widetilde S_U,\eta_D}$.  Since $p\ne2$, the extension
$\Q(\widetilde C)(\sqrt a)/\Q(\widetilde C)$ ramifies there if
and only if $v_D(a)$ is odd.  The group $G$ acts transitively on the irreducible components of
$(\widetilde S_U)_p$, because $X_p$ is geometrically integral and
$\widetilde S_U\to X$ is a $G$-torsor. The $E_\theta$-action on
$\Q(\widetilde C)(\sqrt a)$ extends to $N_a$. Choose a point $w\in N_a$ with $n_a(w)=\eta_D$, and define
\begin{equation}\label{eq:obs-mechanical-divisorial-inertia-definition}
 I_w=\{a\in j_\theta(\F):aw=w,\quad
          a|_{\kappa(w)}=\operatorname{id}_{\kappa(w)}\}.
\end{equation}
The group $j_\theta(\F)$ acts transitively on the points of
$n_a^{-1}(\eta_D)$. Since $j_\theta(\F)$ is abelian, the
subgroup $I_w$ is independent of the choice of $w$; denote that
subgroup by $I_D$. If $e\in E_\theta$ and $g=\rho_\theta(e)$,
then $n_a(ew)=g\eta_D$ and $I_{ew}=eI_we^{-1}$. Centrality
of $j_\theta(\F)$ now gives
\begin{equation}\label{eq:obs-mechanical-inertia-parity}
 I_{gD}=eI_De^{-1}=I_D,\qquad
 v_{gD}(a)\equiv v_D(a)\pmod2.
\end{equation}

For $p\notin\Sigma$, let $\epsilon_p\in\{0,1\}$ be the
common parity in \eqref{eq:obs-mechanical-inertia-parity}. By
\eqref{eq:obs-mechanical-ramification-prime-set}, only finitely many
$\epsilon_p$ are nonzero. Set
\begin{equation}\label{eq:obs-mechanical-constant-twist}
 d=\prod_{\substack{p\notin\Sigma\\\epsilon_p=1}}p.
\end{equation}
For every irreducible component $D$ of $(\widetilde S_U)_p$,
smoothness gives $v_D(p)=1$. Therefore
\begin{equation}\label{eq:obs-mechanical-twisted-valuations}
 v_D(da)=v_D(a)+\epsilon_p\equiv0\pmod2.
\end{equation}  At the generic point of every horizontal divisor on $\widetilde S_U$,
$d$ is a unit and the residue characteristic is zero, so adjoining
$\sqrt d$ is also unramified there.  Since
$P_{\Q}\to\widetilde C$ was \'etale, the normalization in the
quadratic algebra defined by $z^2=da$ is unramified at every
codimension-one point of the regular surface $\widetilde S_U$.  Purity of the branch locus now shows that the normalization morphism
to $\widetilde S_U$ is \'etale; see \st{0BMB}.

To identify the extension acting on the twist, let $G_{\Q(C)}$ be
the absolute Galois group of $\Q(C)$ and let
$\rho:G_{\Q(C)}\to E_\theta$ be the homomorphism obtained from
$P_{\Q}\to C$.  Its composite with $\rho_\theta$ is the quotient
corresponding to $\Q(\widetilde C)/\Q(C)$.  Let
$\chi_d:G_{\Q(C)}\to\F$ be the pullback of the character of
$G_\Q$ classified by $D_d\to\Spec\Q$.  Then
$\rho_d(\gamma)=j_\theta(\chi_d(\gamma))\rho(\gamma)$ is a
homomorphism and $\rho_\theta\rho_d=\rho_\theta\rho$, since
$j_\theta(\F)$ is central.  The $E_\theta$-action on
$P_{\Q}\star(C\times_{\Q}D_d)$ extends to the normalization of
$\widetilde S_U$ in the algebra of
$P_{\Q}\star(C\times_{\Q}D_d)$, by functoriality of normalization.  If the twist $P_{\Q}\star(C\times_{\Q}D_d)$ is
disconnected, apply the normalization and purity arguments to each
factor of the associated finite $\Q(\widetilde C)$-algebra. The
result is a central lift over $U$ in that case as well. By
\eqref{eq:obs-mechanical-integral-realizability} and the complement
argument preceding that equation, a disconnected twist can occur only
when $\theta=0$.
\end{proof}

Put $H=\Gal(\widetilde S/S_0)=\ker(G\to A)$.  The map
$r:\widetilde S\to S_0$ is an \'etale $H$-torsor.  Descent along
$r$ takes $G$-equivariant quadratic torsors on $\widetilde S$ to
$A$-equivariant quadratic torsors on $S_0$.  We use $A$-equivariant torsors on $S_0$ for twisting. The morphism
$q:S_0\to S_{-1}$ can ramify over primes in $\Sigma$, so an
$A$-equivariant \'etale torsor on $S_0$ need not descend to an
\'etale torsor on $S_{-1}$.

Let $T$ be a scheme on which a finite group $B$ acts. We denote
by $H_B^i(T,\F)$ the right derived functors of
$\mathcal F\mapsto\Gamma(T,\mathcal F)^B$ on $B$-equivariant \'etale
sheaves, evaluated at the constant sheaf.  A $B$-equivariant sheaf $\mathcal F$ has isomorphisms
$b^*\mathcal F\simeq\mathcal F$ compatible with multiplication in $B$.  In degree one, $H_B^1(T,\F)$ classifies finite \'etale
$\F$-torsors $D\to T$ with a $B$-action lifting the action on $T$
and commuting with torsor translations.  For $T=(S_0)_V$ and $B=A$,
if $D\to(S_0)_V$ represents a class in $H_A^1((S_0)_V,\F)$
and $P\to\widetilde S_V$ is a central lift of $\theta$, then
$(P\times_{\widetilde S_V}r_V^*D)/\Delta\F
\to\widetilde S_V$ is a central lift of $\theta$.  We write
$h_B^i(T)=\dim_{\F}H_B^i(T,\F)$ and
$h^i(T)=\dim_{\F}H^i_{\et}(T,\F)$.

Take an injective resolution $I^\bullet$ of the constant sheaf $\F$
in the category of $B$-equivariant \'etale sheaves on $T$, and form
the double complex of group cochains
\begin{equation}\label{eq:obs-explicit-equivariant-double-complex}
 C^a(B,\Gamma(T,I^b)).
\end{equation}
To embed a $B$-equivariant sheaf $\mathcal F$ in a
$B$-equivariant injective, first embed the underlying sheaf in an
injective sheaf $J$ on $T$, and then use the coinduced sheaf
$\prod_{b\in B}b^*J$.
The forgetful functor preserves injectives, since its left adjoint is an
exact finite direct sum of translates.  The global sections of an
equivariant injective form an injective $B$-module, by adjunction with
the exact constant-sheaf functor.  Filtering the double complex
\eqref{eq:obs-explicit-equivariant-double-complex} by group-cochain
degree and by resolution degree gives the spectral sequence
\begin{equation}\label{eq:realizable-equivariant-sequence}
 H^a(B,H^b_{\et}(T,\F))\ \Longrightarrow\ H_B^{a+b}(T,\F).
\end{equation}
This is the usual derived-functor construction; see \st{015N}.
For a finite \'etale $B$-torsor $T\to Z$, descent identifies
$H_B^i(T,\F)$ with $H^i_{\et}(Z,\F)$.

\begin{proposition}\label{prop:local-lift-criterion}
For an anchored diagram and
$V\in\{\Spec\Q,U,B_p,\Spec K_p\}$, there is a natural linear map
\begin{equation}\label{eq:obs-explicit-obstruction-map}
 \alpha_V:H^2(G,\F)\longrightarrow H_A^2((S_0)_V,\F).
\end{equation}
For $\theta\in H^2(G,\F)$, a central lift
$P_V\to\widetilde S_V$ exists if and only if $\alpha_V(\theta)=0$.  If $\alpha_V(\theta)=0$, the isomorphism classes of central lifts
of $\theta$ form a torsor under $H_A^1((S_0)_V,\F)$, acting by quadratic twisting.
The maps $\alpha_V$ and the twisting actions commute with pullback
under every base change $V'\to V$ between schemes in
$\{\operatorname{Spec}\mathbb Q,U,B_p,\operatorname{Spec}K_p\}$.  For
$V=U,\Spec\Q,\Spec K_p$, descent along the \'etale $A$-torsor
$q_V:(S_0)_V\to(S_{-1})_V$ identifies the target with
$H^2_{\et}((S_{-1})_V,\F)$.
\end{proposition}

\begin{proof}
Recall that $H=\ker(G\to A)$. The map
$r_V:\widetilde S_V\to(S_0)_V$ is an \'etale $H$-torsor.
Pullback along $r_V$ gives an exact equivalence from
$A$-equivariant \'etale sheaves on $(S_0)_V$ to $G$-equivariant
\'etale sheaves on $\widetilde S_V$.  The inverse functor is
$\mathcal F\mapsto(r_{V*}\mathcal F)^H$. To check the assertion,
choose an \'etale morphism $W\to(S_0)_V$ trivializing $r_V$. Then
$\widetilde S_V\times_{(S_0)_V}W\simeq\coprod_{h\in H}W$.
The $H$-equivariance identifies the restrictions of $\mathcal F$ to
the copies of $W$. Descent therefore recovers a sheaf on $(S_0)_V$,
and the quotient action of $G/H=A$ gives the $A$-equivariance.
Exactness can be checked after the same \'etale base change.
For a $G$-equivariant sheaf $\mathcal F$ on $\widetilde S_V$, this
equivalence identifies
$\Gamma(\widetilde S_V,\mathcal F)^G$ with
$\Gamma((S_0)_V,(r_{V*}\mathcal F)^H)^A$.  Passing to derived
functors gives
\begin{equation}\label{eq:realizable-fixed-descent}
 H_G^i(\widetilde S_V,\F)\simeq H_A^i((S_0)_V,\F).
\end{equation}
The group $H_A^i((S_0)_V,\F)$ depends on $q:S_0\to S_{-1}$
and $V$, independently of the choice of $r:\widetilde S\to S_0$.
The scheme $\widetilde S_V$ is connected, because
$\widetilde S_{\Q}$ is geometrically integral by Proposition~\ref{prop:realizing-surface-properties}.
The five-term sequence of \eqref{eq:realizable-equivariant-sequence} gives
\begin{equation}\label{eq:obs-explicit-five-term-lift-sequence}
 H^1_{\et}(\widetilde S_V,\F)^G
 \xrightarrow{d_2}H^2(G,\F)
 \xrightarrow{\alpha_V}H_A^2((S_0)_V,\F).
\end{equation}
Here $\alpha_V$ is the edge map to
$H_G^2(\widetilde S_V,\F)$ followed by
\eqref{eq:realizable-fixed-descent}.  A class in
$H^1_{\et}(\widetilde S_V,\F)^G$ is represented by a quadratic
torsor $P\to\widetilde S_V$ such that each $g\in G$ lifts to an
automorphism of $P$.  Define
\begin{equation}\label{eq:obs-mechanical-torsor-automorphism-extension}
 E(P)=\{(g,\varphi):g\in G,\ \varphi:P\xrightarrow{\sim}P,
             \ \pi_V\varphi=g\pi_V,\
             \varphi\text{ commutes with the }\F\text{-action}\}.
\end{equation}
Give $E(P)$ the multiplication
\begin{equation}\label{eq:obs-mechanical-torsor-extension-multiplication}
 (g,\varphi)(h,\psi)=(gh,\varphi\circ\psi).
\end{equation}
The homomorphism $E(P)\to G$ given by $(g,\varphi)\mapsto g$
is a central extension of $G$ by $\F$: every $g$ has a lift, and an
$\F$-equivariant automorphism of $P$ over the connected scheme
$\widetilde S_V$ is translation by an element of $\F$.
Choose $(g,\widehat g)\in E(P)$ for each $g\in G$. Then
\begin{equation}\label{eq:obs-mechanical-transgression-cocycle}
 d_2[P]=[c_P],\qquad
 c_P(g,h)=\widehat g\widehat h\widehat{gh}^{-1}\in\F.
\end{equation}  To verify \eqref{eq:obs-mechanical-transgression-cocycle}, choose
an \'etale cover $W_0\to\widetilde S_V$ trivializing $P$, and set
\begin{equation}\label{eq:obs-mechanical-equivariant-trivializing-cover}
 W=\prod_{g\in G,\,\widetilde S_V}g^*W_0.
\end{equation}
Here the product is the fibre product over $\widetilde S_V$.
Permuting the factors by the action of $G$ makes
$W\to\widetilde S_V$ a $G$-equivariant \'etale cover, and the
factor indexed by $1\in G$ shows that $W$ trivializes $P$.  Represent
$P$ by a \v Cech cocycle $z$ on
$W\times_{\widetilde S_V}W$, and choose degree-zero \v Cech
cochains $b_g$ on $W$ with $gz-z=\delta b_g$, where $\delta$ is
the \v Cech differential.  The group-cochain differential gives
$d_2z(g,h)=b_g+gb_h-b_{gh}$, the same translation as
$\widehat g\widehat h\widehat{gh}^{-1}$.  Exactness of
\eqref{eq:obs-explicit-five-term-lift-sequence} therefore identifies
$\ker\alpha_V$ with the classes admitting central lifts.  Torsors are
classified by $H^1$ and represented by finite \'etale schemes; see
\st{03AJ}, \st{03RV}.

Let $P_1,P_2\to\widetilde S_V$ be central lifts of the extension
$E_\theta$. Define the quadratic torsor
\begin{equation}\label{eq:obs-explicit-difference-quotient}
 D_{12}:=(P_1\times_{\widetilde S_V}P_2)/\Delta\F.
\end{equation}
The action $e(p_1,p_2)=(ep_1,ep_2)$ induces an action of
$E_\theta$ on $D_{12}$. For $a\in\F$, the element $j_\theta(a)$
acts on $P_1\times_{\widetilde S_V}P_2$ by $(p_1,p_2)\mapsto
(p_1+a,p_2+a)$ and hence acts trivially on $D_{12}$. Thus the
$E_\theta$-action on $D_{12}$ factors through $G$, and descent
along $r_V$ gives an $A$-equivariant quadratic torsor on $(S_0)_V$.
Conversely, pulling back an $A$-equivariant quadratic torsor from
$(S_0)_V$ and twisting $P_1$ gives a central lift of $\theta$.
Write $D=D_{12}$ in the following formula.
There is an $E_\theta$-equivariant isomorphism
\begin{equation}\label{eq:obs-explicit-difference-inverse}
 \begin{aligned}
 (P_1\times_{\widetilde S_V}D)/\Delta\F&\longrightarrow P_2,\\
 [p,[p_1,p_2]]&\longmapsto p_2+a,\qquad p=p_1+a.
 \end{aligned}
\end{equation}
The equality $p=p_1+a$ determines $a\in\F$ uniquely.
Replacing $(p_1,p_2)$ by $(p_1+b,p_2+b)$ replaces $a$ by $a+b$,
so $p_2+a$ is unchanged. Replacing $(p,[p_1,p_2])$ by
$(p+b,[p_1+b,p_2])$ also leaves $p_2+a$ unchanged. Thus
\eqref{eq:obs-explicit-difference-inverse} is well defined.
After an \'etale base change trivializing $P_1$ and $P_2$, the map
in \eqref{eq:obs-explicit-difference-inverse} is the addition map
on the copies of $\F$, and is an isomorphism. The same formula,
with a $G$-equivariant torsor $D'\to\widetilde S_V$ in place of
$D_{12}$, gives
\begin{equation}\label{eq:obs-mechanical-twist-difference-identities}
 P_1\star D_{12}\simeq P_2,\qquad
 (P_1\times_{\widetilde S_V}(P_1\star D'))/\Delta\F\simeq D'.
\end{equation}
Here $P_1\star D'$ denotes the contracted product over
$\widetilde S_V$. Hence twisting acts freely and transitively on the isomorphism
classes of central lifts of $\theta$. The isomorphism
\eqref{eq:realizable-fixed-descent} identifies the group of twists
with $H_A^1((S_0)_V,\F)$, proving the assertion.  For $V=U,\Spec\Q,\Spec K_p$,
descent along $q_V$ identifies
$H_A^i((S_0)_V,\F)$ with $H^i_{\et}((S_{-1})_V,\F)$.
Pullback under a base change among the schemes $V$ commutes with
$r_V$-descent, the transgression and contracted products, proving the
restriction assertions.
\end{proof}

Let $f:X\to U$ be the structure map. The morphism $s=(q\circ s_0)|_U:U\to X$ satisfies
$f\circ s=\operatorname{id}_U$. Proper smooth base change
(\st{0GLI}) shows that $R^1f_*\F$ is locally constant of rank $2g$.
The generic geometric stalk is $H^1_{\et}(C_{\overline\Q},\F)$,
and $G_{\Q}$ acts trivially on $H^1_{\et}(C_{\overline\Q},\F)$
because $J(C)[2](\overline\Q)=J(C)[2](\Q)$.
The homomorphism $\pi_1(\Spec\Q)\to\pi_1(U)$ is surjective
by \st{0BQM}. Consequently
\begin{equation}\label{eq:obs-mechanical-constant-first-direct-image}
 R^1f_*\F\simeq{\F}_U^{\,2g}.
\end{equation}

The following proposition will detect whether
$\alpha_U(\theta)\in H^2_{\et}(X,\F)$ is zero from the restrictions
$\alpha_U(\theta)|_{C_{K_p}}$ in
$H^2_{\et}(C_{K_p},\F)$ for $p\in\Sigma$.

\begin{proposition}[Detection at the omitted primes]
\label{prop:good-detection}
Let $\Sigma$ be a finite set of rational primes containing two, and put
\begin{equation}\label{eq:obs-explicit-detection-bases}
 U=\operatorname{Spec}\mathbb Z[1/\prod_{p\in\Sigma}p],
 \qquad K_p=\operatorname{Frac}(\mathbb Z_{(p)}^h).
\end{equation}
Let $f:X\to U$ be a smooth proper curve with geometrically connected
fibres and a section $s:U\to X$.  Suppose that
$R^1f_*\mathbb F_2$ is constant.  Write $C=X_{\mathbb Q}$.
Then restriction is injective:
\begin{equation}\label{eq:obs-explicit-detection-restriction}
 H^2_{\mathrm{\acute et}}(X,\mathbb F_2)
 \longrightarrow
 \bigoplus_{p\in\Sigma}
 H^2_{\mathrm{\acute et}}(C_{K_p},\mathbb F_2).
\end{equation}
The same assertion holds with $K_p$ replaced by $\mathbb Q_p$.
\end{proposition}

\begin{proof}
We first prove that
$H^i_{\et}(U,\mathbb F_2)\to
\bigoplus_{p\in\Sigma}H^i_{\et}(\Spec\mathbb Q_p,\mathbb F_2)$
is injective for $i=1,2$.  In this proof all scheme cohomology is
\'etale cohomology, and cohomology of a field means that of its spectrum.  Since $\operatorname{Pic}(U)=0$, Kummer theory
gives
\begin{equation}\label{eq:obs-explicit-base-kummer-classes}
 H^1(U,\mathbb F_2)
   =\mathbb Z[1/\prod_{p\in\Sigma}p]^\times/
     \bigl(\mathbb Z[1/\prod_{p\in\Sigma}p]^\times\bigr)^2.
\end{equation}
Let $\eta\in H^1(U,\mathbb F_2)$. By
\eqref{eq:obs-explicit-base-kummer-classes}, there are exponents
$\epsilon,\epsilon_p\in\{0,1\}$ such that
\begin{equation}\label{eq:obs-mechanical-kummer-representative}
 \eta=[a],\qquad a=(-1)^\epsilon\prod_{p\in\Sigma}p^{\epsilon_p}.
\end{equation}
Suppose that $a\in\mathbb Q_p^{\times2}$ for every $p\in\Sigma$.
For each $p\in\Sigma$, the normalized $p$-adic valuation gives
$\epsilon_p=v_p(a)\in2\mathbb Z$, so $\epsilon_p=0$.
It follows that $a=(-1)^\epsilon$. An odd $2$-adic square is
congruent to $1$ modulo $8$, whereas $-1\equiv7\pmod8$.
Thus $\epsilon=0$ and $\eta=0$. We have proved that
\begin{equation}\label{eq:base-degree-one-detection}
 H^1(U,\mathbb F_2)\longrightarrow
 \bigoplus_{p\in\Sigma}H^1(\mathbb Q_p,\mathbb F_2)
 \quad\text{is injective}.
\end{equation}
Taking a basis of a finite-dimensional $\mathbb F_2$-vector space
$M$ and applying \eqref{eq:base-degree-one-detection} to each basis
coordinate proves that
\begin{equation}\label{eq:obs-mechanical-constant-coefficient-detection}
 H^1(U,M)\longrightarrow\bigoplus_{p\in\Sigma}H^1(\mathbb Q_p,M)
 \quad\text{is injective}
\end{equation}
when $M$ denotes the constant sheaf.

Kummer theory and $\operatorname{Pic}(U)=0$ identify
$H^2(U,\mathbb F_2)$ with $H^2(U,\mathbb G_m)[2]$.
We next prove that the restriction homomorphism
$H^2(U,\mathbb G_m)\to H^2(\mathbb Q,\mathbb G_m)$ is injective.
Consider the valuation sequence
\begin{equation}\label{eq:obs-explicit-valuation-sequence}
 0\longrightarrow\mathbb G_{m,U}\longrightarrow
 j_*\mathbb G_{m,\mathbb Q}\longrightarrow
 \bigoplus_{q\notin\Sigma}i_{q*}\mathbb Z\longrightarrow0,
\end{equation}
where $j:\Spec\mathbb Q\to U$ is the generic inclusion and
$i_q:\Spec\mathbb F_q\to U$ is the closed immersion at $q$.
Here
$H^1(U,\bigoplus_{q\notin\Sigma}i_{q*}\mathbb Z)=0$:
$H^1(\mathbb F_q,\mathbb Z)=0$, finite direct image is exact, and
\'etale cohomology on the quasi-compact quasi-separated scheme $U$
commutes with filtered colimits of abelian sheaves.  The long exact cohomology sequence of
\eqref{eq:obs-explicit-valuation-sequence} therefore gives the first
injection below. The stalk formula and Hilbert 90 give
$R^1j_*\mathbb G_m=0$, so the Leray spectral sequence for $j$
gives the second injection:
\begin{equation}\label{eq:obs-mechanical-brauer-restriction-injections}
 H^2(U,\mathbb G_m)\lhook\joinrel\longrightarrow
 H^2(U,j_*\mathbb G_{m,\mathbb Q})
 \lhook\joinrel\longrightarrow
 H^2(\mathbb Q,\mathbb G_m)=\operatorname{Br}(\mathbb Q).
\end{equation}

Let $\xi\in H^2(U,\mathbb F_2)$.  For $q\notin\Sigma$, restriction factors as
\begin{equation}\label{eq:obs-mechanical-unramified-local-vanishing}
 H^2(U,\mathbb F_2)\longrightarrow H^2(\mathbb Z_q,\mathbb F_2)=0
 \longrightarrow H^2(\mathbb Q_q,\mathbb F_2).
\end{equation}
The equality in \eqref{eq:obs-mechanical-unramified-local-vanishing}
follows by descent from the strict henselization of $\mathbb Z_q$:
a strictly henselian local scheme has no higher \'etale cohomology,
and $\operatorname{Gal}(\overline{\mathbb F}_q/\mathbb F_q)$
has torsion cohomological dimension one.
Suppose also that $\xi|_{\mathbb Q_p}=0$ for every $p\in\Sigma$.
Let $\beta\in\operatorname{Br}(\mathbb Q)[2]$ be the image of
$\xi$ under \eqref{eq:obs-mechanical-brauer-restriction-injections}.
Equation~\eqref{eq:obs-mechanical-unramified-local-vanishing} gives
$\operatorname{inv}_q(\beta|_{\mathbb Q_q})=0$ for every finite
prime $q$. The reciprocity law for Brauer invariants gives
\begin{equation}\label{eq:obs-mechanical-real-invariant}
 \operatorname{inv}_{\infty}(\beta|_{\mathbb R})
 =-\sum_{q\text{ prime}}\operatorname{inv}_q(\beta|_{\mathbb Q_q})=0.
\end{equation}
Brauer--Hasse--Noether therefore gives
\begin{equation}\label{eq:base-degree-two-detection}
 H^2(U,\mathbb F_2)\longrightarrow
 \bigoplus_{p\in\Sigma}H^2(\mathbb Q_p,\mathbb F_2)
 \quad\text{is injective}.
\end{equation}
See \cite[Chapter~VIII, Theorem~4.2]{MilneCFT}.

Apply the Leray spectral sequence (\st{03QC})
\begin{equation}\label{eq:obs-explicit-leray-sequence}
 E_2^{a,b}=H^a(U,R^bf_*\mathbb F_2)
 \ \Longrightarrow\ H^{a+b}(X,\mathbb F_2),
\end{equation}
and denote the induced filtration on $H^2(X,\mathbb F_2)$ by
\begin{equation}\label{eq:obs-explicit-leray-filtration}
 H^2(X,\mathbb F_2)=F^0\supseteq F^1\supseteq F^2\supseteq0.
\end{equation}
Proper smooth base change identifies $R^0f_*\mathbb F_2$ and
$R^2f_*\mathbb F_2$ with the constant rank-one sheaf.  For $R^2f_*\mathbb F_2$, local constancy follows from
\st{0GLI}; a rank-one $\mathbb F_2$-local system is constant because
$\operatorname{GL}_1(\mathbb F_2)=\{1\}$.
The equality $s^*\circ f^*=\operatorname{id}_{H^2(U,\mathbb F_2)}$
shows that $f^*:H^2(U,\mathbb F_2)\to H^2(X,\mathbb F_2)$ is
injective.  Therefore
\begin{equation}\label{eq:obs-explicit-leray-base-piece}
 F^2=f^*H^2(U,\mathbb F_2)
       \simeq H^2(U,\mathbb F_2).
\end{equation}
Thus the differential
$d_2^{0,1}:H^0(U,R^1f_*\mathbb F_2)\to H^2(U,\mathbb F_2)$
is zero.
There are no incoming differentials to the positions $(0,2)$ and $(1,1)$
on pages $r\geq2$.  Consequently the quotients $F^0/F^1$ and $F^1/F^2$ have natural
injections
\begin{equation}\label{eq:obs-explicit-leray-graded-injections}
 F^0/F^1=E_\infty^{0,2}
     \lhook\joinrel\longrightarrow H^0(U,R^2f_*\mathbb F_2),
 \qquad
 F^1/F^2=E_\infty^{1,1}
     \lhook\joinrel\longrightarrow H^1(U,R^1f_*\mathbb F_2).
\end{equation}
For $f_p:C_{\mathbb Q_p}\to\Spec\mathbb Q_p$, the same
argument gives a filtration $F_p^\bullet$ of
$H^2(C_{\mathbb Q_p},\mathbb F_2)$ and injections of
$F_p^0/F_p^1$ and $F_p^1/F_p^2$ into
$H^0(\mathbb Q_p,R^2f_{p*}\mathbb F_2)$ and
$H^1(\mathbb Q_p,R^1f_{p*}\mathbb F_2)$, respectively.
Proper base change (\st{095T}) identifies $R^bf_{p*}\mathbb F_2$
with the pullback of $R^bf_*\mathbb F_2$, so the restriction map
$H^2(X,\mathbb F_2)\to H^2(C_{\mathbb Q_p},\mathbb F_2)$
takes $F^i$ into $F_p^i$. The induced maps on $F^0/F^1$ and
$F^1/F^2$ commute with the injections in
\eqref{eq:obs-explicit-leray-graded-injections} and the corresponding
injections for $f_p$.

Let $\alpha\in H^2(X,\mathbb F_2)$ restrict to zero in
$H^2(C_{\mathbb Q_p},\mathbb F_2)$ for every $p\in\Sigma$.
The image of $\alpha$ in $F^0/F^1$ is zero, because
\begin{equation}\label{eq:obs-explicit-geometric-degree-two-restriction}
 H^0(U,R^2f_*\mathbb F_2)\longrightarrow
 H^0(\mathbb Q_p,R^2f_{p*}\mathbb F_2)
\end{equation}
is the identity on $\mathbb F_2$.  Hence $\alpha\in F^1$.
The image of $\alpha$ in
$F^1/F^2\subseteq H^1(U,R^1f_*\mathbb F_2)$ restricts to zero in
$H^1(\mathbb Q_p,R^1f_{p*}\mathbb F_2)$ for every $p\in\Sigma$.
Apply \eqref{eq:obs-mechanical-constant-coefficient-detection} with
$M=H^1_{\et}(C_{\overline\Q},\mathbb F_2)$, using the assumed
constancy of $R^1f_*\mathbb F_2$. The image of $\alpha$ in $F^1/F^2$
is zero, so $\alpha\in F^2$.  By \eqref{eq:obs-explicit-leray-base-piece},
$\alpha=f^*\beta$ for a unique
$\beta\in H^2(U,\mathbb F_2)$.  For the base-changed section
$s_p:\Spec\mathbb Q_p\to C_{\mathbb Q_p}$,
\begin{equation}\label{eq:obs-explicit-section-detects-base-class}
 \beta|_{\mathbb Q_p}
 =s_p^*f_p^*(\beta|_{\mathbb Q_p})
 =s_p^*(\alpha|_{C_{\mathbb Q_p}})=0.
\end{equation}
Equations~\eqref{eq:obs-explicit-section-detects-base-class} and
\eqref{eq:base-degree-two-detection} give $\beta=0$ and hence $\alpha=f^*\beta=0$. This proves
injectivity of \eqref{eq:obs-explicit-detection-restriction} with
$\mathbb Q_p$ in place of $K_p$.

Finally, the map $\mathbb Z_{(p)}\to\mathbb Z_p$ extends to the
henselization, since $\mathbb Z_p$ is henselian. The extended local homomorphism
$\mathbb Z_{(p)}^h\to\mathbb Z_p$ is injective: a nonzero kernel
in the discrete valuation ring $\mathbb Z_{(p)}^h$ would contain a
power of $p$, whereas every power of $p$ has nonzero image in
$\mathbb Z_p$. We obtain embeddings $K_p\hookrightarrow\mathbb Q_p$
and factorizations
\begin{equation}\label{eq:obs-mechanical-henselian-completed-restriction}
 H^2(X,\mathbb F_2)\longrightarrow H^2(C_{K_p},\mathbb F_2)
 \longrightarrow H^2(C_{\mathbb Q_p},\mathbb F_2).
\end{equation}
Injectivity after restriction to all $C_{\mathbb Q_p}$ therefore
implies \eqref{eq:obs-explicit-detection-restriction}.
\end{proof}

For $p\in\Sigma$ and $i\in\{1,2\}$, let
\begin{equation}\label{eq:obs-mechanical-local-restriction-definition}
 \operatorname{res}_p^i:
 H_A^i((S_0)_{B_p},\F)\longrightarrow H_A^i((S_0)_{K_p},\F)
 \xrightarrow{\ \sim\ }H^i_{\et}(C_{K_p},\F).
\end{equation}
The first arrow is pullback along
$(S_0)_{K_p}\to(S_0)_{B_p}$, and the second is descent along
the \'etale $A$-torsor $q_{K_p}:(S_0)_{K_p}\to C_{K_p}$.
Write
\begin{equation}\label{eq:obs-explicit-local-groups-and-restriction}
 \mathcal B_p^i=H_A^i((S_0)_{B_p},\F),\qquad
 Q_p=H^1_{\et}(C_{K_p},\F),\qquad
 \rho_p:\mathcal B_p^2\longrightarrow H^2_{\et}(C_{K_p},\F).
\end{equation}
Thus $\rho_p=\operatorname{res}_p^2$, while
$\operatorname{res}_p^1:\mathcal B_p^1\to Q_p$ is the restriction
map in degree one. Put $\alpha_p=\alpha_{B_p}$.  The group $Q_p$ classifies quadratic
torsors on $C_{K_p}$, including the descents of differences between
central lifts on $\widetilde C_{K_p}$.  If $P_U\to\widetilde S_U$ and
$P_p\to\widetilde S_{B_p}$ are central lifts of $\theta$, the
difference between $(P_U)_{K_p}$ and $(P_p)_{K_p}$ descends to a
class in $Q_p$. Twisting $P_U$ by a class
$\eta\in H^1_{\et}(X,\F)$ adds $\eta|_{C_{K_p}}$ to that class.
Twisting $P_p$ by a class $\eta_p\in\mathcal B_p^1$ adds
$\operatorname{res}_p^1(\eta_p)$. Accordingly, define
\begin{equation}\label{eq:twist-defect-group}
 \mathcal D_\Sigma=
 \frac{\displaystyle\bigoplus_{p\in\Sigma}Q_p}
 {\displaystyle\operatorname{res}H^1_{\et}(X,\F)
       +\bigoplus_{p\in\Sigma}
                    \operatorname{im}(\mathcal B_p^1\to Q_p)}.
\end{equation}
Here
\begin{equation}\label{eq:obs-mechanical-global-twist-restriction}
 \operatorname{res}:H^1_{\et}(X,\F)\longrightarrow
 \bigoplus_{p\in\Sigma}Q_p,\qquad
 \eta\longmapsto(\eta|_{C_{K_p}})_{p\in\Sigma},
\end{equation}
and $\operatorname{im}(\mathcal B_p^1\to Q_p)$ means
$\operatorname{im}(\operatorname{res}_p^1)$.
For a representative $D_p\to(S_0)_{B_p}$ of
$\eta_p\in\mathcal B_p^1$, the contracted product
$(P_p\times_{\widetilde S_{B_p}}r_{B_p}^*D_p)/\Delta\F$
is finite \'etale over $\widetilde S_{B_p}$. No descent of $D_p$
to an \'etale cover of $(S_{-1})_{B_p}$ is needed;
$q_{B_p}$ may be ramified.

\begin{proposition}\label{prop:geometric-twist-obstruction}
For an anchored diagram, put
\begin{equation}\label{eq:obs-explicit-rational-integral-kernels}
 \mathscr R=\ker\bigl((\rho_p\alpha_p)_p\bigr),\qquad
 \mathscr L=\ker\bigl((\alpha_p)_p\bigr)
 \quad\text{inside }H^2(G,\F).
\end{equation}
Then $\mathscr R$ is the space of extension classes admitting a central
lift $P_{\Q}\to\widetilde C$, and $\mathscr L$ is the space
admitting a central lift $P_p\to\widetilde S_{B_p}$ for every
$p\in\Sigma$.
There is a canonical linear map
\begin{equation}\label{eq:obs-explicit-compatibility-map}
 \delta:\mathscr L\longrightarrow\mathcal D_\Sigma,
 \qquad \ker\delta=\mathscr E(\widetilde S,S_{-1}).
\end{equation}
In particular
\begin{equation}\label{eq:obs-explicit-obstruction-filtration}
 H^2(G,\F)\supseteq\mathscr R\supseteq\mathscr L
       \supseteq\mathscr E(\widetilde S,S_{-1}),
\end{equation}
and the successive quotients inject, respectively, into
\begin{equation}\label{eq:obs-explicit-quotient-targets}
 \bigoplus_{p\in\Sigma}\operatorname{im}\rho_p,
 \qquad\bigoplus_{p\in\Sigma}\ker\rho_p,
 \qquad\mathcal D_\Sigma.
\end{equation}
Consequently
\begin{equation}\label{eq:bad-only-obstruction-budget}
 \dim_{\F}\mathcal O(\widetilde S,S_{-1})
 \le\sum_{p\in\Sigma}\dim_{\F}\mathcal B_p^2
       +\dim_{\F}\mathcal D_\Sigma.
\end{equation}
The groups $\mathcal B_p^2$ and $\mathcal D_\Sigma$ depend only
on $q:S_0\to S_{-1}$ and $\Sigma$, independently of
$r:\widetilde S\to S_0$.
\end{proposition}

\begin{proof}
If $\theta$ admits a central lift $P_{\Q}\to\widetilde C$,
the morphism $(P_{\Q})_{K_p}\to\widetilde C_{K_p}$ is a central lift, so
$\rho_p\alpha_p(\theta)=\alpha_{\Spec K_p}(\theta)=0$.
Conversely, if $\rho_p\alpha_p(\theta)=0$ for every $p\in\Sigma$,
then $\alpha_U(\theta)\in H^2_{\et}(X,\F)$ restricts to zero in
each $H^2_{\et}(C_{K_p},\F)$.  Proposition~\ref{prop:good-detection}
implies that $\alpha_U(\theta)=0$.  By
Proposition~\ref{prop:local-lift-criterion}, there is then a central lift
$P_U\to\widetilde S_U$, whose generic fibre is a central lift
$P_{\Q}\to\widetilde C$.  This proves the assertion about
$\mathscr R$. Moreover, Proposition~\ref{prop:good-prime-twist}
shows that, for every central lift $P_{\Q}\to\widetilde C$, there
is a squarefree $d>0$ such that
$P_{\Q}\star(C\times_{\Q}D_d)$ extends to a central lift over $U$.
By the definition of $\mathscr R$, we obtain the injection
\begin{equation}\label{eq:obs-mechanical-rational-obstruction-injection}
 H^2(G,\F)/\mathscr R\lhook\joinrel\longrightarrow
 \bigoplus_{p\in\Sigma}\operatorname{im}\rho_p,\qquad
 [\theta]\longmapsto(\rho_p\alpha_p(\theta))_{p\in\Sigma}.
\end{equation}

For $\theta\in\mathscr R$ and $p\in\Sigma$, the definition
of $\mathscr R$ gives $\alpha_p(\theta)\in\ker\rho_p$.
Proposition~\ref{prop:local-lift-criterion} gives the equivalences
\begin{equation}\label{eq:obs-mechanical-integral-local-obstruction}
 \theta\in\mathscr L
 \ \Longleftrightarrow\ 
 \alpha_p(\theta)=0\text{ for every }p\in\Sigma
 \ \Longleftrightarrow\ 
 \text{a central lift over }B_p\text{ exists for every }p\in\Sigma.
\end{equation}
Consequently we obtain the injection
\begin{equation}\label{eq:obs-mechanical-local-integral-injection}
 \mathscr R/\mathscr L\lhook\joinrel\longrightarrow
 \bigoplus_{p\in\Sigma}\ker\rho_p,\qquad
 [\theta]\longmapsto(\alpha_p(\theta))_{p\in\Sigma}.
\end{equation}

Let $\theta\in\mathscr L$. Choose central lifts
$P_U\to\widetilde S_U$ and $P_p\to\widetilde S_{B_p}$ for each
$p\in\Sigma$. Define the following torsor on $\widetilde C_{K_p}$:
\begin{equation}\label{eq:obs-explicit-local-difference-quotient}
 D_p:=\bigl((P_U)_{K_p}\times_{\widetilde C_{K_p}}(P_p)_{K_p}\bigr)
       /\Delta\F.
\end{equation}
As in Proposition~\ref{prop:local-lift-criterion}, the diagonal
$E_\theta$-action on $D_p$ factors through $G$. Descent along
the $G$-torsor $\widetilde C_{K_p}\to C_{K_p}$ therefore gives
a quadratic torsor $\overline D_p\to C_{K_p}$. Put
\begin{equation}\label{eq:obs-mechanical-delta-definition}
 \chi_p=[\overline D_p]\in Q_p,\qquad
 \delta(\theta)=[(\chi_p)_{p\in\Sigma}]\in\mathcal D_\Sigma.
\end{equation}
Twisting $P_U$ by $\eta\in H^1_{\et}(X,\F)$ and $P_p$ by
$\eta_p\in\mathcal B_p^1$ replaces $\chi_p$ by
\begin{equation}\label{eq:obs-mechanical-choice-change}
 \chi_p+\eta|_{C_{K_p}}+\operatorname{res}_p^1(\eta_p).
\end{equation}
By \eqref{eq:twist-defect-group}, the tuples in
\eqref{eq:obs-mechanical-choice-change} and
\eqref{eq:obs-mechanical-delta-definition} have the same image in
$\mathcal D_\Sigma$. Thus $\delta(\theta)$ is independent of
the choices of $P_U$ and $P_p$.
An isomorphism $E_\theta\simeq E'_\theta$ inducing the identity
on $\F$ and $G$ transports the actions on $P_U$ and every $P_p$.
The induced $G$-actions on $D_p$ remain unchanged, so the classes
$\chi_p\in Q_p$ do not depend on the representative extension.
For $\theta_1,\theta_2\in\mathscr L$ and
$V\in\{U\}\cup\{B_p:p\in\Sigma\}$, choose central lifts
$P_{1,V},P_{2,V}\to\widetilde S_V$ and put
\begin{equation}\label{eq:obs-mechanical-baer-lifts}
 P_{12,V}=(P_{1,V}\times_{\widetilde S_V}P_{2,V})/\Delta\F,
 \qquad E_{12}=(E_{\theta_1}\times_G E_{\theta_2})/\Delta\F.
\end{equation}
Proposition~\ref{prop:realizable-subspace} gives an $E_{12}$-action
on $P_{12,V}$ and identifies the class of $E_{12}$ with
$\theta_1+\theta_2$. Let $\chi_{i,p}$ be the class defined by
\eqref{eq:obs-mechanical-delta-definition} using $P_{i,U},P_{i,B_p}$.
The difference between $(P_{12,U})_{K_p}$ and
$(P_{12,B_p})_{K_p}$ is the contracted product of the torsors
representing $\chi_{1,p}$ and $\chi_{2,p}$. Hence
\begin{equation}\label{eq:obs-mechanical-delta-additivity}
 \delta(\theta_1+\theta_2)
 =[(\chi_{1,p}+\chi_{2,p})_{p\in\Sigma}]
 =\delta(\theta_1)+\delta(\theta_2).
\end{equation}

If $P_{\Z}\to\widetilde S$ is a central lift of $\theta$, take
$P_U=(P_{\Z})_U$ and $P_p=(P_{\Z})_{B_p}$. Then
$(P_U)_{K_p}=(P_p)_{K_p}$, so $\chi_p=0$ for every $p\in\Sigma$.
Thus $\delta(\theta)=0$.  Conversely, suppose that
$\delta(\theta)=0$.  Then there are
$\eta\in H^1_{\et}(X,\F)$ and $\eta_p\in\mathcal B_p^1$
with
\begin{equation}\label{eq:obs-mechanical-compatible-twists}
 \chi_p=\eta|_{C_{K_p}}+\operatorname{res}_p^1(\eta_p)
 \quad\text{in }Q_p\quad(p\in\Sigma).
\end{equation}  Twist $P_U$ by the pullback to
$\widetilde S_U$ of a torsor representing $\eta$, and twist $P_p$
by the pullback along $r_{B_p}$ of an $A$-equivariant torsor
representing $\eta_p$.  Keep the notation $P_U,P_p$ for the twisted torsors. By
\eqref{eq:obs-mechanical-choice-change} and
\eqref{eq:obs-mechanical-compatible-twists}, the difference between
$(P_U)_{K_p}$ and $(P_p)_{K_p}$ now represents $0\in Q_p$.
Proposition~\ref{prop:local-lift-criterion} therefore gives
$E_\theta$-equivariant isomorphisms
\begin{equation}\label{eq:obs-explicit-generic-lift-isomorphisms}
 (P_U)_{K_p}\simeq(P_p)_{K_p}
       \quad\text{over }\widetilde C_{K_p}.
\end{equation}
For $\theta=0$, the identity cover of $\widetilde S$ realizes
$\theta$ by definition. Suppose that $\theta\ne0$; then $P_U$
is connected. Let $M=\Q(P_U)$, and let
$\mathcal N\to\widetilde S$ be the normalization of
$\widetilde S$ in $M/\Q(\widetilde C)$. Excellence of
$\widetilde S$ makes $\mathcal N\to\widetilde S$ finite.
The scheme $P_U$ is normal, finite over $\widetilde S_U$, and has
function field $M$. Uniqueness of normalization gives
\begin{equation}\label{eq:obs-mechanical-global-normalization-restriction}
 \mathcal N_U\simeq P_U\quad\text{over }\widetilde S_U.
\end{equation}

Put $L_p=K_p(\widetilde C_{K_p})$.  Normalization commutes with
\'etale localization and henselization (\st{03GE}, \st{0CBF}), so
$\mathcal N_{B_p}$ is the integral closure of $\widetilde S_{B_p}$ in
$M\otimes_{\Q(\widetilde C)}L_p$.  The isomorphism
\eqref{eq:obs-explicit-generic-lift-isomorphisms} identifies
$M\otimes_{\Q(\widetilde C)}L_p$ with the algebra of the generic
fibre of $P_p\to\widetilde S_{B_p}$. Since $P_p$ is normal and
finite over $\widetilde S_{B_p}$, uniqueness of integral closure
gives
\begin{equation}\label{eq:obs-mechanical-henselian-normalization-restriction}
 \mathcal N_{B_p}\simeq P_p\quad\text{over }\widetilde S_{B_p}.
\end{equation}
Consequently $\mathcal N_{B_p}\to\widetilde S_{B_p}$ is \'etale.  Faithfully flat descent along
$\Spec\Z_{(p)}^h\to\Spec\Z_{(p)}$ then shows that
$\mathcal N\to\widetilde S$ is \'etale at the points over $p$.  Equation~\eqref{eq:obs-mechanical-global-normalization-restriction}
also proves that $\mathcal N_U\to\widetilde S_U$ is \'etale.
Since every point of $\widetilde S$ lies over $U$ or over a prime
in $\Sigma$, the morphism $\mathcal N\to\widetilde S$ is
\'etale. For an affine open $V=\operatorname{Spec}A_V\subseteq S_{-1}$,
write $(q\circ r)^{-1}(V)=\operatorname{Spec}B_V$ and let $D_V$
be the integral closure of $B_V$ in $M$. The group $E_\theta$
fixes $A_V$ pointwise and preserves $\Q(\widetilde C)$.
Since $B_V$ is the integral closure of $A_V$ in
$\Q(\widetilde C)$, every $e\in E_\theta$ satisfies
\begin{equation}\label{eq:obs-mechanical-equivariant-normalization-rings}
 e(B_V)=B_V,\qquad e(D_V)=D_V.
\end{equation}
The actions on $\operatorname{Spec}D_V$ agree after restriction to
$V\cap V'$ for any affine opens $V,V'\subseteq S_{-1}$, because
all actions are induced by the action on $M$. Hence the actions on
$\operatorname{Spec}D_V$ define an $E_\theta$-action on $\mathcal N$.
Proposition~\ref{prop:realizing-surface-properties} proves that
$\mathcal N$ is regular, integral, projective and flat over
$\operatorname{Spec}\mathbb Z$, that $\mathcal N_{\Q}$ is
smooth and geometrically integral, and that
$\mathcal N\to\operatorname{Spec}\mathbb Z$ has a section.
Thus $\theta\in\mathscr E(\widetilde S,S_{-1})$, proving
$\ker\delta=\mathscr E(\widetilde S,S_{-1})$.

Finally, the definition of $\mathcal O$ and the inclusions in
\eqref{eq:obs-explicit-obstruction-filtration} give
\begin{equation}\label{eq:obs-mechanical-obstruction-dimension-sum}
 \begin{aligned}
 \dim_{\F}\mathcal O(\widetilde S,S_{-1})
 &=\dim_{\F}(H^2(G,\F)/\mathscr R)
   +\dim_{\F}(\mathscr R/\mathscr L)
   +\dim_{\F}(\mathscr L/\mathscr E)\\
 &\le\sum_{p\in\Sigma}
   \bigl(\dim_{\F}\operatorname{im}\rho_p+\dim_{\F}\ker\rho_p\bigr)
     +\dim_{\F}\mathcal D_\Sigma\\
 &=\sum_{p\in\Sigma}\dim_{\F}\mathcal B_p^2
     +\dim_{\F}\mathcal D_\Sigma.
 \end{aligned}
\end{equation}
The inequality uses \eqref{eq:obs-mechanical-rational-obstruction-injection},
\eqref{eq:obs-mechanical-local-integral-injection} and
$\ker\delta=\mathscr E$. This proves
\eqref{eq:bad-only-obstruction-budget}.
\end{proof}

Let $z$ be a geometric point of $\widetilde S$, and let $x=r(z)$
and $y=q(x)$.  Write
$I_G(z)=\{g\in G:gz=z\}$ and $I_A(x)=\{a\in A:ax=x\}$.
The quotient homomorphism $G\to A$ restricts to an isomorphism
\begin{equation}\label{eq:obs-explicit-stabilizer-isomorphism}
 I_G(z)\simeq I_A(x).
\end{equation}
The kernel of $I_G(z)\to I_A(x)$ equals $I_G(z)\cap H$.
Since $r:\widetilde S\to S_0$ is an $H$-torsor, $H$ acts freely
on $\widetilde S$, and $I_G(z)\cap H=\{1\}$.  Given $a\in I_A(x)$, choose $g\in G$ mapping to
$a$.  The points $gz$ and $z$ lie in the same fibre of the
$H$-torsor $r$, so there is a unique $h\in H$ with $hgz=z$.
Then $hg\in I_G(z)$ maps to $a$.  Therefore \eqref{eq:obs-explicit-stabilizer-isomorphism} identifies
every stabilizer for the action of $G$ on $\widetilde S$ with a
stabilizer for the action of $A$ on $S_0$.

We now bound $\mathcal B_p^2=H_A^2((S_0)_{B_p},\F)$ using the
fibres of $q$ over $p\in\Sigma$, and bound $\mathcal D_\Sigma$
by the dimensions of $Q_p=H^1_{\et}(C_{K_p},\F)$.

\begin{proposition}\label{prop:realizable-local-bounds}
For an anchored diagram, with $t=|\Sigma|-1$, one has for every
$p\in\Sigma$
\begin{equation}\label{eq:obs-explicit-local-dimension-bounds}
 \dim_{\F}\mathcal B_p^2\le
 \begin{cases}
 324g+324,&p\ne2,\\[2pt]
 (3g+1)/2,&p=2,
 \end{cases}
 \qquad
 \dim_{\F}\mathcal D_\Sigma\le(2g+2)t+(2g+3).
\end{equation}
\end{proposition}

\begin{proof}
Fix $p\in\Sigma$, and put
\begin{equation}\label{eq:obs-explicit-geometric-fibres}
 T=(S_0\times_{\Z}\overline{\mathbb F}_p)_{\mathrm{red}},
 \qquad
 Y=(S_{-1}\times_{\Z}\overline{\mathbb F}_p)_{\mathrm{red}}.
\end{equation}
Let $q_p:T\to Y$ be the finite morphism induced by $q:S_0\to S_{-1}$.
On the \'etale site of $Y$, define the group-cochain complex
$\mathcal K^\bullet$ and its cohomology sheaves $\mathcal H^b$ by
\begin{equation}\label{eq:obs-explicit-stabilizer-cochain-sheaf}
 \mathcal K^\bullet=C^\bullet(A,q_{p*}\F),\qquad
 \mathcal K^n=\prod_{A^n}q_{p*}\F,\qquad
 \mathcal H^b=H^b(\mathcal K^\bullet).
\end{equation}
The direct image functor $q_{p*}$ is exact on \'etale sheaves.
Indeed, for a geometric point $\bar y$ of $Y$, there is a decomposition
\begin{equation}\label{eq:obs-mechanical-strict-local-finite-decomposition}
 T\times_Y\Spec\mathcal O^{\sh}_{Y,\bar y}=\coprod_j T_j,
\end{equation}
where each $T_j$ is the spectrum of a strictly henselian local ring.  For every abelian
\'etale sheaf $\mathcal G$ on $T_j$,
$H^i_{\et}(T_j,\mathcal G)=0$ for $i>0$; see \st{03QP},
\st{03QO}.  The stalk formula therefore gives
$(R^iq_{p*}\mathcal G)_{\bar y}=0$ for every $i>0$, every abelian
\'etale sheaf $\mathcal G$ on $T$, and every geometric point
$\bar y$ of $Y$. Hence $R^iq_{p*}=0$ for $i>0$.
The hypercohomology spectral sequence of $\mathcal K^\bullet$ is
\begin{equation}\label{eq:realizable-stabilizer-sequence}
 H^a_{\et}(Y,\mathcal H^b)\ \Longrightarrow\ H_A^{a+b}(T,\F),
 \qquad
 (\mathcal H^b)_{\bar y}=H^b(I_{\bar y},\F),\qquad
 \mathcal H^0=\F.
\end{equation}
Here $I_{\bar y}=\{a\in A:a\bar x=\bar x\}$ for a chosen
geometric point $\bar x\in T$ over $\bar y\in Y$.
For the stalk formula in \eqref{eq:realizable-stabilizer-sequence},
we first prove that $A$ acts transitively on the geometric points of
$q_p^{-1}(\bar y)$. Choose an affine neighbourhood
$\Spec R\subseteq S_{-1}$ of the image of $\bar y$.
Write $q^{-1}(\Spec R)=\Spec B$.  Since $S_0$ is the normalization
of $S_{-1}$ in $\Q((S_0)_{\Q})$, $B$ is the integral closure of
$R$ in that field and $B^A=R$.  For
$R^{\sh}=\mathcal O^{\sh}_{S_{-1},\bar y}$, flat base change gives
\begin{equation}\label{eq:obs-explicit-strict-local-invariants}
 (B\otimes_R R^{\sh})^A=R^{\sh}.
\end{equation}
Indeed, invariants are the kernel of the finite collection of maps
$b\mapsto a(b)-b$, for $a\in A$, and this kernel commutes with flat
base change.  The finite $R^{\sh}$-algebra
$B\otimes_R R^{\sh}$ is a product of local rings, one for each
geometric point over $\bar y$.  Write $B\otimes_RR^{\sh}=\prod_{j\in J}B_j$ with $B_j$ local,
and let $e_j$ be the idempotent equal to $1$ in $B_j$ and $0$ in
every $B_k$ for $k\ne j$. If an $A$-orbit $J'\subseteq J$
satisfied $\varnothing\ne J'\ne J$, then
$e=\sum_{j\in J'}e_j$ would satisfy
\begin{equation}\label{eq:obs-mechanical-invariant-idempotent}
 e^2=e,\qquad e\notin\{0,1\},\qquad
 e\in(B\otimes_RR^{\sh})^A=R^{\sh}.
\end{equation}
A local ring has only the idempotents $0$ and $1$, contradicting
\eqref{eq:obs-mechanical-invariant-idempotent}. Hence $J$ is a
single $A$-orbit. This proves transitivity on $q_p^{-1}(\bar y)$,
also when $p=2$. It follows that
\begin{equation}\label{eq:obs-mechanical-shapiro-stalk}
 (q_{p*}\F)_{\bar y}\simeq\F[A/I_{\bar y}],\qquad
 (\mathcal H^b)_{\bar y}
 \simeq H^b(A,\F[A/I_{\bar y}])
 \simeq H^b(I_{\bar y},\F),
\end{equation}
where the last isomorphism is Shapiro's lemma.
Passing from $S_0\times_{\mathbb Z}\overline{\mathbb F}_p$ to $T$
and from $S_{-1}\times_{\mathbb Z}\overline{\mathbb F}_p$ to $Y$
removes nilpotents. The \'etale sites are unchanged by either
reduction (\st{03SI}).

Suppose first that $p$ is odd.  Write $Y_1,\ldots,Y_c$ for the
irreducible components of $Y$, $Z=\operatorname{Sing}(Y)$ for the set of nodes of $Y$,
and $Y_i^\circ=Y_i\setminus Z$.  Put $s=|Z|=s_p$ and $a=a_p$, so $c=c_p$.
For a geometric point $\bar y\in Y$, put
$R_{\bar y}=\mathcal O^{\sh}_{S_{-1},\bar y}$ and
$F_{\bar y}=\operatorname{Frac}(R_{\bar y})$.
The ring $R_{\bar y}$ is regular and factorial. Since $p\ne2$,
Hensel's lemma gives $R_{\bar y}^{\times}=R_{\bar y}^{\times2}$.
For $\chi\in\operatorname{Hom}(A,\F)\setminus\{0\}$, choose
$a_\chi\in\Q(C)^\times$ such that
\begin{equation}\label{eq:obs-mechanical-character-subfields}
 \Q((S_0)_{\Q})^{\ker\chi}=\Q(C)(\sqrt{a_\chi}),
\end{equation}
and put $a_0=1$. Since $q_{\Q}$ is \'etale, the valuation of
$a_\chi$ along every horizontal prime divisor of $S_{-1}$ is even.
If $\bar y\in Y_i^\circ$, choose a local equation $u=0$ for the
component of the fibre through $\bar y$. If $\bar y\in Z$, choose
local equations $u=0$ and $v=0$ for the two components through
$\bar y$. Unique factorization in $R_{\bar y}$ gives
\begin{equation}\label{eq:obs-mechanical-local-character-square-classes}
 [a_\chi]\in
 \begin{cases}
   \langle[u]\rangle\subseteq F_{\bar y}^{\times}/F_{\bar y}^{\times2},
          &\bar y\in Y_i^\circ,\\
   \langle[u],[v]\rangle\subseteq F_{\bar y}^{\times}/F_{\bar y}^{\times2},
          &\bar y\in Z.
 \end{cases}
\end{equation}
Consequently
\begin{equation}\label{eq:obs-mechanical-odd-stabilizer-ranks}
 \dim_{\F}I_{\bar y}\le1\quad(\bar y\in Y_i^\circ),\qquad
 \dim_{\F}I_{\bar y}\le2\quad(\bar y\in Z).
\end{equation}
Let $B_p^{\sh}=\operatorname{Spec}(\mathbb Z_{(p)}^{\sh})$,
where the strict henselization is taken with residue field
$\overline{\mathbb F}_p$, and put $\mathscr Y=(S_{-1})_{B_p^{\sh}}$.
Let $\eta_i$ be the generic point of $Y_i\subseteq\mathscr Y$,
and let $v_i$ be the normalized discrete valuation of
$\mathcal O_{\mathscr Y,\eta_i}$. Define
\begin{equation}\label{eq:obs-mechanical-component-stabilizer}
 V_i=\{\chi\in\operatorname{Hom}(A,\F):v_i(a_\chi)\in2\mathbb Z\},
 \qquad I_i=\bigcap_{\chi\in V_i}\ker\chi\subseteq A.
\end{equation}
For every $\bar y\in Y_i^\circ$, equation
\eqref{eq:obs-mechanical-local-character-square-classes} identifies
$I_{\bar y}$ with $I_i$. More precisely, over $R_{\bar y}$ the
normalization has no ramified factor other than the quadratic algebra
$R_{\bar y}[t]/(t^2-u)$. The reduction of this algebra modulo $u$
is $(R_{\bar y}/(u))[t]/(t^2)$, whose reduced ring is
$R_{\bar y}/(u)$. It follows that
\begin{equation}\label{eq:obs-mechanical-component-residual-torsor}
 (q_p^{-1}(Y_i^\circ))_{\mathrm{red}}\longrightarrow Y_i^\circ
 \quad\text{is an \'etale }A/I_i\text{-torsor}.
\end{equation}
Pulling back $\mathcal H^b|_{Y_i^\circ}$ along
\eqref{eq:obs-mechanical-component-residual-torsor} gives the constant
sheaf with fibre $H^b(I_i,\F)$. The descent automorphism associated
with $a\in A$ is induced by $h\mapsto aha^{-1}$ on $I_i$.
Since $A$ is abelian, $aha^{-1}=h$, so every descent automorphism
is the identity. Thus
\begin{equation}\label{eq:obs-mechanical-constant-stabilizer-cohomology}
 \mathcal H^b|_{Y_i^\circ}\simeq
       \underline{H^b(I_i,\F)}_{Y_i^\circ}.
\end{equation}

For $I\simeq\F^r$,
\begin{equation}\label{eq:obs-explicit-elementary-abelian-cohomology}
 \dim_{\F}H^1(I,\F)=r,\qquad
 \dim_{\F}H^2(I,\F)=r(r+1)/2,
\end{equation}
as one sees by tensoring the periodic resolutions of $C_2$.
Equations~\eqref{eq:obs-mechanical-shapiro-stalk},
\eqref{eq:obs-mechanical-odd-stabilizer-ranks} and
\eqref{eq:obs-explicit-elementary-abelian-cohomology} give
\begin{equation}\label{eq:obs-mechanical-stalk-dimension-bounds}
 \dim_{\F}(\mathcal H^1)_{\bar y}\le2,\qquad
 \dim_{\F}(\mathcal H^2)_{\bar y}\le3\qquad(\bar y\in Y).
\end{equation}  Let
$\nu:\widetilde Y=\coprod_{i=1}^c\widetilde Y_i\to Y$ be the
normalization, let $a_i$ be the genus of $\widetilde Y_i$, and let
$i_z:\{z\}\to Y$ be the inclusion of a node.  Let ${\F}_Y$ and ${\F}_{\widetilde Y}$ be the constant sheaves
on $Y$ and $\widetilde Y$, respectively. There is an exact sequence
\begin{equation}\label{eq:obs-explicit-normalization-sequence}
 0\longrightarrow{\F}_Y\longrightarrow\nu_*{\F}_{\widetilde Y}
 \longrightarrow\bigoplus_{z\in Z}i_{z*}\F\longrightarrow0
\end{equation}
in which the component at $z$ of the last map sends a pair of
values $(a,b)\in\F\oplus\F$ at the two points of $\nu^{-1}(z)$
to $a-b$. The cohomology groups of the smooth projective curve
$\widetilde Y_i$ satisfy
$\dim_{\F}H^1_{\et}(\widetilde Y_i,\F)=2a_i$ and
$\dim_{\F}H^2_{\et}(\widetilde Y_i,\F)=1$.
Since a sheaf supported on the finite set $Z$ has zero cohomology in
positive degree, the long exact sequence of
\eqref{eq:obs-explicit-normalization-sequence} gives
\begin{equation}\label{eq:obs-explicit-nodal-cohomology-bounds}
 h^1(Y)\le2a+s,\qquad h^2(Y)=c.
\end{equation}

More generally, let $\mathcal F$ be a constructible $\F$-sheaf on
$Y$, of stalk rank at most $b$, whose restriction to each
$Y_i^\circ$ is constant.  Then
\begin{equation}\label{eq:obs-explicit-constructible-cohomology-bounds}
 \dim H^0_{\et}(Y,\mathcal F)\le b(c+s),\qquad
 \dim H^1_{\et}(Y,\mathcal F)\le b(2a+2s+2c).
\end{equation}
Choose a geometric point $\bar\eta_i\in Y_i^\circ$ for each $i$.
Restriction defines an injection
\begin{equation}\label{eq:obs-mechanical-constructible-section-injection}
 H^0_{\et}(Y,\mathcal F)\lhook\joinrel\longrightarrow
 \bigoplus_{i=1}^{c}\mathcal F_{\bar\eta_i}
 \oplus\bigoplus_{z\in Z}\mathcal F_z.
\end{equation}
Indeed, $\mathcal F|_{Y_i^\circ}$ is constant and $Y_i^\circ$ is
connected, so a section vanishing at $\bar\eta_i$ vanishes on
$Y_i^\circ$. A section with zero image in
\eqref{eq:obs-mechanical-constructible-section-injection} therefore
vanishes at every geometric point of $Y$ and is zero. Each summand
on the right has dimension at most $b$, proving
$\dim_{\F}H^0_{\et}(Y,\mathcal F)\le b(c+s)$.
To bound $\dim_{\F}H^1_{\et}(Y,\mathcal F)$, put
$Y^\circ=Y\setminus Z$, and write
$j:Y^\circ\hookrightarrow Y$ and $i:Z\hookrightarrow Y$.
The exact sequence
\begin{equation}\label{eq:obs-mechanical-constructible-localization}
 0\longrightarrow j_!j^*\mathcal F\longrightarrow\mathcal F
 \longrightarrow i_*i^*\mathcal F\longrightarrow0
\end{equation}
and $H^1_{\et}(Z,i^*\mathcal F)=0$ give a surjection
\begin{equation}\label{eq:obs-mechanical-compact-support-surjection}
 H_c^1(Y^\circ,j^*\mathcal F)
 \twoheadrightarrow H^1_{\et}(Y,\mathcal F).
\end{equation}  Put
$d_i=|\nu^{-1}(Z)\cap\widetilde Y_i|$, so
$Y_i^\circ\simeq\widetilde Y_i\setminus\nu^{-1}(Z)$.
The compactly supported cohomology of $Y_i^\circ$ satisfies
\begin{equation}\label{eq:obs-mechanical-punctured-component-cohomology}
 \dim_{\F}H_c^1(Y_i^\circ,\F)=
 \begin{cases}2a_i+d_i-1,&d_i>0,\\2a_i,&d_i=0.\end{cases}
\end{equation}
Let $b_i=\dim_{\F}\mathcal F_{\bar\eta_i}\le b$.
Since $\sum_i d_i=2s$ and $\sum_i a_i=a$, equations
\eqref{eq:obs-mechanical-compact-support-surjection} and
\eqref{eq:obs-mechanical-punctured-component-cohomology} give
\begin{equation}\label{eq:obs-mechanical-constructible-first-dimension}
 \dim_{\F}H^1_{\et}(Y,\mathcal F)
 \le\sum_{i=1}^c b_i\dim_{\F}H_c^1(Y_i^\circ,\F)
 \le b\sum_{i=1}^c(2a_i+d_i)=b(2a+2s).
\end{equation}
This implies the stated bound
$\dim_{\F}H^1_{\et}(Y,\mathcal F)\le b(2a+2s+2c)$ in
\eqref{eq:obs-explicit-constructible-cohomology-bounds}.
For the cohomology formulas used in
\eqref{eq:obs-explicit-nodal-cohomology-bounds} and
\eqref{eq:obs-mechanical-punctured-component-cohomology}, and the
localization sequence \eqref{eq:obs-mechanical-constructible-localization},
see \st{03RQ}, \st{03S2}.

The filtration on $H_A^n(T,\F)$ obtained from
\eqref{eq:realizable-stabilizer-sequence} has associated graded
pieces that are subquotients of
$H^a_{\et}(Y,\mathcal H^b)$ for $a+b=n$. For $n=1,2$,
apply \eqref{eq:obs-explicit-nodal-cohomology-bounds},
\eqref{eq:obs-explicit-constructible-cohomology-bounds} and
\eqref{eq:obs-mechanical-stalk-dimension-bounds}. We obtain
\begin{equation}\label{eq:obs-explicit-odd-geometric-equivariant-bounds}
 \begin{aligned}
 h_A^1(T)&\le h^1(Y)+\dim H^0_{\et}(Y,\mathcal H^1)
              \le2a+2c+3s,\\
 h_A^2(T)&\le h^2(Y)+\dim H^1_{\et}(Y,\mathcal H^1)
                    +\dim H^0_{\et}(Y,\mathcal H^2)
              \le4a+8c+7s.
 \end{aligned}
\end{equation}
For $p=2$, also put
$B_2^{\sh}=\operatorname{Spec}(\mathbb Z_{(2)}^{\sh})$, with
residue field $\overline{\mathbb F}_2$.  Proper base change (\st{095T}) gives $A$-equivariant
isomorphisms on ordinary cohomology and therefore gives
\begin{equation}\label{eq:obs-mechanical-equivariant-proper-base-change}
 H_A^b((S_0)_{B_p^{\sh}},\F)\simeq H_A^b(T,\F).
\end{equation}
Write $\Phi_p=\operatorname{Gal}(\overline{\mathbb F}_p/\mathbb F_p)$.
Descent along $B_p^{\sh}\to B_p$ gives the spectral sequence
\begin{equation}\label{eq:obs-explicit-henselian-descent-sequence}
 H^a\bigl(\Gal(\overline{\mathbb F}_p/\mathbb F_p),
                 H_A^b(T,\F)\bigr)
 \ \Longrightarrow\ H_A^{a+b}((S_0)_{B_p},\F).
\end{equation}
The group $\Phi_p$ has torsion cohomological dimension one. Thus
\eqref{eq:obs-explicit-henselian-descent-sequence} gives an exact sequence
\begin{equation}\label{eq:obs-mechanical-henselian-degree-two-exact}
 0\longrightarrow H^1(\Phi_p,H_A^1(T,\F))
 \longrightarrow\mathcal B_p^2
 \longrightarrow H^0(\Phi_p,H_A^2(T,\F))\longrightarrow0.
\end{equation}
For a finite-dimensional $\F$-vector space $W$ with a continuous
$\Phi_p$-action, let $\varphi_p$ be a topological generator of
$\Phi_p$. Then
\begin{equation}\label{eq:obs-mechanical-finite-field-cohomology}
 H^0(\Phi_p,W)=\ker(\varphi_p-1:W\to W),\qquad
 H^1(\Phi_p,W)=\operatorname{coker}(\varphi_p-1:W\to W).
\end{equation}
The kernel and cokernel in
\eqref{eq:obs-mechanical-finite-field-cohomology} have dimension at
most $\dim_{\F}W$. Equation
\eqref{eq:obs-mechanical-henselian-degree-two-exact} therefore gives
\begin{equation}\label{eq:realizable-trait-bound}
 \dim\mathcal B_p^2\le h_A^2(T)+h_A^1(T).
\end{equation}
For odd $p$, equations \eqref{eq:realizable-trait-bound},
\eqref{eq:obs-explicit-odd-geometric-equivariant-bounds} and
\eqref{eq:real-explicit-fibre-bounds} give
\begin{equation}\label{eq:obs-explicit-odd-trait-bound}
 \dim\mathcal B_p^2\le6a_p+10c_p+10s_p
       \le162N=324g+324.
\end{equation}

For $p=2$, let $\nu:\widetilde Y\to Y=Y_2$ be the normalization
and let $Z=\operatorname{Sing}(Y)$.  For every $z\in Z$, the set $\nu^{-1}(z)$ has two points.
Consequently the cokernel of
${\F}_Y\to\nu_*{\F}_{\widetilde Y}$ has stalk $\F$ at $z$
and zero stalk at each point of $Y\setminus Z$, so
\eqref{eq:obs-explicit-normalization-sequence} also holds for $p=2$.
Write $\widetilde Y=\coprod_{i=1}^c\widetilde Y_i$. By the
hypotheses on an anchored diagram,
$\widetilde Y_i\simeq\mathbb P^1_{\overline{\mathbb F}_2}$.
Since $Y$ is connected and
$H^1_{\et}(\mathbb P^1_{\overline{\mathbb F}_2},\F)=0$, the
long exact sequence of \eqref{eq:obs-explicit-normalization-sequence}
reduces in degrees zero and one to
\begin{equation}\label{eq:obs-mechanical-dyadic-normalization-cohomology}
 0\longrightarrow\F\longrightarrow\F^c
 \xrightarrow{\partial}\F^{|Z|}
 \longrightarrow H^1_{\et}(Y,\F)\longrightarrow0.
\end{equation}
The coordinate of $\partial(a_1,\ldots,a_c)$ at $z$ is the
difference of the values on the two components containing the points
of $\nu^{-1}(z)$. Thus $\partial$ is the incidence map of
$\Gamma_Y=\Gamma(Y)$ from
\eqref{eq:bg-explicit-incidence-graph}, and
\begin{equation}\label{eq:obs-explicit-dyadic-incidence-bound}
 h^1(Y)=b_1(\Gamma_Y)\le(g-1)/2,
\end{equation}
by the assumed bound on $b_1(\Gamma_Y)$.

To compute $H^2_{\et}(Y,\F)$, use the Artin--Schreier sequence
\begin{equation}\label{eq:obs-mechanical-artin-schreier-sequence}
 0\longrightarrow\F\longrightarrow\mathcal O_Y
 \xrightarrow{F-1}\mathcal O_Y\longrightarrow0,
 \qquad F(f)=f^2.
\end{equation}
Since $Y$ is a proper curve, $H^i(Y,\mathcal O_Y)=0$ for $i>1$,
and $H^1(Y,\mathcal O_Y)$ is finite-dimensional over
$\overline{\mathbb F}_2$. The induced map $F$ on
$H^1(Y,\mathcal O_Y)$ satisfies $F(av)=a^2F(v)$. For an
endomorphism with this semilinearity property, the map $F-1$ is
surjective; see \st{0A3J}, \st{0A3L}. The long exact sequence of
\eqref{eq:obs-mechanical-artin-schreier-sequence} gives
\begin{equation}\label{eq:obs-mechanical-dyadic-degree-two-vanishing}
 H^2_{\et}(Y,\F)
 =\operatorname{coker}\bigl(F-1:H^1(Y,\mathcal O_Y)
                                  \to H^1(Y,\mathcal O_Y)\bigr)=0.
\end{equation}
By the hypotheses on the stabilizers, there is a finite subset
$Z_A\subseteq Y$ with $|Z_A|\le b=(g+1)/2$ such that
$I_{\bar y}=\{1\}$ for $\bar y\notin Z_A$ and
$|I_{\bar y}|\le2$ for $\bar y\in Z_A$.
Equations~\eqref{eq:obs-mechanical-shapiro-stalk} and
\eqref{eq:obs-explicit-elementary-abelian-cohomology} therefore give
\begin{equation}\label{eq:obs-mechanical-dyadic-supported-sheaf-bounds}
 \begin{gathered}
 \operatorname{Supp}(\mathcal H^j)\subseteq Z_A,\qquad
 \dim_{\F}(\mathcal H^j)_{\bar y}\le1\quad(j=1,2),\\
 \dim_{\F}H^0_{\et}(Y,\mathcal H^j)\le b,\qquad
 H^i_{\et}(Y,\mathcal H^j)=0\quad(i>0,\ j=1,2).
 \end{gathered}
\end{equation}
Apply \eqref{eq:realizable-stabilizer-sequence}, using
\eqref{eq:obs-explicit-dyadic-incidence-bound},
\eqref{eq:obs-mechanical-dyadic-degree-two-vanishing} and
\eqref{eq:obs-mechanical-dyadic-supported-sheaf-bounds}. We obtain
\begin{equation}\label{eq:obs-explicit-dyadic-equivariant-bounds}
 h_A^1(T)\le(g-1)/2+b,\qquad h_A^2(T)\le b.
\end{equation}
Equations~\eqref{eq:obs-explicit-dyadic-equivariant-bounds} and
\eqref{eq:realizable-trait-bound} give
\begin{equation}\label{eq:obs-mechanical-dyadic-trait-dimension}
 \dim_{\F}\mathcal B_2^2
 \le b+\frac{g-1}{2}+b
 =\frac{3g+1}{2}.
\end{equation}

Finally, the degree-one Hochschild--Serre sequence for
$C_{K_p}\to\Spec K_p$ is
\begin{equation}\label{eq:obs-explicit-local-hochschild-serre}
 0\longrightarrow H^1(K_p,\F)\longrightarrow Q_p
 \longrightarrow H^1_{\et}(C_{\overline K_p},\F)^{G_{K_p}},
\end{equation}
where $G_{K_p}=\Gal(\overline K_p/K_p)$.  Since $C$ has genus $g$, we have
$\dim_{\F}H^1_{\et}(C_{\overline K_p},\F)=2g$ and therefore
$\dim_{\F}H^1_{\et}(C_{\overline K_p},\F)^{G_{K_p}}\le2g$.  Hence
\begin{equation}\label{eq:obs-explicit-local-twist-dimension}
 \dim Q_p\le\dim H^1(K_p,\F)+2g.
\end{equation}
Kummer theory identifies $H^1(K_p,\F)$ with
$K_p^\times/K_p^{\times2}$. Let $R_p=\mathbb Z_{(p)}^h$.
For $u\in R_p^\times$, Hensel's lemma gives
\begin{equation}\label{eq:obs-mechanical-henselian-unit-square-criteria}
 u\in K_p^{\times2}\ \Longleftrightarrow\
 \begin{cases}
   \overline u\in\mathbb F_p^{\times2},&p\ne2,\\
   u\equiv1\pmod{8R_2},&p=2.
 \end{cases}
\end{equation}
For odd $p$, choose $u_p\in R_p^\times$ with
$\overline u_p\notin\mathbb F_p^{\times2}$.
Every $a\in K_p^\times$ has the form $p^nu$ with
$n\in\mathbb Z$ and $u\in R_p^\times$.
Equation~\eqref{eq:obs-mechanical-henselian-unit-square-criteria}
therefore gives
\begin{equation}\label{eq:obs-mechanical-local-kummer-bases}
 \begin{array}{lll}
 K_p^\times/K_p^{\times2}=\langle[p],[u_p]\rangle,
   &\dim_{\F}H^1(K_p,\F)=2,&p\ne2,\\
 K_2^\times/K_2^{\times2}=\langle[2],[-1],[5]\rangle,
   &\dim_{\F}H^1(K_2,\F)=3.&
 \end{array}
\end{equation}
For $p=2$, independence of $[-1]$ and $[5]$ follows from
$-1\equiv7\pmod8$ and $5\equiv5\pmod8$. The class $[2]$
is independent of $[-1]$ and $[5]$ because $v_2(2)=1$, whereas
$v_2(-1)=v_2(5)=0$. For odd $p$, independence of
$[p]$ and $[u_p]$ follows from the valuation and the residue class,
respectively. Equations~\eqref{eq:obs-explicit-local-twist-dimension}
and \eqref{eq:obs-mechanical-local-kummer-bases} give
\begin{equation}\label{eq:obs-mechanical-local-twist-bounds}
 \dim_{\F}Q_p\le2g+2\quad(p\ne2),\qquad
 \dim_{\F}Q_2\le2g+3.
\end{equation}
Since $\mathcal D_\Sigma$ is a quotient of
$\bigoplus_{p\in\Sigma}Q_p$, we obtain
\begin{equation}\label{eq:obs-explicit-difference-space-bound}
 \dim\mathcal D_\Sigma\le\sum_{p\in\Sigma}\dim Q_p
       \le(2g+2)t+(2g+3).
\end{equation}
\end{proof}

Every prime in $\Sigma$ is at most $2g+1$.  We now substitute
\eqref{eq:obs-explicit-local-dimension-bounds} into
\eqref{eq:bad-only-obstruction-budget} and bound
$t=|\Sigma|-1$ by the number of primes at most $2g+1$.

\begin{proposition}\label{prop:uniform-realizable-obstruction}
Let $\widetilde S\to S_0\to S_{-1}$ be an anchored diagram, let
$g=g((S_{-1})_{\Q})$, and let $t=|\Sigma|-1$.  Then
\begin{equation}\label{eq:obs-explicit-uniform-obstruction-bound}
 \dim_{\F}\mathcal O(\widetilde S,S_{-1})
 \le(326g+326)t+\frac{7g+7}{2}
 \le\frac{5000g^2}{\log g}.
\end{equation}
In particular the implicit constant in
$\dim_{\F}\mathcal O(\widetilde S,S_{-1})=O(g^2/\log g)$ is universal.
\end{proposition}

\begin{proof}
Substituting \eqref{eq:obs-explicit-local-dimension-bounds} into
\eqref{eq:bad-only-obstruction-budget} gives
\begin{equation}\label{eq:obs-explicit-summed-obstruction-bound}
 \begin{aligned}
 \dim_{\F}\mathcal O(\widetilde S,S_{-1})
 &\le(324g+324)t+\frac{3g+1}{2}
                    +(2g+2)t+(2g+3)\\
 &=(326g+326)t+\frac{7g+7}{2}.
 \end{aligned}
\end{equation}
It remains to give a numerical bound for $t$.  Write
$\vartheta(x)=\sum_{p\le x}\log p$ and
$\pi(x)=|\{p:p\text{ prime},\ p\le x\}|$.  Every prime $n<p\le2n$ divides
$\binom{2n}{n}$, so
$\vartheta(2n)-\vartheta(n)\le2n\log2$.
For $k\ge1$, summing the inequalities with
$n=2^j$, $0\le j\le k-1$, gives
\begin{equation}\label{eq:obs-mechanical-chebyshev-dyadic-sum}
 \vartheta(2^k)
 \le\sum_{j=0}^{k-1}2^{j+1}\log2
 =2(2^k-1)\log2.
\end{equation}
For $x\ge2$, choose the integer $k\ge1$ with
$2^{k-1}<x\le2^k$. Then
$\vartheta(x)\le\vartheta(2^k)\le4(\log2)x$.
There are at most $\sqrt x$ primes at most $\sqrt x$, and
$\log p>\tfrac12\log x$ for every prime $\sqrt x<p\le x$.
Consequently
\begin{equation}\label{eq:obs-explicit-prime-count-bound}
 \pi(x)\le\sqrt x+\frac{2\vartheta(x)}{\log x}
           \le\sqrt x+\frac{8(\log2)x}{\log x}.
\end{equation}
For $g\ge15$, one has $2g+1\le(31/15)g$ and $\log g\le\sqrt g$.
Substituting $x=2g+1$ into \eqref{eq:obs-explicit-prime-count-bound}
therefore gives
\begin{equation}\label{eq:obs-explicit-omitted-prime-count}
 t\le\pi(2g+1)
 \le\left(\sqrt{31/15}+\frac{248}{15}\log2\right)
                   \frac{g}{\log g}
 <\frac{14g}{\log g}.
\end{equation}
Also $326g+326<350g$ and $(7g+7)/2\le4g$.  By
\eqref{eq:obs-explicit-omitted-prime-count}, equation \eqref{eq:obs-explicit-summed-obstruction-bound} gives
\begin{equation}\label{eq:obs-mechanical-final-numerical-obstruction-bound}
 \dim_{\F}\mathcal O(\widetilde S,S_{-1})
 <\frac{4900g^2}{\log g}+4g
 \le\frac{4904g^2}{\log g}
 <\frac{5000g^2}{\log g},
\end{equation}
as desired.
\end{proof}

\section{Proofs of the main theorems}\label{sec:main-proofs}

We use the realizable extension classes of Section~\ref{sec:realizable}
to construct finite \'etale covers $S_n\to S_0$ of degree $2^n$.
For $C_n=(S_n)_{\Q}$, we prove
$\log\mathfrak N(C_n)\le A_m g(C_n)$, where $A_m$ is independent of $n$.
To obtain a tower over a curve $C/\Q$ of any sufficiently large genus,
we pull back the
\'etale tower over $C_{-1}$ along a ramified double cover
$C\to C_{-1}$.

\subsection{Proof of Theorem~\ref{thm:main-surface}}
\label{sec:proof-surface}

For a finite $2$-group $H$, write $d(H)$ for its minimal number of
generators.  The cohomological Golod--Shafarevich inequality gives
\begin{equation}\label{eq:finite-GS}
 \dim_{\F}H^2(H,\F)>\frac{d(H)^2}{4}
\end{equation}
for every nontrivial $H$; see \cite[Theorem~3.4(a)]{Ershov} and
the expression for the relation rank in terms of cohomology in
\cite[Corollary~3.9.5]{NSW}.

Let $1\to\F\to E\xrightarrow{\rho}H\to1$ be a nonsplit central
extension of finite $2$-groups. We claim that
\begin{equation}\label{eq:final-subgroup-surjection}
 J\le E,\quad \rho(J)=H\quad\Longrightarrow\quad J=E.
\end{equation}
Indeed, if $J$ contains $\ker\rho$, then $J=E$. Otherwise
$J\cap\ker\rho=1$, since $\ker\rho$ has order two, and the
isomorphism $\rho|_J:J\to H$ gives a splitting, a contradiction.
Choose generators $h_1,\ldots,h_{d(H)}$ of $H$ and elements
$e_1,\ldots,e_{d(H)}\in E$ with $\rho(e_i)=h_i$. For
$J=\langle e_1,\ldots,e_{d(H)}\rangle$, we have $\rho(J)=H$;
\eqref{eq:final-subgroup-surjection} therefore gives $J=E$.
Conversely, $\rho$ sends any generating set of $E$ to a generating
set of $H$. Consequently,
\begin{equation}\label{eq:final-generator-equality}
 d(E)\le d(H)\le d(E),\qquad d(E)=d(H).
\end{equation}

For the surfaces of Section~\ref{sec:starting-surface}, we compare
the bound of Proposition~\ref{prop:uniform-realizable-obstruction}
with $g^2/4$. The inequality $5000g^2/\log g<g^2/4$ will give a nonzero class in
$\mathscr E(S_n,S_{-1})\subseteq H^2(G_n,\F)$ and hence a finite
\'etale double cover $S_{n+1}\to S_n$ such that
$S_{n+1}\to S_{-1}$ is Galois. Restriction to $S_n$ identifies
$G_{n+1}=\Gal(S_{n+1}/S_{-1})$ with a nonsplit central extension
of $G_n$, so $d(G_{n+1})=d(G_n)=g$.

\Needspace{6\baselineskip}
\begin{proposition}\label{prop:constructed-tower}
Let $m\ge32768$, put $g=2^m-1$, and let $S_0\to S_{-1}$ and
$s_0:\Spec\Z\to S_0$ be the surfaces and section constructed in
Propositions~\ref{prop:start-integral-galois} and~\ref{prop:start-integral-section}.
There is a tower
\begin{equation}\label{eq:final-explicit-surface-tower}
 \cdots\longrightarrow S_{n+1}\longrightarrow S_n
       \longrightarrow\cdots\longrightarrow S_0
\end{equation}
with the following properties.
\begin{enumerate}[label=(\roman*)]
\item Every $S_n$ is an integral regular projective flat arithmetic
surface with smooth geometrically integral generic fibre.
\item The map $S_n\to S_0$ is a connected finite \'etale Galois
cover of degree $2^n$.  Each transition $S_{n+1}\to S_n$ has degree
two.
\item The composite $S_n\to S_{-1}$ is Galois. Put
$G_n=\Gal(S_n/S_{-1})$. Then
\begin{equation}\label{eq:final-explicit-group-invariants}
 G_0\simeq\F^g,\qquad |G_n|=2^{g+n},\qquad d(G_n)=g.
\end{equation}
For every $n\ge0$, restriction gives a nonsplit central extension
\begin{equation}\label{eq:final-explicit-central-extension}
 1\longrightarrow\F\longrightarrow G_{n+1}
       \longrightarrow G_n\longrightarrow1.
\end{equation}
\item The section $s_0$ splits completely at every level:
\begin{equation}\label{eq:final-explicit-section-splitting}
 S_n\times_{S_0,s_0}\Spec\Z
       \simeq\coprod_{j=1}^{2^n}\Spec\Z.
\end{equation}
\end{enumerate}
\end{proposition}

\begin{proof}
Since $m\ge32768$ and $g=2^m-1\ge2^{m-1}$, we have
\begin{equation}\label{eq:final-explicit-numerical-threshold}
 \log g\ge(m-1)\log2\ge32767\log2>20000,
 \qquad
 \frac{5000g^2}{\log g}<\frac{g^2}{4}.
\end{equation}
For $n=0$, we have $G_0=\Gal(S_0/S_{-1})\simeq\F^g$, by
Proposition~\ref{prop:start-integral-galois}.  Suppose $S_n$ has
been constructed, with $S_n\to S_0$ connected finite \'etale and
$S_n\to S_{-1}$ Galois with a finite $2$-group $G_n$ satisfying
$d(G_n)=g$.  By Proposition~\ref{prop:starting-anchor}, the diagram
$S_n\to S_0\to S_{-1}$ satisfies the conditions in
\eqref{eq:anchored-diagram}. The obstruction estimate of
Proposition~\ref{prop:uniform-realizable-obstruction}, together with
\eqref{eq:final-explicit-numerical-threshold} and \eqref{eq:finite-GS}, gives
\begin{equation}\label{eq:final-explicit-positive-realizable-dimension}
 \dim_{\F}\mathcal O(S_n,S_{-1})
 \le\frac{5000g^2}{\log g}
 <\frac{g^2}{4}
 <\dim_{\F}H^2(G_n,\F).
\end{equation}
Since
$\mathcal O(S_n,S_{-1})=H^2(G_n,\F)/\mathscr E(S_n,S_{-1})$,
\eqref{eq:final-explicit-positive-realizable-dimension} shows that
$\mathscr E(S_n,S_{-1})$ contains a nonzero class $\theta_n$.

Choose a realization $S_{n+1}\to S_n$ of $\theta_n$.
By the definition of realizability, $S_{n+1}\to S_n$ is connected and finite \'etale,
$S_{n+1}\to S_{-1}$ is Galois, and
$G_{n+1}=\Gal(S_{n+1}/S_{-1})$ embeds in $E_{\theta_n}$
so that restriction to $S_n$ is the composite
$G_{n+1}\hookrightarrow E_{\theta_n}\to G_n$.
The homomorphism $G_{n+1}\to G_n$ is surjective. Since
$\theta_n\ne0$, the extension $E_{\theta_n}\to G_n$ is nonsplit.
Apply \eqref{eq:final-subgroup-surjection} with
$E=E_{\theta_n}$ and $J=G_{n+1}$ to obtain $G_{n+1}=E_{\theta_n}$.
This proves \eqref{eq:final-explicit-central-extension}. In particular,
\begin{equation}\label{eq:final-induction-degree}
 |\ker(G_{n+1}\to G_n)|=2,\qquad
 [\Q(S_{n+1}):\Q(S_n)]=2,\qquad
 |G_{n+1}|=2|G_n|.
\end{equation}
Equation~\eqref{eq:final-generator-equality} gives
$d(G_{n+1})=d(G_n)=g$. This completes the induction and proves
\eqref{eq:final-explicit-group-invariants}.

Each composite $S_n\to S_0$ is finite \'etale of degree $2^n$.
The field extension $\Q(S_n)/\Q(S_0)$ is Galois because
$\Q(S_n)/\Q(S_{-1})$ is Galois and
$\Q(S_{-1})\subseteq\Q(S_0)\subseteq\Q(S_n)$. Moreover,
\begin{equation}\label{eq:final-galois-kernel}
 \Gal(\Q(S_n)/\Q(S_0))=\ker(G_n\to G_0).
\end{equation}
The finite \'etale map $S_n\to S_0$ makes $S_n$ regular and hence
normal, so $S_n$ is the normalization of $S_0$ in $\Q(S_n)$.
Every automorphism of $\Q(S_n)$ fixing
$\Q(S_0)$ therefore extends to an automorphism of $S_n$ over
$S_0$, proving that $S_n\to S_0$ is Galois. Finally,
Proposition~\ref{prop:realizing-surface-properties}, applied to
$S_n\to S_0$, proves (i) and the splitting formula
\eqref{eq:final-explicit-section-splitting} in (iv).
\end{proof}

\begin{proof}[Proof of Theorem~\ref{thm:main-surface}]
Take $S=S_0$ and $s=s_0$ in Proposition~\ref{prop:constructed-tower}.
Assertions (i), (ii) and (iv) of Proposition~\ref{prop:constructed-tower}
are assertions (ii), (i) and (iii), respectively, of
Theorem~\ref{thm:main-surface}.
The connected Galois covers $S_n\to S$ give quotients
of $\piet(S)$ of orders $2^n$, so $\piet(S)$ is infinite.

Put $H=\im(\piet(S_{\overline\Q})\to\piet(S))$.
Let $Q$ be a finite quotient of $\piet(S)$, and let $T\to S$
be the corresponding connected finite \'etale Galois cover with
group $Q$. By Proposition~\ref{prop:realizing-surface-properties},
$T_{\overline\Q}$ is connected. Hence the image of
$\piet(S_{\overline\Q})$ in $Q$ acts transitively on the fibre
of $T\to S$ over a geometric point of $S_{\overline\Q}$.
Choose a geometric point $\bar x$ of $S_{\overline\Q}$ and identify
the fibre $T_{\bar x}$ with $Q$ by choosing one point of $T_{\bar x}$.
The image of $H$ in $Q$ then acts on $Q$ by left multiplication.
Transitivity implies $\im(H\to Q)=Q$. Since $Q$ was arbitrary,
$\overline H=\piet(S)$. The subgroup $H$ is compact, as the
continuous image of the profinite group $\piet(S_{\overline\Q})$,
and is therefore closed in $\piet(S)$. Hence
$\pigeom(S)=H=\overline H=\piet(S)$.
\end{proof}

\subsection{Proof of Theorem \ref{thm:main-curves}}\label{sec:proof-curves}

\begin{proof}
Fix $m\ge32768$, put $b=2^m-1$, and take the tower of
Proposition~\ref{prop:constructed-tower}. For each $n\ge0$, put
$C_n=(S_n)_{\Q}$. Then $C_n$ is a smooth projective geometrically
integral curve over $\Q$.
The maps $C_n\to C_{-1}$ are Galois with groups $G_n$ of order
$2^{b+n}$ by \eqref{eq:final-explicit-group-invariants}.
The maps $C_n\to C_0$ are finite \'etale of degree
$2^n$, by base change from $S_n\to S_0$ to $\Spec\Q$.
The map $C_0\to C_{-1}$ is \'etale by
Proposition~\ref{prop:start-etale}, so the composites
$C_n\to C_{-1}$ are \'etale. Each transition
$C_{n+1}\to C_n$ has degree two.
Applying Riemann--Hurwitz to $C_n\to C_{-1}$, whose target has
genus $b$, gives
\begin{equation}\label{eq:final-explicit-generic-genus}
 g(C_n)-1=2^{b+n}(b-1).
\end{equation}

Let $s_{-1}:\Spec\Z\to S_{-1}$ be the section extending
$P_+\in C_{-1}(\Q)$, and put $A=\Gal(S_0/S_{-1})$. For $a\in A$,
put $s_a=a\circ s_0:\Spec\Z\to S_0$.
Proposition~\ref{prop:start-integral-section} gives
\begin{equation}\label{eq:final-base-translated-sections}
 S_0\times_{S_{-1},s_{-1}}\Spec\Z
 \simeq\coprod_{a\in A}\Spec\Z,
\end{equation}
where the map from the copy indexed by $a$ to $S_0$ is $s_a$.
Apply Proposition~\ref{prop:realizing-surface-properties} to each $s_a$
and the degree-$2^n$ finite \'etale cover $S_n\to S_0$. For each
$a\in A$, the fibre product $S_n\times_{S_0,s_a}\Spec\Z$ is
a disjoint union of $2^n$ copies of $\Spec\Z$.
Taking the disjoint union over $a\in A$ gives
\begin{equation}\label{eq:final-explicit-translated-section-splitting}
 S_n\times_{S_{-1},s_{-1}}\Spec\Z
       \simeq\coprod_{j=1}^{2^{b+n}}\Spec\Z.
\end{equation}
Base change of \eqref{eq:final-explicit-translated-section-splitting}
from $\Spec\Z$ to $\Spec\Q$ identifies
$C_n\times_{C_{-1},P_+}\Spec\Q$ with a disjoint union of
$2^{b+n}$ copies of $\Spec\Q$. Thus $P_+$ splits completely
under $C_n\to C_{-1}$ and $\#C_n(\Q)\ge2^{b+n}$.

Choose subgroups $H_0,\ldots,H_b$ of $A$ satisfying
\begin{equation}\label{eq:final-explicit-subgroup-flag}
 A=H_0\supset H_1\supset\cdots\supset H_b=1,
 \qquad [A:H_i]=2^i,
\end{equation}
in $A=\Gal(C_0/C_{-1})\simeq\F^b$, and set $E_i=C_0/H_i$
for $0\le i\le b$. Each $H_i$ is normal in the abelian group
$A$, so $E_i\to C_{-1}$ is Galois with group $A/H_i$ and
degree $2^i$ by \eqref{eq:final-explicit-subgroup-flag}.
The morphism $E_i\to C_{-1}$ is \'etale because $E_i=C_0/H_i$ is an intermediate cover
of the finite \'etale Galois cover $C_0\to C_{-1}$.
Set $E_{b+n}=C_n$ for $n\ge0$; the two definitions agree at
$E_b=C_0$. The curves $E_i$ form a tower of finite \'etale Galois
covers $E_i\to C_{-1}$ of degree $2^i$, with
$E_0=C_{-1}$, with transitions of degree two. For $0\le i\le b$, take the quotient by $H_i$ of the
$A$-torsor $C_0\times_{C_{-1},P_+}\Spec\Q\simeq A\times\Spec\Q$.
For $i\ge b$, use \eqref{eq:final-explicit-translated-section-splitting}.
We obtain, for every $i\ge0$,
\begin{equation}\label{eq:final-all-intermediate-splitting}
 E_i\times_{C_{-1},P_+}\Spec\Q
 \simeq\coprod_{j=1}^{2^i}\Spec\Q.
\end{equation}

For every $g\ge3b$, Proposition~\ref{prop:genus-spreading} gives a
ramified double cover $f:C\to C_{-1}$ with $g(C)=g$ and a rational
point $P\in C(\Q)$ above $P_+$. Set
\begin{equation}\label{eq:final-explicit-pullback}
 D_i=C\times_{C_{-1}}E_i,\qquad
 \pi_i:D_i\longrightarrow C,\qquad (c,e)\longmapsto c.
\end{equation}
For each $i\ge0$, the morphism $D_{i+1}\to D_i$ is
$\operatorname{id}_C\times(E_{i+1}\to E_i)$. Also
$D_0=C\times_{C_{-1}}C_{-1}=C$.

Put $\Gamma_i=\Gal(E_i/C_{-1})$. Define the action of $\Gamma_i$ on $D_i$ by
\begin{equation}\label{eq:final-explicit-pullback-action}
 \gamma(c,e)=(c,\gamma e)\qquad(\gamma\in\Gamma_i).
\end{equation}
Thus $\pi_i$ is the base change of the finite \'etale
$\Gamma_i$-torsor $E_i\to C_{-1}$ and has degree $2^i$.
Proposition~\ref{prop:genus-spreading} shows that $D_i$ is geometrically
connected: the ramified quadratic extension
$\overline\Q(C)/\overline\Q(C_{-1})$ is linearly disjoint from
$\overline\Q(E_i)/\overline\Q(C_{-1})$, because
$(E_i)_{\overline\Q}\to(C_{-1})_{\overline\Q}$ is \'etale.
Hence $D_i$ is a smooth projective geometrically connected curve over $\Q$.

Riemann--Hurwitz for $\pi_i$ gives
\begin{equation}\label{eq:final-explicit-pullback-genus}
 g(D_i)-1=2^i(g-1).
\end{equation}
Since $f(P)=P_+$, the fibre of $\pi_i$ at $P$ is
\begin{equation}\label{eq:final-explicit-pullback-splitting}
 D_i\times_{C,P}\Spec\Q
 \simeq E_i\times_{C_{-1},P_+}\Spec\Q
 \simeq\coprod_{j=1}^{2^i}\Spec\Q.
\end{equation}
Equation~\eqref{eq:final-explicit-pullback-splitting} shows that
$\pi_i^{-1}(P)$ consists of $2^i$ distinct $\Q$-rational points; hence
$P$ splits completely under $\pi_i:D_i\to C$ and $\#D_i(\Q)\ge2^i$.
Because $P$ splits completely, the section
$s_P:\Gal(\overline\Q/\Q)\to\piet(C,\bar P)$ has trivial
image in $\Gal(D_i/C)$. The quotient map
$\piet(C,\bar P)\to\Gal(D_i/C)$ therefore factors through
$\pi_1^{\Q\text{-}\mathrm{geom}}(C,P)$. Since $|\Gal(D_i/C)|=2^i$ for every $i\ge0$,
$\pi_1^{\Q\text{-}\mathrm{geom}}(C,P)$ is infinite. This proves the theorem,
for example with $g_0=3(2^{32768}-1)$.
\end{proof}

\subsection{Proof of Theorem \ref{thm:main-conductor}}\label{sec:proof-conductor}

\begin{proof}[Proof of Theorem~\ref{thm:main-conductor}]
Choose $m\ge32768$, put $g=2^m-1$, and take the tower of Proposition~\ref{prop:constructed-tower}. Write $C_n=(S_n)_{\Q}$.
Every $S_n\to S_0$ is finite \'etale.  The starting surface $S_0$ is
smooth over $U=\Spec\Z[1/\prod_{p\le2g+1}p]$ by
Proposition~\ref{prop:start-good-fibres}.
Proposition~\ref{prop:conductor-uniform-integral-covers}, applied with
$\Sigma=\{p\text{ prime}:p\le2g+1\}$, therefore gives
\begin{equation}\label{eq:final-explicit-conductor-bound}
 \log\mathfrak N(C_n)\le A_m g(C_n)\qquad(n\ge0),
 \qquad A_m=2\sum_{p\in\Sigma}(1+b_p)\log p,
\end{equation}
where, for each $p\in\Sigma$, we fix a finite Galois extension
$E_p/\Q_p$ over which $C_0$ has semistable reduction and put
$b_p=b(E_p/\Q_p)$ as in Proposition~\ref{prop:conductor-local-bound}.
The extensions $E_p/\Q_p$ are chosen using $C_0$ and are independent
of $n$. Thus $A_m=2\sum_{p\in\Sigma}(1+b_p)\log p$ is independent of $n$.
The Riemann--Hurwitz formula \eqref{eq:final-explicit-generic-genus},
with $b=g$, gives
\begin{equation}\label{eq:final-conductor-genus-limit}
 g(C_n)=1+2^{g+n}(g-1)\longrightarrow\infty
 \quad\text{as }n\to\infty.
\end{equation}
Base change of \eqref{eq:final-explicit-section-splitting}
to $\Spec\Q$ shows that
$(s_0)_{\Q}\in C_0(\Q)$ splits completely under every map
$C_n\to C_0$.
\end{proof}

\end{document}